\documentclass[11pt, letterpaper]{article}

\pdfoutput=1

\usepackage{amsmath}
\usepackage{amsfonts}
\usepackage{amssymb} 
\usepackage{amsthm}
\usepackage{mathrsfs} 
\usepackage{mathtools} 
\usepackage[nottoc,numbib]{tocbibind} 
\usepackage[usenames,dvipsnames,svgnames,table]{xcolor} 

\usepackage{tikz}
\usetikzlibrary{shapes,matrix,arrows,chains} 

\usepackage{xr} 
\usepackage{cite} 
\usepackage{enumerate} 
\usepackage[colorlinks=true, allcolors=MidnightBlue, bookmarksdepth=3, pdftitle={5-dimensional geometries III: the fibered geometries}, pdfauthor={Andrew Geng}]{hyperref}
\usepackage{url}
\usepackage[total={6.5in,9in}]{geometry}
\usepackage[T1]{fontenc} 
\usepackage[utf8]{inputenc} 

\title{5-dimensional geometries III: the fibered geometries}
\author{Andrew Geng}

\theoremstyle{plain}
\newtheorem{thm}{Theorem}[section]
\newtheorem{prop}[thm]{Proposition}

\newtheorem{lemma}[thm]{Lemma}
\newtheorem{cor}[thm]{Corollary}
\theoremstyle{definition}
\newtheorem{defn}[thm]{Definition}

\theoremstyle{remark}
\newtheorem{rmk}[thm]{Remark}
\newtheorem{eg}[thm]{Example}
\makeatletter
\let\c@figure\c@thm
\let\c@table\c@thm
\makeatother
\numberwithin{figure}{section}
\numberwithin{table}{section}

\newcommand{\keyword}{\emph}
\newcommand{\op}{\operatorname}
\newcommand{\ol}{\overline}
\newcommand{\R}{\mathbb{R}}
\newcommand{\C}{\mathbb{C}}
\newcommand{\Q}{\mathbb{Q}}

\newcommand{\Z}{\mathbb{Z}}
\newcommand{\Hyp}{\mathbb{H}}
\newcommand{\Euc}{\mathbb{E}}
\newcommand{\Heis}{\mathrm{Heis}}
\newcommand{\Sol}{\mathrm{Sol}}

\newcommand{\SLcover}{\widetilde{\mathrm{SL}_2}}

\newcommand{\lie}{\mathfrak}
\newcommand{\tanalg}{T_{\mathbf{1}}} 

\newcommand{\tanisom}{\operatorname{\lie{isom}}}
\newcommand{\idisom}{\operatorname{Isom}_0}
\newcommand{\isomplus}{\idisom}
\newcommand{\idcompo}[1]{ {{#1}}^\circ{} }

\newcommand{\foliation}{\mathcal} 
\newcommand{\vecspan}[1]{ \operatorname{span}\left( { { #1 } } \right) } 
\newcommand{\semisum}{\mathrlap{+}{\supset}} 

\newcommand{\SemiR}[2]{ { \underset{ {#2} }{ { {#1} } \rtimes \R} } }

\begin{document}

\maketitle

\begin{abstract}
    We classify the $5$-dimensional homogeneous geometries in the sense of Thurston.
    The present paper (part 3 of 3) classifies those in which the linear isotropy
    representation is nontrivial but reducible. Most of the resulting geometries
    are products.
    Some interesting examples include a countably infinite family
    $L(a;1) \times_{S^1} L(b;1)$ of inequivalent geometries
    diffeomorphic to $S^3 \times S^2$; an uncountable
    family $\SLcover \times_\alpha S^3$ in which only a countable
    subfamily admits compact quotients; and the non-maximal geometry
    $\op{SO}(4)/\op{SO}(2)$ realized by two distinct maximal geometries.
\end{abstract}

\setcounter{tocdepth}{2}
\tableofcontents

\section{Introduction}

Thurston's geometries are a family of eight homogeneous spaces
that form the building blocks of $3$-manifolds
in Thurston's Geometrization Conjecture. Building on
Thurston's classification \cite[Thm.~3.8.4]{thurstonbook} in dimension $3$
and Filipkiewicz's classification \cite{filipk} in dimension $4$,
this paper is part of a series carrying out the classification in dimension $5$.

In the sense of Thurston,
a \keyword{geometry} is a simply-connected Riemannian homogeneous
space $M = G/G_p$, with the additional conditions that $M$ has a finite-volume
quotient---``model''---and $G$ is as large as possible---``maximality''
(details are in Defn.~\ref{defn:geometries3}).
Part I \cite{geng1} outlines the division of the problem into cases,
following the strategy of Thurston and Filipkiewicz,
using the action of point stabilizers $G_p$ on tangent spaces $T_p M$
(the ``linear isotropy representation'').
Part II \cite{geng2} performs the classification for the case
when this representation is trivial or irreducible,
by leveraging other classification results.

The present paper finishes the classification by working out the case
when $T_p M$ is nontrivial and reducible. The decomposition
of $T_p M$ is used to construct a $G$-invariant fiber bundle structure
on $M$---hence the name ``fibering geometries''---and
the classification of these bundles provides a way to access
a classification of geometries.
To cope with the increased richness in fiber bundle structures
compared to what is possible in lower dimensions,
the tools required include conformal geometry, Galois theory,
and Lie algebra cohomology in addition to everything used in
previous classifications.
In particular, extension problems feature much more noticeably
than for the analogous cases in lower dimensions.
The main result is the following.
\begin{thm}[\textbf{Classification of $5$-dimensional maximal model
        geometries with nontrivial, reducible isotropy}]
    \label{thm:main}
    Let $M = G/G_p$ be a $5$-dimensional maximal model geometry, and
    let $V$ be an irreducible subrepresentation
    of $G_p \curvearrowright T_p M$ of maximal dimension.
    \begin{enumerate}
        \item[(i)] If $\dim V = 4$ (Section \ref{chap:fiber4}), then $M$ is one of
            the spaces
            \begin{align*}
                S^4 &\times \Euc & \Hyp^4 &\times \Euc &
                \C P^2 &\times \Euc & \C\Hyp^2 &\times \Euc &
                &{}\widetilde{\op{U}(2,1)/\op{U}(2)}
                &{}\Heis_5 .
            \end{align*}
        \item[(ii)] If $\dim V = 3$ (Section \ref{chap:fiber3}), then $M$ is a product
            of $2$-dimensional and $3$-dimensional constant-curvature geometries.
        \item[(iii)] If $\dim V = 2$ (Section \ref{chap:fiber2}), then $M$ is either
            a product of lower-dimensional geometries or one of the following.
            \begin{enumerate}
                \item The unit tangent bundles,
                    \begin{align*}
                        T^1 \Hyp^3 = \op{PSL}(2,\C)/\op{PSO}(2)  & &
                        T^1 \Euc^{1,2} = \R^3 \rtimes \idcompo{\op{SO}(1,2)} / \op{SO}(2);
                    \end{align*}
                \item The associated bundles
                    (see Defn.~\ref{defn:fiber2_slope} and Table
                    \ref{table:fiber2_associated_bundles}),
                    \begin{align*}
                        \Heis_3 &\times_\R S^3  &
                        \SLcover \times_\alpha S^3 &,
                            \quad 0 < \alpha < \infty \\
                        \Heis_3 &\times_\R \SLcover  &
                        \SLcover \times_\alpha \SLcover &,
                            \quad 0 < \alpha \leq 1 \\
                        && L(a;1) \times_{S^1} L(b;1) &,
                            \quad 0 < a \leq b \text{ coprime in } \Z ;
                    \end{align*}
                \item The line bundles over $\mathbb{F}^4$,
                    \begin{align*}
                        \R^2 \rtimes \SLcover
                            &\cong (\R^2 \rtimes \SLcover) \rtimes \op{SO}(2) / \op{SO}(2) \\
                        \mathbb{F}^5_a
                            &= \Heis_3 \rtimes \SLcover / \{ atz, \gamma(t) \}_{t \in \R} ,
                                \quad a = 0 \text{ or } 1 ;
                    \end{align*}
                \item The indecomposable non-nilpotent solvable Lie groups
                    $\SemiR{\R^4}{\text{polynomials}}$,
                    specified by the list of characteristic polynomials of the Jordan
                    blocks of a matrix $A$ where $t \in \R$ acts on $\R^4$ by $e^{tA}$,
                    \begin{align*}
                        A^{-1,-1}_{5,9}
                            &= \SemiR{\R^4}{(x-1)^2,\, x+1,\, x+1} \\
                        A^{1,-1,-1}_{5,7}
                            &= \SemiR{\R^4}{x-1,\, x-1,\, x+1,\, x+1}  \\
                        A^{1,-1-a,-1+a}_{5,7}
                            &= \SemiR{\R^4}{x-1,\, x-1,\, x-a+1,\, x+a+1}
                            \text{ where } a > 0,\,a \neq 1,\,a \neq 2, \\
                            &\qquad \text{ and }
                                \op{det} (\lambda - e^{tA}) \in \Z[\lambda]
                                \text{ for some } t > 0 ;
                    \end{align*}
                \item and the indecomposable nilpotent Lie groups,
                    \begin{align*}
                        A_{5,1} &= \SemiR{\R^4}{x^2, x^2}  &
                        A_{5,3} &= \SemiR{(\R \times \Heis_3)}{x_3 \to x_2 \to y} .
                    \end{align*}
            \end{enumerate}
    \end{enumerate}
    Moreover, all of the explicitly named spaces above
    are indeed maximal model geometries; and each
    product geometry is a model geometry,
    and maximal if at most one factor is Euclidean.
\end{thm}

The solvable Lie groups $M$ are given with names from \cite[Table II]{patera}
for Mubarakzyanov's classification of $5$-dimensional
solvable real Lie algebras \cite{muba_solvable5}.
Their isometry groups are $M \rtimes G_p$,
where the point stabilizer $G_p$ is a maximal compact subgroup of $\op{Aut} M$.
Table \ref{table:geoms_by_isotropy} lists these groups $G_p$;
more explicit descriptions for the solvable Lie groups can be found in
Prop.~\ref{prop:fiber2_nilmanifolds} Step 4
and Prop.~\ref{prop:fiber2_essentials}(iii) Step 6.
The lower dimensional geometries
have their usual isometry groups; e.g.\ $\C P^2$ and $\C\Hyp^2$
are as specified in \cite[Thm.~3.1.1]{filipk}.
A full list of point stabilizers in dimension $4$
can be found in \cite[Table 1]{wall86}.

\begin{table}[h!]
    \caption{Non-product fibering geometries by isotropy group $G_p$
        (see also Fig.~\ref{fig:isotropy_poset})}
    \label{table:geoms_by_isotropy}
    \begin{center}\begin{tabular}{c@{\hskip 24pt}l}
        Isotropy &
        Geometries \\
        \hline
        \rule[-9pt]{0pt}{27pt}
        $\op{U}(2)$ &
        $\Heis_5$ and $\widetilde{\op{U}(2,1)/\op{U}(2)}$ \\
        \rule[-15pt]{0pt}{0pt}
        $\op{SO}(2) \times \op{SO}(2)$ &
        $\SemiR{\R^4}{x-1,\,x-1,\,x+1,\,x+1}$
        and the associated bundles (Thm.~\ref{thm:main}(iii)(b)) \\
        \rule[-9pt]{0pt}{0pt}
        $\op{SO}(2)$ &
        The remaining solvable groups from Thm.~\ref{thm:main}(iii)(d) \\
        \rule[-12pt]{0pt}{0pt}
        $S^1_{1/2}$ &
        All line bundles over $\mathbb{F}^4$ (Thm.~\ref{thm:main}(iii)(c)) \\
        \rule[-6pt]{0pt}{0pt}
        $S^1_1$ &
        The two unit tangent bundles (Thm.~\ref{thm:main}(iii)(a)) \\ &
        and the nilpotent Lie groups from Thm.~\ref{thm:main}(iii)(e)
        \\
    \end{tabular}\end{center}
\end{table}

In the spirit of \cite[Cor.~p.~624]{mostow1950}, 
\cite{gorbatsevich1977}, \cite{ishihara1955},
and \cite[Thm.~1.0.3]{ottenburger2009}, one can also give a
classification up to diffeomorphism (Table \ref{table:geoms_by_diffeo}).
Most of the diffeomorphism types are readily guessed
and can be verified by a theorem of Mostow \cite[Thm.~A]{mostow_covariant_1962}.
The exception is the family of associated bundles $L(a;1) \times_{S^1} L(b;1)$;
Ottenburger names them $N^{ab1}$ in \cite[\S{3.1}]{ottenburger2009} and shows that
they are all diffeomorphic to $S^3 \times S^2$
in \cite[Cor.~3.3.2]{ottenburger2009}.

\begin{table}[h!]
    \caption{Non-contractible non-product fibering geometries by diffeomorphism type}
    \label{table:geoms_by_diffeo}
    \begin{center}\begin{tabular}{cl}
        Type & Geometries \\
        \hline
        \rule[-6pt]{0pt}{21pt}
        $S^2 \times \R^3$ &
        $T^1 \Hyp^3$ \\
        \rule[-6pt]{0pt}{0pt}
        $S^3 \times \R^2$ &
        $\Heis_3 \times_\R S^3$ and
        $S^3 \times_\alpha \SLcover$ \\
        \rule[-6pt]{0pt}{0pt}
        $S^3 \times S^2$  &
        $L(a;1) \times_{S^1} L(b;1)$  \\
    \end{tabular}\end{center}
\end{table}

\paragraph{Roadmap.}
Section \ref{chap:background} lists basic definitions and any
external results that need to be used frequently.
Then Section \ref{chap:fiber_bundles} uses the decomposition of the linear
isotropy representation to establish the existence of
invariant fiber bundle structures on geometries
(Prop.~\ref{prop:fibering_description}),
introducing related notations (such as names for invariant distributions) along the way.
The remaining sections each deal with base spaces of a single dimension:
\begin{itemize}
    \item Section \ref{chap:fiber4} handles the case where the fiber bundle
    has a $4$-dimensional base with irreducible isotropy.
    The geometries are classified by curvature, using a strategy
    closely following that of Thurston
    for $3$-dimensional geometries over $2$-dimensional bases
    in \cite[Thm.~3.8.4(b)]{thurstonbook}.
    \item Section \ref{chap:fiber3} handles $3$-dimensional isotropy irreducible bases.
    This case produces one non-product geometry that is shown not to be a model geometry
    by Galois theory (Section \ref{sec:fiber3_solconf}).
    \item Section \ref{chap:fiber2} handles $2$-dimensional bases.
    This case produces the ``associated bundle geometries'',
    a source of examples such as
    an uncountable family of geometries (Prop.~\ref{prop:fiber2_associated_lattice})
    and a geometry whose maximal realization is not unique
    (Rmk.~\ref{rmk:nonunique_maximality}).
\end{itemize}

\section{Background}
\label{chap:background}

\subsection{Geometries, products, and isotropy}

Recall the definition of a geometry, following Thurston and Filipkiewicz in
\cite[Defn.\ 3.8.1]{thurstonbook} and \cite[\S 1.1]{filipk}.
This is given in terms of homogeneous spaces,
since the upcoming classification will rely heavily on them;
\cite[Prop.~\ref{ii:prop:geometries3}]{geng2} in Part II outlines
the equivalence to earlier definitions.
\begin{defn}[\textbf{Geometries}] \label{defn:geometries3} ~
    \begin{enumerate}[(i)]
        \item A \keyword{geometry} is a connected, simply-connected homogeneous
            space $M = G/G_p$ where $G$ is a connected Lie group acting faithfully
            with compact point stabilizers $G_p$.
        \item $M$ is a \keyword{model geometry} if there is some lattice
            $\Gamma \subset G$ that acts freely on $M$.
            Then the manifold $\Gamma \backslash G / G_p$
            is said to be \keyword{modeled on} $M$.
        \item $M$ is \keyword{maximal} if it is not $G$-equivariantly
            diffeomorphic to any other geometry $G'/G'_p$ with $G \subsetneq G'$.
            Any such $G'/G'_p$ is said to \keyword{subsume} $G/G_p$.
    \end{enumerate}
\end{defn}

Many properties of this definition---including the existence of invariant
Riemannian metrics and the correspondence between quotients
$\Gamma \backslash G / G_p$ and complete, finite-volume manifolds
locally isometric to $M$---are taken for granted here
but stated more explicitly in Part II \cite[\S{2}]{geng2}.

Discussion of product geometries was omitted from Part II---but
among the fibering geometries are many (28 and one countable family).
Some shortcuts are possible with their classification, such as the following.
\begin{prop}[\textbf{Products are models}]
    \label{prop:products_are_models}
    If $M = G/G_p$ and $M' G'/G'_q$ are model geometries, then
    their product $M \times M' = (G \times G')/(G_p \times G'_q)$
    is a model geometry.
\end{prop}
\begin{proof}
    If $\Gamma \backslash G / G_p$ is modeled on $M$
    and $\Gamma' \backslash G' / G'_q$ is modeled on $M'$,
    then $(\Gamma \times \Gamma') \backslash (G \times G') / (G_p \times G'_q)$
    is modeled on $M \times M'$.
\end{proof}
Maximality is in general a more difficult question;
we prove in Prop.~\ref{prop:products_are_maximal}
that the product of two maximal geometries is maximal,
but under the assumption that
at most one factor admits nonzero invariant vector fields
and at most one factor is itself a product with a Euclidean factor.
This happens to be enough for our usage in dimension $5$
(Prop.~\ref{prop:fiber2_explicit_products}).

There is, however, a shortcut to prove maximality
for geometries realized by solvable Lie groups,
given by the following rephrasing of a theorem of Gordon and Wilson.
\begin{lemma}[\textbf{Maximality for solvable Lie groups}]
    \label{lemma:solvable_maximality}
    Let $M$ be a simply-connected, unimodular, solvable Lie group
    whose adjoint representation acts with only real eigenvalues.
    Then the maximal geometry subsuming $M/\{1\}$ is $M \rtimes K / K$
    where $K$ is a maximal compact subgroup of $\op{Aut} M$.
\end{lemma}
\begin{proof}
    Under the provided assumptions, in any invariant metric on $M$,
    there is some $K \subseteq \op{Aut} M$ such that the identity component
    of the isometry group of $M$ is $\idisom M = M \rtimes K$
    \cite[Thm.~4.3]{gordonwilsonisometry}.
    Since $K$ is the point stabilizer of the identity, it is compact.

    Since the transformation group of a maximal geometry is realizable
    as the isometry group in some invariant metric \cite[Prop.~1.1.2]{filipk},
    the maximal geometry subsuming $M/\{1\}$ is of the form
    $M \rtimes K / K$, with $K$ not in any larger compact group of automorphisms.
\end{proof}
When $M$ is nilpotent, the adjoint eigenvalues are always $1$; this case
is an earlier result by Wilson in \cite[Thm.~2(3)]{wilson1982isometry}.

The existence of non-product geometries forces a classification to confront
fiber bundle structures. Section \ref{chap:fiber_bundles} details how
the occurrence of these structures is controlled by
the action of the point stabilizer $G_p$ on the tangent space $T_p M$.
Since $G_p$ is compact, it preserves an inner product,
which allows expressing such a representation by a subgroup of $\op{SO}(5)$.
Figure \ref{fig:isotropy_poset} recalls the classification of such subgroups
from Part II in \cite[Prop.~\ref{ii:isotropy_classification}]{geng2}.

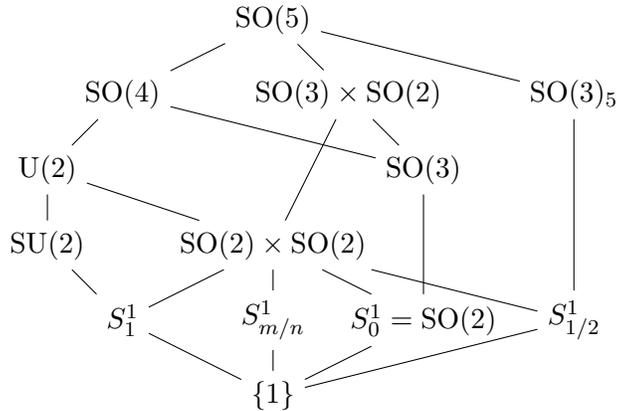
\begin{figure}[h!]
    \caption[Closed connected subgroups of $\op{SO}(5)$.
        (Duplicate of \ref{fig_thesis:isotropy_poset})]{
        Closed connected subgroups of $\op{SO}(5)$, with inclusions.
        $\op{SO}(3)_5$ denotes $\op{SO}(3)$ acting on its $5$-dimensional
        irreducible representation;
        and $S^1_{m/n}$ acts as on the direct sum $V_m \oplus V_n \oplus \R$
        where $S^1$ acts irreducibly on $V_m$ with kernel of order $m$.
    }
    \label{fig:isotropy_poset}
    \begin{center}\begin{tikzpicture}
        \draw (0, 5) node (so5) {$\op{SO}(5)$};
        \draw (-2, 4) node (so4) {$\op{SO}(4)$};
        \draw (1, 4) node (so3so2) {$\op{SO}(3) \times \op{SO}(2)$};
        \draw (4, 4) node (so35) {$\op{SO}(3)_5$};

        \draw (-3, 3) node (u2) {$\op{U}(2)$};
        \draw (-3, 2) node (su2) {$\op{SU}(2)$};

        \draw (2, 3) node (so3) {$\op{SO}(3)$};
        \draw (0, 2) node (t2) {$\op{SO}(2) \times \op{SO}(2)$};

        \draw (-2, 1) node (s11) {$S^1_1$};
        \draw (0, 1) node (s1q) {$S^1_{m/n}$};
        \draw (2, 1) node (so2) {$S^1_0 = \op{SO}(2)$};
        \draw (4, 1) node (s12) {$S^1_{1/2}$};

        \draw (0, 0) node (triv) {$\{1\}$};

        \draw (so5) -- (so4) -- (u2) -- (su2) -- (s11) -- (triv);
        \draw (so5) -- (so3so2) -- (so3) -- (so2) -- (triv);
        \draw (so5) -- (so35) -- (s12) -- (triv);
        \draw (so4) -- (so3);
        \draw (so3so2) -- (t2) -- (so2);
        \draw (u2) -- (t2) -- (s11);
        \draw (t2) -- (s1q) -- (triv);
        \draw (t2) -- (s12);
    \end{tikzpicture}\end{center}
\end{figure}

\subsection{Lie algebra extensions}

If $M$ is a $G$-invariant fiber bundle over a space $B$, then
the isometry group $G$ is an extension of the transformation group of $B$.
Passing to Lie algebras permits use of the well-known classification
of Lie algebra extensions by second cohomology, which this subsection recalls.
For more details,
a survey of Lie algebra cohomology in low degrees can be found in \cite{wagemann}
or \cite[\S 2--4]{alekseevsky_nonsuper}\footnote{
    An almost identical version, generalized to super (i.e.\ $\Z/2\Z$-graded)
    Lie algebras, has been published as \cite{alekseevsky_super}.
}.

\begin{defn}[\textbf{Lie algebra cohomology}, following {\cite[\S 2]{wagemann}}]
    \label{defn:liecoho}
    Let $M$ be a module of a Lie algebra $\lie{g}$ over a field $k$.
    The \keyword{Chevalley-Eilenberg complex} is the cochain complex
	is the cochain complex
		\[ C^p(\lie{g}, M) = \op{Hom}_k(\Lambda^p \lie{g}, M) \]
	with boundary maps
	\begin{align*}
		d_p: C^p &\to C^{p+1} \\
		(d_p c)(x_1, \ldots, x_{p+1})
			&= \sum_{1 \leq i < j \leq p+1} (-1)^{i+j}
            c\left( [x_i, x_j], x_1, \ldots, \hat{x_i}, \ldots, \hat{x_j}, \ldots, x_{p+1} \right) \\
				&\quad + \sum_{1 \leq i \leq p+1} (-1)^{i+1}
                    x_i c\left( x_1, \ldots, \hat{x_i}, \ldots, x_{p+1} \right)
	\end{align*}
    where $\hat{x_i}$ means $x_i$ should be omitted.
    The cohomology of $\lie{g}$ with coefficients in $M$
    is defined to be the cohomology of this complex
    and denoted $H^p(\lie{g}, M)$.
\end{defn}

\begin{thm}[\textbf{Classification of extensions by second cohomology}
        {\cite[Thm.\ 8]{alekseevsky_nonsuper}}] \label{thm:extensions_h2}
    Let $\lie{g}$ and $\lie{h}$ be Lie algebras and let
        \[ \overline{\alpha}: \lie{g} \to \op{out} \lie{h}
            = \op{der}(\lie{h})/\op{ad}(\lie{h}) \]
    be a Lie algebra homomorphism. Then the following are equivalent.
    \begin{enumerate}[(i)]
        \item For one (equivalently: any) linear lift
            $\alpha: \lie{g} \to \op{der} \lie{h}$ of $\overline{\alpha}$ choose
            $\rho: \bigwedge^2 \lie{g} \to \lie{h}$ satisfying
                \[ [\alpha_X, \alpha_Y] - \alpha_{[X,Y]} = \op{ad}_{\rho([X,Y])} . \]
            Then the cohomology class of
            $d \rho$ in $H^3(\lie{g}, Z(\lie{h}))$ vanishes,
            where the action of $\lie{g}$ on $Z(\lie{h})$ is induced by $\alpha$
            and $d$ is defined by the formula in \ref{defn:liecoho}
            (ignoring that $\lie{h}$ may not be a $\lie{g}$-module).
        \item There exists an extension $0 \to \lie{h} \to \lie{e} \to \lie{g} \to 0$
            inducing the homomorphism $\overline{\alpha}$.
    \end{enumerate}
    If either occurs, then the extensions satisfying (ii) are parametrized
    by $[\rho] \in H^2(\lie{g}; Z(\lie{h}))$.
\end{thm}

\subsection{Notations}

The naming of Lie groups and Lie algebras requires a number of notations;
the following conventions will be used throughout, with scattered reminders.
\begin{itemize}
    \item $\tanalg G$ denotes the tangent algebra of a Lie group $G$---the
        tangent space at the identity, identified
        with \emph{right}-invariant vector fields on $G$
        so that the resulting flows correspond with the action
        of $1$-parameter subgroups by multiplication on the left.
    \item $\idcompo G$ denotes the identity component of a topological group $G$,
        and $\idisom M$ is short for the identity component of the isometry group
        of $M$.
    \item Fraktur letters usually denote Lie algebras, e.g.\ an occurrence of
        $\lie{g}$ near $G$ probably means $\lie{g} = \tanalg G$;
        and $\tanisom M = \tanalg \op{Isom} M$.
\end{itemize}

\section{Fiber bundle structures on geometries}
\label{chap:fiber_bundles}

If $M = G/G_p$ admits a $G$-invariant fiber bundle structure $F \to M \to B$,
then $F$ and $B$ admit transitive group actions, which
affords some hope of reduction to lower-dimensional problems.
The aim of this section is to prove that
when $\dim M = 5$, reducibility of the isotropy representation
$G_p \curvearrowright T_p M$ both implies the existence of such
a fibering and constrains some of its properties
(Prop.\ \ref{prop:fibering_description}).
This begins with a definition of the structure being sought.
\begin{defn}[\textbf{Fiberings}]
    A geometry $M = G/G_p$ \keyword{fibers over} a smooth manifold $B$
    if $B$ is diffeomoprhic to $M/\foliation{F}$
    for some $G$-invariant foliation $\foliation{F}$
    with closed leaves.
    (Equivalently, $B \cong G/H$
    for some closed subgroup $H \subseteq G$ containing $G_p$.)
    The fibering is described as \keyword{isometric} if $B$ admits a $G$-invariant
    Riemannian metric, \keyword{conformal} if $B$ admits a $G$-invariant
    conformal structure, and \keyword{essentially conformal}
    (or \keyword{essential} for short) if it is conformal but not isometric.
\end{defn}
\begin{rmk}
    \label{rmk:fibering_closed_leaves}
    Closedness of $H$ ensures that $G/H$ has a natural smooth structure
    and smooth action by $G$ \cite[Thm.~II.4.2]{helgasonnew}
    and, by the existence of local cross sections for $B \to G$
    \cite[Lemma~II.4.1]{helgasonnew}, that $M$ is a smooth fiber bundle $F \to M \to B$
    where $F \cong H/G_p$ is a leaf of $\foliation{F}$.
\end{rmk}


This section's main result, guaranteeing the existence of
useful fiberings, is the following Proposition.
Its proof will be given in the last subsection,
after the first two subsections introduce the $G$-invariant distributions
that will be used to work infinitesimally with fiberings.
\begin{prop}[\textbf{The fibering description}] \label{prop:fibering_description}
    Let $M = G/G_p$ be a $5$-dimensional geometry, and let $d$
    be the dimension of the largest irreducible subrepresentation
    of $G_p \curvearrowright T_p M$.
    \begin{enumerate}[(i)]
        \item If $d = 4$ and $M$ is a model geometry
            then $M$ fibers isometrically over a $4$-dimensional
            simply-connected Riemannian symmetric space.
        \item If $d = 3$
            then $M$ fibers conformally over $S^3$, $\Euc^3$, or $\Hyp^3$.
        \item If $d = 2$ and $M$ is a model geometry
            then $M$ fibers conformally over $S^2$, $\Euc^2$, or $\Hyp^2$.
    \end{enumerate}
\end{prop}
\begin{rmk}[\textbf{Generalizations to higher dimensions}]
    The first two cases generalize to $n$ dimensions without major
    modifications of the proof or conclusions.
    Case (ii) is when the restriction of $T_p M$ to some normal subgroup of $G_p$
    decomposes with exactly one nontrivial irreducible subrepresentation,
    and case (i) is the sub-case when this subrepresentation has codimension $1$.
    In both of these generalizations, the base space is an isotropy irreducible space;
    and if the fibering is essentially conformal,
    the base is Euclidean or a sphere.

    Other fiberings can be produced from the ``natural bundle''
    if $M$ is compact \cite[\S II.5.3.2]{onishchik1},
    or by using the Levi decomposition (see e.g.\ \cite[\S 1.4]{onishchik3})
    as part of something like \cite[Thm.~C]{mostow2005}.
\end{rmk}
\begin{eg}[\textbf{Reducible isotropy does not in general imply fibering}]
    \label{eg:nofiber}
    If $G/G_p$ has an invariant fiber bundle structure with positive-dimensional
    fiber and base, then the point stabilizer of the base is an intermediate group
    $G_p \subsetneq H \subsetneq G$.
    Building on Dynkin's work classifying maximal proper connected subgroups,
    Dickinson and Kerr classified compact Riemannian homogeneous
    spaces with two isotropy summands in \cite[\S{6}]{dickinsonkerr}
    and found examples $G/G_p$ where $G_p$ is maximal.
    So these spaces have reducible isotropy but no nontrivial fiberings.

    The lowest-dimensional counterexample
    (that is compact and has two isotropy summands)
    is $\op{Sp}(3)/\op{Sp}(1)$, in dimension $18$ (Example V.10),
    where the embedding $\op{Sp}(1) \hookrightarrow \op{Sp}(3)$
    is given by the irreducible representation of $\op{Sp}(1) \cong \op{SU}(2)$
    on $\C^6$.
\end{eg}

\subsection{Invariant distributions and foliations}

This subsection introduces the fixed distributions (Defn.~\ref{defn:distros_vertical})
that will usually become the vertical distributions---i.e.\ tangent distributions
to the fibers in a fiber bundle.
Showing that such a distribution can be integrated to a foliation by submanifolds
will require some tools
such as the \emph{integrability tensor} (Lemma \ref{bracket_infinitesimally}).

The action of $G$ on $TM$ induces the correspondence
\[
    \left\{\begin{array}{c}
        G_p\text{-invariant subspaces} \\
        \text{(subrepresentations) of } T_p M
    \end{array}\right\}
    \leftrightarrow
    \left\{\begin{array}{c}
        G\text{-invariant} \\
        \text{distributions on } M = G/G_p
    \end{array}\right\},
\]
which provides a convenient way to refer to $G$-invariant distributions,
including the following.

\begin{defn}[\textbf{Fixed distributions of normal subgroups}]
    \label{defn:distros_vertical}
    Given a geometry $M = G/G_p$ and a normal subgroup $H \trianglelefteq G_p$,
    let $TM^H$ be the $G$-invariant distribution
    defined at $p$ by $T_p M^H$ (the subspace on which $H$ acts trivially).
    Let $TM^G$ denote the case $H = G_p$.
\end{defn}

The most visibly useful property of $TM^H$ is that it can be integrated
to produce the fibers of a fiber bundle. That is
(see also Rmk.~\ref{rmk:fibering_closed_leaves} above),

\begin{prop}
	\label{fixed_distros_integrable}
    $TM^H$ is integrable to a $G$-invariant foliation
    $\foliation{F}^H$ with closed leaves.
\end{prop}

The proof uses a construction known in the context of
Riemannian submersions as the ``integrability tensor''
(see e.g.\ \cite[Lemma 2]{oneill} and the discussion
before \cite[Thm.~3.5.5]{petersen}).
Its definition, its tensoriality, and its compatibility with
the action of point stabilizers are covered by the following rephrasing of
\cite[Lemma 3.2.2]{filipk}.

\begin{lemma}[\textbf{Integrability tensor as a representation homomorphism}]
    \label{bracket_infinitesimally}
	Let $D$ be a distribution on a manifold $M$.
    Then at each point $p$, the Lie bracket of vector fields
    induces a linear map $\mu: \Lambda^2 D_p \to T_p M / D_p$.
    If $D$ is $G$-invariant on a homogeneous space $M = G/G_p$,
    then $\Lambda^2 D_p \to T_p M / D_p$ is a
    $G_p$-representation homomorphism.
\end{lemma}
\begin{proof}
    Suppose $\{X_1, \ldots, X_k\}$ is a local basis of $D$,
    and $Y_1$ and $Y_2$ are vector fields in $D$ with
    $\vecspan{Y_1(p), Y_2(p)} \subseteq \vecspan{X_1(p), X_2(p)}$.
    Pick functions $\{a_i\}$ and $\{b_i\}$ so that
    \begin{align*}
        Y_1 &= \sum_{i=1}^m a_i X_i  &  Y_2 &= \sum_{i=1}^m b_i X_i
    \end{align*}
    in a neighborhood of $p$. Using the Leibniz rule for Lie brackets,
    \begin{align*}
        [Y_1, Y_2]
            &= \sum_{i,j} a_i b_j [X_i, X_j]
                + a_i X_i(b_j) X_j - b_j X_j(a_i) X_i .
    \end{align*}
    The last two terms are pointwise linear combinations of
    $X_1,\ldots,X_m$, so they add to some vector field $Z$ in $D$.
    At $p$, only $a_1$, $a_2$, $b_1$, and $b_2$ can be nonzero, so
    \begin{align*}
        [Y_1, Y_2](p)
            &= (a_1 b_2 - a_2 b_1)[X_1, X_2](p) + Z(p) ,
    \end{align*}
    which ensures that $\mu$ is well-defined.
    The $G_p$-equivariant version then follows by recalling
    that diffeomorphisms respect Lie brackets.
\end{proof}

\begin{proof}[Proof of Prop.~\ref{fixed_distros_integrable}]
    The integrability tensor
    $\mu: \Lambda^2 T_p M^H \to T_p M / (T_p M)^H$ is zero,
    since the domain is a trivial representation of $H$,
    and the codomain is a representation with no trivial summands.
    So by the Frobenius condition, $TM^H$ is integrable.

    A $G$-invariant foliation on a geometry $M = G/G_p$
    whose tangent distribution contains $TM^G$
    has closed leaves \cite[Prop.~2.1.1, 2.1.2]{filipk}.
    Since $H \subseteq G_p$ implies $T_p M^H \supseteq T_p M^{G_p}$,
    the foliation integrating $TM^H$ has closed leaves.
\end{proof}

\subsection{Complementary distributions and conformal actions}

If the distribution $TM^H$ is interpreted as vertical---i.e.\ tangent to the fibers of
a fibering $M \to B = M/\foliation{F}^H$---then a complementary horizontal distribution
should offer some understanding of the target $B$.
This subsection introduces such a distribution and its applications
to maximality of products (Prop.~\ref{prop:products_are_maximal})
and conformal fiberings (Lemma \ref{conformal_downstairs}).

\begin{defn}[\textbf{Complementary distributions}]
    Given a geometry $M = G/G_p$ and a normal subgroup $H \trianglelefteq G_p$,
    let $(TM^H)^\perp$ denote the distribution defined at $p$
    by the complementary $H$-representation to $(T_p M)^H$ in $T_p M$.
\end{defn}

Conveniently, this always agrees with the other distribution deserving
the same name---the orthogonal complement---which makes $TM^H$ and $(TM^H)^\perp$
useful together for studying invariant metrics on $M$. Explicitly,

\begin{lemma} \label{lemma:perp_orthogonality}
    $(TM^H)^\perp$ is the orthogonal complement to $TM^H$
    in every $G$-invariant metric on $M$.
\end{lemma}
\begin{proof}
    A $G_p$-invariant metric on $T_p M$
    induces a representation isomorphism $T_p M \cong_H T^*_p M$.
    Since $(TM^H)^\perp_p$ contains no trivial subrepresentation by definition,
    its image in the trivial representation $(T^*_p M)^H$ is zero;
    so the pairing of $(TM^H)^\perp_p$ with $(T_p M)^H$ is zero.
\end{proof}
\begin{rmk}
    This property makes any isometric fibering $M \to M/\foliation{F}^H$
    a \keyword{Riemannian submersion} (see e.g.\ the discussion
    before \cite[Example 1.1.5]{petersen}) for some invariant metric on $M$
    for each invariant metric on $M/\foliation{F}^H$.
\end{rmk}

Our main purposes in defining $(TM^H)^\perp$
are to understand properties of fiberings given properties
of the isotropy representation (Lemma \ref{conformal_downstairs} below,
used to establish the fibering description in Prop.~\ref{prop:fibering_description})
and to prove the following statement about maximality of products.

\begin{prop}[\textbf{Products are often maximal}] \label{prop:products_are_maximal}
    Given maximal geometries $M = G/G_p$ and $M' = G'/G'_q$, their product
        \[ M \times M' = (G \times G')/(G_p \times G'_q) \]
    is maximal provided that both of the following hold.
    \begin{enumerate}[(i)]
        \item At most one of $M$ and $M'$ is itself a product with a Euclidean factor.
        \item At least one of $M$ and $M'$ admits no nonzero invariant vector field.
    \end{enumerate}
\end{prop}
\begin{proof}
    Hano proved
    (see e.g.\ \cite[Thm.~VI.3.5]{kobayashinomizu})
    that the de Rham decomposition of a complete, simply-connected Riemannian manifold
    (see e.g.\ \cite[Thm.~IV.6.2]{kobayashinomizu})
    also decomposes the isometry group.
    So if $M$ and $M'$ are complete, connected,
    simply-connected Riemannian manifolds,
    one of which is not isometric to any product $M'' \times \Euc^k$, then
        \[ (\op{Isom} (M \times M'))^0 = (\op{Isom} M)^0 \times (\op{Isom} M')^0 . \]
    Since the transformation group of a maximal geometry
    is realizable as the isometry group in some metric \cite[Prop.~1.1.2]{filipk},
    it suffices to know that all invariant metrics on $M \times M'$
    are realizable as product metrics.
    Condition (ii) implies this as follows.

    To say that $M$ admits no invariant vector field means that $TM^G$ is zero.
    This implies that $\left(T_{(p,q)} (M \times M')\right)^{G_p} = T_q M'$,
    whose orthogonal complement in every invariant metric is the complementary
    $G_p$-representation $T_p M$ (Lemma~\ref{lemma:perp_orthogonality}).
\end{proof}

For non-product fiberings, the main use of $(TM^H)^\perp$
is in understanding metrics on $M/\foliation{F}^H$, using the following
slight generalization of part of \cite[Thm.~4.1.1]{filipk}.

\begin{lemma}[\textbf{Irreducible $+$ trivial isotropy
    produces conformal fibering}] \label{conformal_downstairs}
    Let $M = G/G_p$ be a geometry and $H$ a normal subgroup of $G_p$
    whose action on $(TM^H)^\perp_p$ is irreducible.
    Then there is a metric on $M/\foliation{F}^H$ with respect
    to which $G$ acts conformally.
\end{lemma}

A key ingredient in similar results such as \cite[Thm.~4.1.1]{filipk}
or the start of the proof of \cite[Thm.~3.8.4(b)]{thurstonbook}
is the observation that all points of the same fiber $F$ have the same
subgroup of $G$ for stabilizers.
Then one can speak of representation isomorphisms,
which the following standard fact turns into conformal maps.

\begin{lemma}[\textbf{Existence} {\cite[Thm.~II.1.7]{brockerdieck}}
    \textbf{and uniqueness} {\cite[App.~5 Thm.~1]{kobayashinomizu}}
    \textbf{of invariant inner products}] \label{lemma:inner_product_unique}
    Every finite-dimensional irreducible representation $V$
    of a compact Hausdorff group $K$ over $\R$
    has a unique $K$-invariant inner product
    (i.e.\ positive-definite symmetric bilinear form)
    up to scaling.
\end{lemma}

With this control over possible metrics,
Filipkiewicz then constructs the appropriate metric
using a partition-of-unity argument.
Here is the proof, with details of the steps outlined above.

\begin{proof}[Proof of the conformal fibering Lemma (\ref{conformal_downstairs})]
    This proof repeats most of \cite[Thm.~4.1.1]{filipk};
    the new material is mostly the first step, which works around
    the unavailability of the previously mentioned ``key ingredient'':
    points of the same fiber may not have the same stabilizer in $G$.

    \paragraph{Step 1: Find a group $G_F$ to replace $G_p$.}
    Given $p$ (thus $G_p$) and $H \trianglelefteq G_p$,
    let $F$ be the leaf of $\foliation{F}^H$ containing $p$,
	and let $G_F$ denote the subgroup of $G$ that acts as the identity on
    a leaf $F$ of $\foliation{F}^H$.
    Then $H \subseteq G_F \subseteq G_p$,
    so $G_F$ is a compact group acting
    irreducibly on $(TM^H)^\perp_q$ for every $q \in F$.

    \paragraph{Step 2: Isomorphisms of $G_F$-irreps become conformal linear maps.}
    Fix a $G$-invariant metric $\mu$ on $M$.
    Its restriction $\mu|_{(TM^H)^\perp}$ to the (invariant) distribution $(TM^H)^\perp$ is $G$-invariant,
    hence $G_F$-invariant for every leaf $F$ of $M/\foliation{F}^H$.

    The projection $\pi: M \to M/\foliation{F}^H$
    induces $G_F$-representation isomorphisms
        \[ (TM^H)^\perp_q \to T_{\pi(q)} (M/\foliation{F}^H) , \]
    each producing from $\mu|_{(TM^H)^\perp}$
    a $G_F$-invariant inner product on a tangent space to $M/\foliation{F}^H$.
    Since invariant inner products on irreducible representations
    are unique up to scaling (Lemma \ref{lemma:inner_product_unique} above),
    all such pushed-forward inner products on the same tangent space
    are scalar multiples of each other.

    \paragraph{Step 3: Give sufficient conditions for a conformal structure
        to be $G$-invariant.}
    If a metric on $M/\foliation{F}^H$ is pointwise
    a linear combination of the pushforwards described in Step 2,
    then $G$ preserves its conformal class
    since $G$ preserves $\mu$, $\foliation{F}^H$, and $(TM^H)^\perp$.
    What remains is to construct such a metric---i.e.\ to show
    that these pushforwards can be chosen in a smoothly varying way---using
    the partition-of-unity method from
    \cite[Thm.~4.1.1]{filipk}.

    \paragraph{Step 4: Construct the metric on $M/\foliation{F}^H$.}
    Let $k = \dim (TM^H)^\perp$,
    and choose a collection of $k$-discs $V_i \subset M$ that
    \begin{enumerate}
        \item are transverse to $TM^H$,
        \item each map diffeomorphically to $M/\foliation{F}^H$, and
        \item cover $M/\foliation{F}^H$ with their images.
    \end{enumerate}
    The cover of $M/\foliation{F}^H$ can be made locally finite
    since $M/\foliation{F}^H$ is a manifold; so
    there is a partition of unity $\{\phi_i\}$ subordinate to $\{\pi(V_i)\}$.

    Let $\rho: TM \to (TM^H)^\perp$ be the projection with kernel $TM^H$.
    On each $V_i$, let $\mu_i$ be the metric defined pointwise
    as the pullback of $\mu|_{(TM^H)^\perp}$ by $\rho$. Then the sum
        \[ \bar{\mu} = \sum_i \phi_i \pi_* \mu_i \]
    defines a metric on $M/\foliation{F}^H$ with the property
    described in Step 3.
\end{proof}

\subsection{Proof of the fibering description (Proposition \ref{prop:fibering_description})}

Two more facts will be needed for the proof of Prop.~\ref{prop:fibering_description}.
First,
the base is simply-connected in every case of Prop.~\ref{prop:fibering_description}:
by definition, a geometry $M$ is simply-connected and a leaf $F$ of $\foliation{F}$
is connected. So the homotopy exact sequence for $F \to M \to M/\foliation{F}$
implies $M/\foliation{F}$ is simply-connected. Hence the proofs to follow will
skip proving simply-connectedness.

The second fact is a lemma extracted from Thurston's 3-dimensional classification,
used to reason about invariant vector fields in cases (i) and (iii).
\begin{lemma}[see e.g.\ {\cite[Proof of 3.8.4(b)]{thurstonbook}}]
    \label{lemma:divergence_free}
    A $G$-invariant vector field $X$ on a model geometry $M = G/G_p$
    is divergence-free.
\end{lemma}
\begin{proof}
    Let $\phi_t$ be the $G$-equivariant flow on $M$ integrating $X$.
    If $N$ is any finite-volume manifold modeled on $M$,
    then $X$ and $\phi_t$ descend to some $\ol{X}$ and $\ol{\phi_t}$ on $N$,
    along with any $G$-invariant metric.
    Since $\op{vol} N$ is finite,
    $\ol{\phi_t}$ is globally volume-preserving;
    so $\int_N \op{div} \ol{X} \, d\op{vol} = 0$.
    Since $G$ acts transitively on $M$,
    the value of $\op{div} X = \op{div} \ol{X}$ is constant---thus $0$.
\end{proof}

The proof of the fibering description (Prop.~\ref{prop:fibering_description})
can now proceed. The three cases exhibit somewhat different behavior,
so they are handled separately. In particular, case (iii) (fiberings over
$2$-dimensional bases) is a bit more work to prove than the others,
due to isotropy representations $G_p \curvearrowright T_p M$
in which there is not a canonical choice of a $2$-dimensional irreducible summand
to form the horizontal distribution.

\subsubsection{Case (i): over dimension 4}

\begin{proof}[Proof of \ref{prop:fibering_description}(i)]
    Let $M = G/G_p$ be a model geometry for which $G_p \curvearrowright T_p M$
    has irreducible subrepresentations of dimensions $1$ and $4$.
    This proof mostly follows the argument in \cite[Thm.~3.8.4(b)]{thurstonbook}.

    \paragraph{Case (i), Step 1: $M$ fibers over a $4$-dimensional space.}
    The distribution $TM^G$ is $1$-dimensional
    and integrates to a foliation with closed leaves
    (Prop.~\ref{fixed_distros_integrable}), so $M$ fibers
    over a $4$-dimensional space $M/\foliation{F}^G$.

    \paragraph{Case (i), Step 2: $M$ fibers isometrically.}
    A nonzero vector in $(T_p M)^G$ pushes forward by the action of $G$
    to a $G$-invariant vector field $X$ on $M$, with corresponding
    $G$-equivariant flow $\phi_t$. Then $\phi_t$ commutes with the action of $G$,
    so all points in the same $\phi_t$-orbit have the same stabilizer in $G$.
    Therefore $d_p \phi_t: T_p M \to T_{\phi_t(p)} M$
    is an isomorphism of $G_p$-representations.

    Since $M$ is a model geometry, $\phi_t$ is volume-preserving
    (Lemma~\ref{lemma:divergence_free}).
    Combined with the fact that $\phi_t$ preserves the metric
    (the fiberwise inner product) on $TM^G$
    since it preserves its own velocity field $X$,
    this implies $\phi_t$ preserves volume on the orthogonal complement $(TM^G)^\perp$.
    Since $(TM^G)^\perp$ is irreducible, $\phi_t$ preserves the metric
    on it (Lemma \ref{lemma:inner_product_unique}).
    So the metric on $(TM^G)^\perp$ descends to a $G$-invariant
    metric on $M/\foliation{F}^G$.

    \paragraph{Case (i), Step 3: $M/\foliation{F}^G$ is Riemannian symmetric.}
    By the classification of isotropy groups
    (Fig.~\ref{fig:isotropy_poset}),
    $G_p$ is $\op{SO}(4)$, $\op{U}(2)$, or $\op{SU}(2)$.
    In their standard representations on $\R^4$,
    all of these contain the scalar $-1$, which reverses
    tangent vectors---and thus geodesics---at the image of $p$ in $M/\foliation{F}^G$.
\end{proof}

\subsubsection{Case (ii): over dimension 3}

\begin{proof}[Proof of \ref{prop:fibering_description}(ii)]
    Suppose $M = G/G_p$ is a geometry where $G_p \curvearrowright T_p M$
    contains an irreducible 3-dimensional subrepresentation.
    By the classification of isotropy groups
    (Fig.~\ref{fig:isotropy_poset}),
    $G_p$ contains $\op{SO}(3)$ as a characteristic subgroup.
    Therefore $M$ fibers conformally over $B = M/\foliation{F}^{\op{SO}(3)}$
    (Lemma \ref{conformal_downstairs}).

    Then $G$ acts transitively on $B$, with $\op{SO}(3)$ in the point
    stabilizers. If the fibering $M \to B$ is isometric,
    then $B$ is one of the $3$-dimensional constant-curvature spaces.
    Otherwise,
    \begin{enumerate}
        \item a manifold $B$ with a transitive\footnote{
                A stronger version of Obata's theorem
                (i.e.\ without assuming transitivity)
                could instead be used here,
                but its proof involves some analytic subtleties---see
                \cite{ferrand1996} for an overview.
                The transitive action of $G$ provides a way to
                sidestep these issues.
            }
            essential conformal automorphism group $\op{Conf} B$
            is conformally flat \cite[Lemma 1]{obata1973}; and
        \item if $B$ is conformally flat and the identity component of $\op{Conf} B$
            acts essentially, then $B$ is conformally equivalent
            to a sphere or Euclidean space \cite[Thm.~D.1]{lafontaine}.
    \end{enumerate}
    So if the fibering $M \to B$ is essential, then $B$ is $S^3$ or $\Euc^3$.
\end{proof}

\subsubsection{Case (iii): over dimension 2}

\begin{proof}[Proof of \ref{prop:fibering_description}(iii)]
    Suppose $M = G/G_p$ is a model geometry where the irreducible
    subrepresentations of $G_p \curvearrowright T_p M$
    have dimensions $1$ and $2$.

    If $G_p$ is $\op{SO}(2)$ or $\op{SO}(2) \times \op{SO}(2)$, then
    a strategy like that of case (ii) applies:
    since $G_p$ contains $\op{SO}(2)$ as a normal subgroup,
    $M$ fibers conformally over $B = M/\foliation{F}^{\op{SO}(2)}$
    (Lemma \ref{conformal_downstairs}).
    Since $B$ is simply-connected by the homotopy exact sequence,
    the Uniformization Theorem
    (see e.g.\ \cite[Thm.~10-3]{ahlfors})
    implies it is conformally equivalent to
    $S^2$, $\Euc^2$, or $\Hyp^2$.

    The remaining case is when $G_p = S^1_{m/n}$
    (a $1$-parameter subgroup of $\op{SO}(2) \times \op{SO}(2)$).
    By the same application of the Uniformization Theorem, it suffices
    to find a $G$-invariant foliation $\foliation{F}$ of codimension $2$
    such that $M$ fibers conformally over $M/\foliation{F}$.
    The strategy consists of the following two steps.
    First, a foliation $\foliation{F}$ with closed leaves and codimension $2$
    is found by examining Lie brackets of vector fields
    tangent to subrepresentations of $G_p \curvearrowright T_p M$.
    Then if $M$ does not fiber conformally over $M/\foliation{F}$,
    this information is used to find an alternative fibering of $M$.

    \paragraph{Case (iii), Step 1: Finding a foliation with closed leaves.}
    $G_p = S^1_{m/n}$ means that
		\[
			T_p M
			= (T_p M)^G \oplus ((T_p M)^G)^\perp
			= (T_p M)^G \oplus V_m \oplus V_n ,
		\]
	where $(T_p M)^G$
	is a $1$-dimensional trivial representation of $S^1$,
	and $V_m$ is an irreducible representation on which
    $S^1$ acts with kernel of size $m$.

	If $m \neq n$, then $V_m \ncong V_n$.
	The integrability tensor
	(Lemma \ref{bracket_infinitesimally})
    is a $G_p$-representation homomorphism
		\[
			\R \oplus V_m \cong
			\Lambda^2 ((T_p M)^G \oplus V_m)
			\to T_p M / ((T_p M)^G \oplus V_m)
			\cong V_n
		\]
    that agrees with Lie brackets of vector fields.
    This is zero since $\R$, $V_m$, and $V_n$ are non-isomorphic irreducible
    representations;
	so Lie brackets preserve tangency to the $3$-plane $(T_p M)^G \oplus V_m$.
	Hence the corresponding $G$-invariant distribution is integrable.
    The resulting foliation has closed leaves
	since its tangent distribution contains $TM^G$,
    (\hspace{1sp}\cite[Prop.~2.1.2]{filipk}; also used earlier in
    Prop.~\ref{fixed_distros_integrable}).

	If $m = n$ then $V_m$ and $V_n$ are both isomorphic to
	the standard representation $V_1$ of $SO(2)$. This $SO(2)$ action
	induces the structure of a complex vector space on
	$V_m \oplus V_n = ((T_p M)^G)^\perp$.

	Since $TM^G$ is $1$-dimensional, there is a canonical
	(up to rescaling) $G$-invariant vector field $v$ along $TM^G$,
    with zero divergence (Lemma ~\ref{lemma:divergence_free}).
	Fixing $p \in M$, choose $w \in \tanalg G \hookrightarrow \op{Vect} M$
	such that $w(p) = v(p)$ (possible since $G \to M$ is a submersion).
	Since $G$ acts by isometries, $w$ satisfies the following.
	\begin{itemize}
		\item $w$ is also divergence-free.
		\item $\op{exp} tw$ preserves $v$, so $[v,w] = 0$.
		\item $\op{exp} tw$ preserves the $(\op{exp} tv)$-orbit of $p$,
			so it preserves $G_p$. Conjugation induces a map
				\[ t \in \R \to \op{Aut} G_p \cong \{\pm 1\} . \]
			Since $\R$ is connected, this is the trivial homomorphism;
			so $\op{exp} tw$ commutes with $G_p$.
	\end{itemize}
	Then $v-w$ is divergence-free, $G_p$-invariant, and zero at $p$;
	so $\op{exp} t(v-w)$ is a $G_p$-equivariant operator on $T_p M$---i.e.
	a $\C$-linear operator on $((T_p M)^G)^\perp$. It has an eigenvector,
	whose span $V_p$ (over $\C$) is $G_p$-invariant. Let $V$ be the $G$-invariant
	distribution whose $2$-plane at $p$ is $V_p$.

	Since $V_p$ is an eigenspace
	of $\op{exp} t(v-w)$ (thus of $v-w$) and $V$ is $G$-invariant,
		\[ [v, V](p) \subseteq [v-w, V](p) + [w, V](p) \subseteq V_p + V_p = V_p . \]
	The map $\Lambda^2 V \to T_p M / V$ from Lemma \ref{bracket_infinitesimally}
	lands in a trivial representation since $\Lambda^2 V$ is $1$-dimensional;
	so $[V, V]_p \subseteq (T_p M)^G$. Therefore the $G$-invariant
	distribution of $3$-planes $TM^G \oplus V$ is integrable.
    As above, it contains $TM^G$ so the resulting foliation has closed leaves.
    
    \paragraph{Case (iii), Step 2: $M$ fibers conformally over some space.}
    Let $\foliation{F}$ be the foliation tangent to the distribution
    produced in Step 1, and suppose $M$ does not fiber conformally
    over $B = M/\foliation{F}$---i.e.\ there is no metric on $B$
    with respect to which $G$ acts by conformal automorphisms.

    The subgroup $G_b \subset G$ fixing $b \in B$ acts on the tangent space $T_b B$
    by some
		\[ G_b \to \op{GL}(T_b B)) \cong \op{GL}(2,\R) . \]
    Since $G_p \subseteq G_b$ and $G$ preserves no
    conformal structure on $B$,
    there are two non-coincident copies of $\op{SO}(2)$
	acting on $T_b B$. Together they generate
	all of $\op{SL}(2,\R)$ (compare this to rotation groups around two distinct points
	of $\Hyp^2$ generating all of $\isomplus \Hyp^2$),
    so the semisimple part of $G$ contains a subgroup covering $\op{PSL}(2,\R)$.
	Since $G$ is $6$-dimensional, the classification of simple Lie groups
    \cite[Ch. X, \S{}6 (p. 516)]{helgasonnew},
	implies that $G$ either surjects onto $\op{PSL}(2,\R)$
	or covers $\op{PSL}(2,\C)$.

	In both $\op{PSL}(2,\R)$ and $\op{PSL}(2,\C)$,
    the maximal torus of the maximal compact subgroup is $1$-dimensional;
    so all copies of $S^1$ are conjugate, and $G_p \cong S^1$
	lands in some copy. Thus $M$ fibers isometrically over
	(or is) $\op{PSL}(2,\R)/\op{PSO}(2) \cong \Hyp^2$ or $\op{PSL}(2,\C)/\op{PSO}(2)$.
	The latter fibers conformally over $S^2$, as
		\[ S^2
			\cong \op{Conf}^+ S^2/\op{Conf}^+ \Euc^2
			\cong \op{PSL}(2,\C)/\op{Conf}^+ \Euc^2 . \]
	Fibers are closed since the maps used are continuous,
	and $SO(2)$ is in the point stabilizers on the base
	since $G$ surjects onto $\op{PSL}(2,\R)$ or $\op{PSL}(2,\C)$.
\end{proof}

\begin{eg}
    The necessity of Step 2 is demonstrated by the fibering
        \[ \op{SL}(2,\R) \to \R^2 \rtimes \op{SL}(2,\R) \to \R^2 . \]
    Since $\R^2 \rtimes \op{SL}(2,\R)$
    acts on $\R^2$ with $\op{SL}(2,\R)$ point stabilizers,
	it cannot preserve any conformal structure.
    If Step 1 had produced this fibering, then
    Step 2 would find the conformal fibering
        \[ \R^2 \rtimes \op{SO}(2) \to \R^2 \rtimes \op{SL}(2,\R) \to \Hyp^2 . \]
\end{eg}

\section{Geometries fibering over 4D isotropy-irreducible geometries}
\label{chap:fiber4}

This section proves part (i) of Theorem \ref{thm:main}---that is,
\begin{prop}[Thm.~\ref{thm:main}(i)] \label{prop:main:i}
    The $5$-dimensional maximal model geometries $M = G/G_p$
    for which the isotropy representation $G_p \curvearrowright T_p M$
    contains an irreducible $4$-dimensional summand are the product geometries
    \begin{align*}
        S^4 &\times \Euc & \Hyp^4 &\times \Euc &
        \C P^2 &\times \Euc & \C\Hyp^2 &\times \Euc
    \end{align*}
    and the homogeneous spaces
    \begin{align*}
        &{}\widetilde{\op{U}(2,1)/\op{U}(2)} &
        \Heis_5 &= \Heis_5 \rtimes \op{U}(2)/\op{U}(2) .
    \end{align*}
\end{prop}
Our approach to proving Prop.~\ref{prop:main:i}
closely resembles the classification in \cite[Thm.~3.8.4(b)]{thurstonbook}
of $3$-dimensional geometries with an irreducible $2$-dimensional isotropy summand.
Recall that
$B = M/\foliation{F}^G$ is a $4$-dimensional Riemannian symmetric space
over which $M$ fibers isometrically
(Prop.\ \ref{prop:fibering_description}).
Curvature determines $M$ once $B$ is known. That is:
\begin{prop}[\textbf{Base and curvature determine the geometry}]
    \label{prop:fiber4_base_and_curvature}
    A $5$-dimensional maximal model geometry $M = G/G_p$
    whose isotropy representation contains an irreducible $4$-dimensional summand
    is determined by the following two pieces of information:
    \begin{enumerate}
        \item the geometry $B = M/\foliation{F}^G$; and
        \item whether $(TM^G)^\perp$, as a connection on the fiber bundle
            $M \to B$, has nonzero curvature.
    \end{enumerate}
    Moreover, if $G_p = \op{SO}(4)$, then $(TM^G)^\perp$ has zero curvature.
\end{prop}
This key fact, proven in subsection \ref{sec:fiber4_recipe},
reduces the classification problem to listing pairs $(B,x)$
and checking whether each arises from a maximal model geometry.
Subsection \ref{sec:fiber4_proof} carries this out,
finding and then verifying the candidates
listed in Table \ref{table:fiber4_candidates}.
\begin{table}[h!]
    \caption{Candidate geometries with irreducible $4$-dimensional isotropy summand}
    \label{table:fiber4_candidates}
    \label{table_thesis:fiber4_candidates}
    \begin{center}\begin{tabular}{c|cc}
        Base  &
        Flat (product)  &
        Curved  \\
        \hline
        \rule[0pt]{0pt}{13pt}
        $S^4$  &
        $S^4 \times \Euc$  &
        \\
        $\Euc^4$  &
        non-maximal $\Euc^5$  &
        \\
        $\Hyp^4$  &
        $\Hyp^4 \times \Euc$  &
        \\
        $\C P^2$  &
        $\C P^2 \times \Euc$  &
        non-maximal $S^5$  \\
        $\C^2$  &
        non-maximal $\Euc^5$  &
        $\Heis_5$  \\
        $\C \Hyp^2$  &
        $\C \Hyp^2 \times \Euc$  &
        $\widetilde{\op{U}(2,1)/\op{U}(2)}$  \\
    \end{tabular}\end{center}
\end{table}

\subsection{Reconstructing geometries from base and curvature information}
\label{sec:fiber4_recipe}

This subsection proves Prop.~\ref{prop:fiber4_base_and_curvature} in two steps:
one recovers a connection from its curvature using some general theory,
and one normalizes any nonzero curvature to a single value.
The latter interprets $(TM^G)_p^\perp$ as the quaternions
$\mathbb{H}$ and of $\op{SU}(2)$ as the unit quaternions
in order to write down and work with the isotropy representation, as follows.
Since the action
\begin{align*}
    \widetilde{\op{SO}(4)} \cong
    \op{SU}(2) \times \op{SU}(2) &\curvearrowright \mathbb{H} \\
    (p,q)(x) &= pxq^{-1}
\end{align*}
descends to the standard representation of $\op{SO}(4)$,
all three of the isotropy representations with a $4$-dimensional irreducible
summand---$\op{SU}(2) \subset \op{U}(2) \subset \op{SO}(4)$
(Fig.~\ref{fig:isotropy_poset})---can be written using quaternion multiplication
and subgroups of $\op{SU}(2) \times \op{SU}(2)$.

\begin{proof}[Proof of Prop.~\ref{prop:fiber4_base_and_curvature}]
    Let $M = G/G_p$ be a $5$-dimensional maximal model geometry
    whose isotropy representation contains an irreducible $4$-dimensional summand,
    and let $B = M/\foliation{F}^G$.
    Then $M \to B$ is a $G$-invariant principal $S^1$- or $\R$-bundle, with
    vertical distribution $TM^G$ and a natural connection (horizontal distribution)
    $(TM^G)^\perp$. The curvature is a $G$-invariant $2$-form with values in
    the Lie algebra of $S^1$ or $\R$, i.e.\ an element of
    $(\Omega^2 B \otimes \R)^G \cong (\Lambda^2 (TM^G)_p^\perp)^{G_p}$.

    \paragraph{Step 1: Nonzero curvature can be normalized to a single value.}
    By counting weight vectors,
    $\Lambda^2 (TM^G)_p^\perp \cong_{\op{SU}(2)} 3\R \oplus \R^3$
    (i.e. $\op{SU}(2)$ acts trivially on $3\R$ and as $SO(3)$ on $\R^3$).
    One checks by inspection that the $3\R$ is spanned by
    \begin{align*}
        1 \wedge i &- j \wedge k  &
        1 \wedge j &- k \wedge i  &
        1 \wedge k &- i \wedge j .
    \end{align*}
    Under the action of conjugation by unit quaternions, $(TM^G)_p^\perp$
    decomposes instead as $\R \oplus \R^3$; so its second exterior power
    decomposes as $2\R^3$.
    In fact one of these copies of $\R^3$ is the $3\R$ above, since
    conjugation by $1 + i$ takes $1 \wedge j - k \wedge i$
    to $1 \wedge k - i \wedge j$.

    Then if $G_p \cong SU(2)$, applying some inner automorphism
    of $SU(2)$ makes the curvature a scalar multiple of
    $1 \wedge i - j \wedge k$; and the fibers can be rescaled to make
    the curvature either $0$ or $1 \wedge i - j \wedge k$.

    If instead $G_p \cong U(2)$, take the scalar factor
    to act as multiplication by $e^{i\theta}$ on the right.
    This acts on the span of
    $1 \wedge j - k \wedge i$ and $1 \wedge k - i \wedge j$ by rotation,
    so $1 \wedge i - j \wedge k$ lies in the only invariant direction.

    Finally, if $G_p \cong SO(4)$, then the action of $SU(2)$
    by conjugation of quaternions factors through $G_p$,
    so $\Lambda^2 \R^4$ contains no invariant directions---so
    the curvature is $0$.

    \paragraph{Step 2: Base and curvature determine a geometry.}

    Let $\pi$ be the projection $M \to B$.
    To $x$ in a neighborhood $U$ of $p$, assign the coordinates $(t, \pi(x))$
    where the shortest path from $\pi(x)$ to $\pi(p)$
    lifts to a horizontal path from $x$ to $\phi_t(p)$.
    (For these coordinates to be well-defined, it suffices to
    have the radius of $\pi(U)$ at most the injectivity radius of $B$.)

    If $M$ and $N$ have matching curvature and base,
    then choose $p \in M$ and $q \in N$
    lying over the same point $b_0 \in B$; and define a map $f$
    from $U \ni p$ to $V \ni q$ by $(t,b) \mapsto (t,b)$.
    Since homogeneous spaces are analytic \cite[Prop.~I.4.2]{kobayashinomizu},
    and the isometry type of a complete, connected, simply-connected,
    analytic Riemannian manifold is determined by its local isometry type
    \cite[Cor.~VI.6.4]{kobayashinomizu},
    two geometries are isometric if they contain isometric open sets.
    Since a maximal geometry is determined by any invariant Riemannian
    metric \cite[Prop.~1.1.2]{filipk},
    it suffices to check that $f$ is an isometry.

    Since $f$ descends to the identity on $B$,
    and the metric on $M$ is the direct sum of its restrictions
    to $TM^G$ and $(TM^G)^\perp$,
    it suffices to check that $f$ takes the horizontal distribution
    $(TM^G)^\perp$ on $M$ to the horizontal distribution on $N$.
    (We'll say ``$f$ is horizontal''.)

    Since $f$ takes horizontal lifts of geodesics through $p$
    to horizontal lifts of geodesics through $q$, it's horizontal at $p$.
    At points other than $p$, since $M$ and $N$ have
    matching curvature and base, it suffices to check that there is only
    one $4$-plane distribution with the prescribed $4$-plane at $p$
    and the prescribed curvature.

    Let $\tilde{S}$ be a circular sector in $T_{b_0} B$ with its vertex
    at the origin, and let $S = \op{exp} \tilde{S}$.
    The displacement along the fiber $\mathcal{F}_{p}$
    incurred by traveling around a horizontal lift of $\partial S$
    is the integral of the curvature over $S$.
    So if $s$ denotes the distance along the circular arc in $\partial S$,
    then computing $\frac{dt}{ds}$ in enough directions recovers
    the slope of the horizontal distribution relative to
    the coordinates $(t,b)$.

\end{proof}

\subsection{Classifying and verifying geometries}
\label{sec:fiber4_proof}

Having established that the base and curvature determine a
$5$-dimensional maximal model geometry with irreducible $4$-dimensional isotropy
(Prop.~\ref{prop:fiber4_base_and_curvature}),
this subsection carries out the classification (i.e.\ the proof
of Prop.~\ref{prop:main:i}) according to the plan outlined
at the start of the section.

\begin{proof}[Proof of Thm.~\ref{thm:main}(i)/Prop.~\ref{prop:main:i}]
    Let $M = G/G_p$ be a $5$-dimensional maximal model geometry
    whose isotropy representation contains an irreducible $4$-dimensional summand.
    The base $M/\foliation{F}^G$ and the curvature of $(TM^G)^\perp$
    determine $M$ (Prop.~\ref{prop:fiber4_base_and_curvature});
    so $M$ occurs in Table \ref{table:fiber4_candidates},
    provided that the list of base spaces is exhaustive (Step 1)
    and that every entry has the claimed base and curvature (Step 2).
    Determining which candidates are maximal and model (Step 3)
    finishes the proof.

    \paragraph{Step 1: Classify base spaces for Table \ref{table:fiber4_candidates}.}
    The base $M/\foliation{F}^G$ is a Riemannian symmetric space
    (Prop.~\ref{prop:fibering_description}). The isotropy in $M$ descends
    to irreducible isotropy in $M/\foliation{F}^G$, so $M/\foliation{F}^G$
    is either irreducible or Euclidean. Consulting the classification
    of irreducible Riemannian symmetric spaces in
    \cite[X.6, p.515--518]{helgasonnew}
    yields the following non-Euclidean base spaces.
    \begin{align*}
        \Hyp^4 & &
        S^4 & &
        \C P^2 \cong \op{SU}(3)/\op{S}(\op{U}(2) \times \op{U}(1)) & &
        \C \Hyp^2 \cong \op{SU}(2,1)/\op{S}(\op{U}(2) \times \op{U}(1))
    \end{align*}
    If $M/\foliation{F}^G$ is Euclidean, it may not be a maximal geometry---its
    isotropy may be some proper but irreducibly-acting subgroup of $\op{SO}(4)$.
    Hence Table \ref{table:fiber4_candidates} also lists $\C^2$ as a base,
    to stand for Euclidean space with $\op{SU}(2)$ or $\op{U}(2)$ isotropy.

    \paragraph{Step 2: Verify the candidates in Table \ref{table:fiber4_candidates}.}
    This step proves that every spot in Table \ref{table:fiber4_candidates}
    is occupied by a space with the correct base and curvature
    (or empty if no such space exists).

    For the product spaces, there is nothing to prove.
    For the bases with $\op{SO}(4)$ isotropy, no non-product
    geometries occur (Prop.~\ref{prop:fiber4_base_and_curvature}).
    For the remaining bases, geometries with nonzero curvature of $(TM^G)^\perp$
    are realized by the following.
    \begin{itemize}
        \item Over $\C P^2$ is its tautological circle bundle $S^5$,
            since $\C P^2$ is the quotient of $S^5 \subset \C^3$
            by the action of norm-$1$ scalars. Its curvature is nonzero
            since a flat $S^1$-bundle over the simply-connected $\C P^2$
            is the product bundle---which, unlike $S^5$, has nontrivial $\pi_1$.
        \item Over $\C^2$ is $\Heis_5$, the $5$-dimensional Heisenberg group,
            which can be written as the set $\C^2 \times \R$ with the product
                \[ (v_1,t_1) (v_2, t_2)
                    = (v_1 + v_2, t_1 + t_2 + \op{Im} \left<v_1, v_2\right>) , \]
            where $\left<\cdot,\cdot\right>$ is the standard Hermitian product.
            Dropping the $\R$ coordinate is a fibering over $\C^2$,
            for which an invariant connection with nonzero curvature
            is given by the kernel of the invariant contact form
                \[ \alpha_{(v,t)} = dt - \op{Im} \left<v, dv\right> . \]
        \item Over $\C \Hyp^2 \cong U(2,1)/(U(2) \times U(1))$
            is the line bundle $\widetilde{U(2,1)/U(2)}$.
            That its curvature is nonzero is Prop.~\ref{prop:fiber4_hyp_circle_bundle}
            below.
    \end{itemize}

    \paragraph{Step 3: Determine maximal model geometries.}
    Except when multiple Euclidean factors are involved,
    products of maximal model geometries are maximal model geometries
    (Prop.~\ref{prop:products_are_maximal} and surrounding discussion);
    so it suffices to consider the three non-product geometries.
    \begin{itemize}
        \item $S^5$ as a circle bundle over $\C P^2$ is not maximal---it
            is a homogeneous space $\op{SU}(3)/\op{SU}(2)$ or $\op{U}(3)/\op{U}(2)$,
            both of which are subsumed by the geometry $\op{SO}(6)/\op{SO}(5)$.
        \item $\Heis_5$ is a nilpotent Lie group.
            Its maximal realization is $\Heis_5 \rtimes K / K$
            where $K \subset \op{Aut} \Heis_5$ is maximal compact
            (Lemma \ref{lemma:solvable_maximality}).
            Since $\op{Aut} \Heis_5$ must preserve $Z(\Heis_5)$,
            it is block triangular in some basis with blocks of size $4$ and $1$;
            so its maximal compact subgroup
            is conjugate to a subgroup of $\op{SO}(4) \times \op{SO}(1)$
            \cite[Lemma \ref{ii:lemma:compacts_in_aut}]{geng2}.
            Moreover, $\op{Aut} \Heis_5$ has to preserve the antisymmetric
            pairing on $\Heis_5/Z(\Heis_5)$ used to define $\Heis_5$ in Step 2.
            Then
                \[ K \subseteq \op{SO}(4) \cap \op{Sp}(4,\R) = \op{U}(2) \]
            since $\op{SO}(4)$ preserves the real part of a Hermitian form
            and $\op{Sp}(4,\R)$ preserves the imaginary part.
            Conversely, $K \supseteq \op{U}(2)$
            since $\op{U}(2)$ preserves the Hermitian
            product used to define $\Heis_5$.
            So $\Heis_5 \rtimes \op{U}(2) / \op{U}(2)$ is maximal.
        \item $\widetilde{\op{U}(2,1)/\op{U}(2)}$ is a model geometry since
            $\C \Hyp^2$ is---the deck group $\Gamma$ of any compact
            $\Gamma \backslash \C \Hyp^2$ acts by elements of
            $\op{SU}(2,1) \subset \op{U}(2,1)$, so the circle bundle
            $\op{U}(2,1)/\op{U}(2)$ over $\C \Hyp^2$ descends to a circle bundle
            $N$ over $\Gamma \backslash \C \Hyp^2$.
            To prove maximality it suffices to distinguish
            $\widetilde{\op{U}(2,1)/\op{U}(2)}$ from geometries
            with larger isotropy group; consulting Figure \ref{fig:isotropy_poset},
            these are just $\op{SO}(5)$ and $\op{SO}(4)$.
            Geometries $M = G/G_p$ with irreducible $4$-dimensional isotropy
            are distinguished from each other by $M/\foliation{F}^G$
            and the curvature of $(TM^G)^\perp$
            (Prop.~\ref{prop:fiber4_base_and_curvature}),
            so only the constant-curvature spaces need to be checked.
            \begin{itemize}
                \item $M = \widetilde{\op{U}(2,1)/\op{U}(2)}$ is not $S^5$
                    since $M$ is contractible,
                    being a line bundle over the contractible space $\C \Hyp^2$.
                \item $M$ is not $\Euc^5$ since $\op{SU}(2,1)$
                    is semisimple and noncompact,
                    whereas every semisimple subgroup of $\isomplus \Euc^5$ is
                    compact as its semisimple part is $\op{SO}(5)$.
                \item $M$ is not $\Hyp^5$: a circle bundle $N$ over
                    a compact $\C \Hyp^2$ manifold has quotients of arbitrarily
                    small volume
                    (by the scalar action of $e^{2 \pi i / m}$ for large $m$);
                    whereas for fixed $n \geq 4$, there is a minimum
                    volume for hyperbolic $n$-manifolds by a
                    theorem of Wang \cite[Thm.~E.3.2]{bp}.
            \end{itemize}
    \end{itemize}
    Hence the maximal model geometries in Table \ref{table:fiber4_candidates}
    are the products other than $\Euc^5$ and the two non-product geometries
    $\Heis_5$ and $\widetilde{\op{U}(2,1)/\op{U}(2)}$.
\end{proof}

\begin{prop} \label{prop:fiber4_hyp_circle_bundle}
    For the fibering of $M = \op{U}(1,n)/\op{U}(n)$ over $\C \Hyp^n$,
    the curvature of the connection $(TM^{\op{U}(n)})^\perp$
    is nonzero.
\end{prop}
\begin{proof}
    Let $\{e_0,\ldots,e_n\}$ be the standard basis of $\C^{n+1}$.
    Embed $M$ in $\C^{n+1}$ as the preimage of $1$
    under the $(1,n)$ Hermitian form
        \[ \sum_{i=0}^n z_i e_i \mapsto |z_0|^2 - \sum_{i=1}^n |z_i|^2 , \]
    so that $M$ fibers as a circle bundle over $\C \Hyp^2$
    where the fibers are the intersections of $M$ with complex lines.

    Embed $U(n)$ in $U(1,n)$ as the copy of $U(n)$
    fixing $e_0 \in \C^{n+1}$,
    and let $V$ be the span of $\{e_1,\ldots,e_n\}$.
    Then $(TM^G)^\perp_{e_0} = V$;
    and we can find $(TM^G)^\perp$
    elsewhere by translating by elements of $U(1,n)$---explicitly,
    for $p \in M \subset \C^{n+1}$,
        \[ (TM^G)^\perp_p = p^\perp . \]

    If $v \in V$ has unit length, then there is a
    $1$-parameter subgroup of $U(1,n)$ defined by
    \begin{align*}
        e_0 &\mapsto (\cosh t) e_0 + (\sinh t) v \\
        v &\mapsto (\sinh t) e_0 + (\cosh t) v \\
        w &\mapsto w \text{ if } w \in V \text{ and } w \perp v .
    \end{align*}
    Fix a small $t > 0$ and define a path $\gamma(\theta)$ in $M$ by
        \[ \gamma(\theta)
            = (\cosh t) e^{i \theta (\tanh t)^2} e_0
            + (\sinh t) e^{i\theta} e_1 . \]
    Then
        \[ \gamma'(\theta)
            = i (\tanh t) \left(
                (\sinh t) e^{i\theta (\tanh t)^2} e_0
                + (\cosh t) e^{i\theta} e_1 \right) \]
    so $\gamma$ is horizontal.
    At $\theta = \frac{\pi}{1 - (\tanh t)^2}$,
    both $\gamma(\theta)$ and $\gamma(0)$ are in the same fiber,
    but $\gamma(0) \neq \gamma(\theta)$. So horizontal lifts
    of closed paths in $\C \Hyp^n$ are not necessarily closed.
\end{proof}


\section{Geometries fibering over 3D isotropy-irreducible geometries}
\label{chap:fiber3}

This section proves part (ii) of Theorem \ref{thm:main}---that is,
the $5$-dimensional maximal model geometries $M = G/G_p$ for which
the isotropy representation $G_p \curvearrowright T_p M$ contains
an irreducible $3$-dimensional summand are products.
Major ingredients of the proof include
some reasoning about fiber bundles (set up in Section \ref{chap:fiber_bundles}),
Lie group extension problems,
and some facts about transformation groups of spaces of constant curvature
(postponed to Section \ref{sec:fiber3_constk}).

While the classification seeks only model geometries,
it does encounter one geometry that satisfies the weaker condition
of having a unimodular isometry group.
This geometry is the homogeneous space
    \[ \R \rtimes \op{Conf}^+ \Euc^3 / \op{SO}(3) , \]
where the action on $\R$ is chosen to make $\R \rtimes \op{Conf}^+ \Euc^3$
unimodular. Its failure to be a model geometry is proven using
Galois theory; see Section \ref{sec:fiber3_solconf} for details.

The proof splits into four parts with the following preparation.
Since $M$ is a model geometry, $G$ must admit a lattice
(Defn.~\ref{defn:geometries3}),
which requires $G$ to be unimodular \cite[Prop.~1.1.3]{filipk}.
Under the assumption that
$G_p \curvearrowright T_p M$ has an irreducible $3$-dimensional summand,
$G_p$ contains a characteristic copy of $\op{SO}(3)$
(Fig.~\ref{fig:isotropy_poset}).
So the fibering $M \to B = M/\foliation{F}^{\op{SO}(3)}$
is conformal, with $B = S^3$, $\Euc^3$, or $\Hyp^3$
(Prop. \ref{prop:fibering_description}(ii)).
Hence to prove Theorem \ref{thm:main}(ii) it will suffice
to show the following.
\begin{prop} \label{prop:main:ii}
    Let $M = G/G_p$ be a $5$-dimensional maximal geometry
    where $G$ is unimodular and $G_p = \op{SO}(3)$ or $\op{SO}(3) \times \op{SO}(2)$;
    and let $B = M/\foliation{F}^{\op{SO}(3)}$.
    \begin{enumerate}[(i)]
        \item If $\pi_B: M \to B$ is an isometric fibering,
            then $M$ is a product of $2$-dimensional geometries
            and $3$-dimensional constant-curvature geometries.
        \item The products in (i) that are maximal model geometries
            are those for which both factors have constant curvature
            and at least one factor is not Euclidean.
        \item If $\pi_B$ is essential, then
            $M = \R \rtimes \op{Conf}^+ \Euc^3 / \op{SO}(3)$,
            where $A \in \op{Conf}^+ \Euc^3$ acts on $\R$ as dilation by $(\det A)^{-1}$.
        \item $\R \rtimes \op{Conf}^+ \Euc^3 / \op{SO}(3)$
            is not a model geometry.
    \end{enumerate}
\end{prop}

The rest of this section is devoted to proving Prop.~\ref{prop:main:ii}.

\subsection{Geometries fibering isometrically are products}

This section classifies the isometrically fibering geometries,
in two steps as delineated by the first two parts of Prop.~\ref{prop:main:ii}.
The first is to show that all isometrically fibering geometries
are products, and the second is to determine which products
are maximal model geometries. Standard facts about groups
acting on constant-curvature spaces will be stated where used,
with references either to the literature or to proofs deferred to
Section \ref{sec:fiber3_constk}.

\begin{proof}[Proof of Prop.~\ref{prop:main:ii}(i)]
    Let $M = G/(\op{SO}(3) \times \op{SO}(2))$ or $M = G/\op{SO}(3)$
    be a $5$-dimensional maximal geometry with $G$ unimodular,
    and suppose $\pi_B: M \to B = M/\foliation{F}^{\op{SO}(3)}$ is isometric.

    \paragraph{Step 1: $M \to B$ is a product bundle $F \times B$.}
    The $G_p$-invariant integrability tensor
    (Lemma \ref{bracket_infinitesimally})
        \[ \Lambda^2 \left((TM^{\op{SO}(3)})_p^\perp\right) \to T_p M^{\op{SO}(3)} \]
    is zero since the left side is the standard representation of $\op{SO}(3)$
    while the right side is a trivial representation; so
    $(TM^{\op{SO}(3)})^\perp$ is a flat connection on $M \to B$.
    Since $B$ is simply-connected
    and flat bundles are classified by the monodromy representation
    $\pi_1(B) \to \op{Diff} F$, 
    the fiber bundle $M \to B$ admits an isomorphism
    to a product bundle $F \times B$ taking $(TM^{\op{SO}(3)})^\perp$
    to the $3$-planes tangent to copies of $B$.
    This produces a second projection $\pi_F : M \to F$.\footnote{
        This alone does not make $M$ a product geometry,
        as demonstrated in dimension $3$ by the (non-model) geometry
        $\op{Conf}^+ \Euc^2 / \op{SO}(2)$;
        so there still remains something nontrivial to prove.
    }

    \paragraph{Step 2 (some general theory):
        It suffices to show that $F$ has a $G$-invariant metric.}
    If $G$-invariant metrics on $F$ exist, they correspond one-to-one with
    invariant inner products on $TM^{\op{SO}(3)}$, since both are determined
    by an inner product on a single $2$-plane.
    The same holds for $(TM^{\op{SO}(3)})^\perp$ and $B$.
    Then since $TM^{\op{SO}(3)}$ and $(TM^{\op{SO}(3)})^\perp$ are
    orthogonal in any invariant metric (Lemma~\ref{lemma:perp_orthogonality}),
    every invariant metric on $M$ is a direct sum of invariant inner products
    on $TM^{\op{SO}(3)}$ and $(TM^{\op{SO}(3)})^\perp$.
    So if both $\pi_B$ and $\pi_F$ are both isometric fiberings,
    then any invariant metric on $M$ is isometric to
    some invariant metric on $F \times B$.

    \paragraph{Step 3 (an extension problem):
        Describe point stabilizers $G_f$ of $G \curvearrowright F$.}
    Let $G_f \subseteq G$ be the subgroup fixing a point $f \in F$.
    Identifying $B$ with $\pi_F^{-1}(f)$,
    the image of $G_f$ in $\op{Diff} B$ is all of $\isomplus B$,
    since it contains $\op{SO}(3)$ in the point stabilizers and $G$
    acts transitively.
    If $G_p \cong \op{SO}(3)$, then counting dimensions shows that
    $G_f \cong \isomplus B$.

    Otherwise, the kernel of $G_f \to \isomplus B$
    is the $\op{SO}(2)$ factor in $G_p$;
    so passing to Lie algebras, $\tanalg G_f$ is an extension
        \[ 0 \to \lie{so}_2 \R \to \tanalg G_f \to \tanisom B \to 0 . \]
    A $1$-dimensional representation of $\tanisom B$ must be trivial
    for $B = \Euc^3$, $S^3$, or $\Hyp^3$ (Corollary \ref{quotients}(iii)),
    so $\tanisom B$ acts trivially on $\lie{so}_2 \R \cong \R$.
    Then since $H^2(\tanalg \isomplus B; \R) = 0$ (Lemma \ref{cohozero}),
    the classification of Lie algebra extensions by second cohomology
    (Thm.~\ref{thm:extensions_h2}) implies $G_f$ is covered by
    the product $\R \times \widetilde{\isomplus B}$.

    \paragraph{Step 4 (a study of transformation groups):
        $G_f$ preserves a metric on $F$.}
    The action of $G_f \subseteq G$ on $F$ defines a homomorphism
        \[ \phi: \widetilde{G_f} 
        \to \op{Aut} T_f F \cong \op{GL}(2,\R) . \]
    Since $\widetilde{\isomplus B}$ has no quotients of dimension $1$ or $2$
	(Corollary \ref{quotients}(i)), it acts
    with determinant $1$ on $\R^2$---and thus as a subgroup of $\op{SL}(2,\R)$
    of dimension $0$ or $3$. In fact $\widetilde{\isomplus B}$ does not admit
    $\op{SL}(2,\R)$ as a quotient, since:
	\begin{enumerate}
        \item proper quotients of $\widetilde{\isomplus \Euc^3}$
            factor through $\op{SU}(2)$
            (Lemma \ref{isomeuc_normalsubs});
        \item $\widetilde{\isomplus S^3} \cong \op{SU}(2) \times \op{SU}(2)$; and
        \item $\isomplus \Hyp^3 = \op{SO}(3,1)$ is simple.
	\end{enumerate}
    Therefore $\phi\left(\widetilde{\isomplus B}\right) = \{1\}$.
    If $G_p \cong \op{SO}(3)$, then this means $\phi$ has trivial image.
    If $G_p \cong \op{SO}(3) \times \op{SO}(2)$, then this implies
    $\phi$ factors through the $\op{SO}(2) \subset G_f$
    covered by the $\R$ factor
    in $\widetilde{G_f} \cong \R \times \widetilde{\isomplus B}$.
    Either way, the action of $G_f$ on $T_f F$ factors through a compact
    group, so it preserves an inner product on $T_f F$.
    Then since $G$ acts transitively, it preserves a metric on $F$,
    which finishes the proof by Step 2.
\end{proof}

Completing the classification of geometries fibering
isometrically over $3$-dimensional constant curvature geometries
requires determining which of the products $F \times B$ from (i)
are actually maximal model geometries.
So far, $B$ is known to be $\Euc^3$, $S^3$, or $\Hyp^3$,
while $F$ is just a geometry---a manifold with a smooth transitive group
action with compact point stabilizers.
The first step will be to determine all possibilities for $F$;
then general statements such as the maximality of most products
(Prop.~\ref{prop:products_are_maximal}) will finish the proof.

\begin{proof}[Proof of Prop.~\ref{prop:main:ii}(ii)
    (determination of maximal model geometries)]
    The $2$-dimensional geometries with $\op{SO}(2)$ isotropy
    have constant curvature; and those with trivial isotropy
    are the two simply-connected real Lie groups in dimension $2$,
    namely $\R^2$ and $\op{Aff}^+ \R$. Since $\R^2$ is a non-maximal $\Euc^2$
    and $\op{Aff}^+ \R$ is a non-maximal $\Hyp^2$,\footnote{
        Take $(x,y) \in \R \rtimes \R$ to the upper-half plane by
        $(x,y) \mapsto (x, e^y)$.
    } the only maximal $F \times B$ from part (i)
    are those where both factors are maximal model constant-curvature geometries.
    Since $\Euc^2 \times \Euc^3$ is a non-maximal $\Euc^5$,
    only those products with at most one Euclidean factor remain.

    Conversely, suppose $F \times B$ is a product of maximal model constant-curvature
    geometries of dimensions $2$ and $3$, at most one of which is Euclidean.
    Then $F \times B$ is a model geometry since products are models
    (Prop.~\ref{prop:products_are_models});
    and $F$, having $\op{SO}(2)$ isotropy, has no invariant vector fields,
    so $F \times B$ is maximal (Prop.~\ref{prop:products_are_maximal}).
\end{proof}

\subsection{Case study: The geometry fibering essentially conformally}
\label{sec:fiber3_solconf}

Let $M = G/G_p$ be a $5$-dimensional maximal geometry with $G$ unimodular,
and suppose the fibering $\pi_B: M \to B = M/\foliation{F}^{\op{SO}(3)}$ is
essentially conformal. This section proves
Prop.~\ref{prop:main:ii}(iii)--(iv): that
$M = \R \rtimes \op{Conf}^+ \Euc^3 / \op{SO}(3)$
and that $M$ is not a model geometry.

\begin{proof}[Proof of Prop.~\ref{prop:main:ii}(iii)]
    We will determine $M$ by solving the following
    extension problem to find $G$.
        \[ 0 \to \R \to \tanalg G \to \tanalg \op{Conf}^+ \Euc^3 \to 0 \]
	Such an extension is determined by a homomorphism
    $\phi: \tanalg \op{Conf}^+ \Euc^3 \to \op{Der} \R$
	(realized by lifting to $\lie{g}'$ and taking brackets,
    where $\op{Der} \R$ denotes the algebra of derivations of $\R$)
    and a class in $H^2(\tanalg \op{Conf}^+ \Euc^3; \R)$.

    \paragraph{Step 1: $G$ is an abelian extension of $\op{Conf}^+ \Euc^3$.}
    In an essential fibering, $B$ is $\Euc^3$ or $S^3$
    (since $\op{Conf}^+ \Hyp^3 = \isomplus \Hyp^3$ \cite[Thm.~A.4.1]{bp}).
    Since the only connected, transitive, essential subgroup of $\op{Conf}^+ B$
    containing $\op{SO}(3)$ in the point stabilizers is the entire group
    (Lemma \ref{essential_subgroups_of_conf}),
    the image of $G$ in $\op{Diff} B$ is either
	$\op{Conf}^+ \Euc^3 \cong \R^3 \rtimes (SO(3) \times \R)$ ($7$-dimensional)
	or $\op{Conf}^+ S^3 \cong \isomplus \Hyp^4$ ($10$-dimensional).
    Since $G_p$ is $\op{SO}(3)$ or $\op{SO}(3) \times \op{SO}(2)$,
    the dimension of $G$ is $8$ or $9$; so $B = \Euc^3$, and
	$G$ is an extension of $\op{Conf}^+ \Euc^3$ by a
	unimodular group $H$ of dimension $1$ or $2$.

    \paragraph{Step 2: $G_p$ is $\op{SO}(3)$.}
    As a $G_p$-representation, $\tanalg G \cong \tanalg G_p \oplus T_p M$.
    Since $\tanalg H \subset \tanalg G$ is an ideal, it is also a subrepresentation.
    Then letting $V$ denote the standard representation of
    $\op{SO}(3)$ and $\R$ the trivial representation,
        \[ \tanalg G
        \cong_{\op{SO}(3)} \tanalg \op{Conf}^+ \Euc^3 \oplus \tanalg H
        \cong_{\op{SO}(3)} 2V \oplus \R \oplus \tanalg H. \]
    The $\R$ is tangent to a group of dilations in $\op{Conf}^+ \Euc^3$;
    and the two copies of $V$ are $\lie{so}_3 \R \subset \tanalg G_p$
    and the $3$-dimensional subspace of $T_p M$ on which it acts.

    If $G_p = \op{SO}(3) \times \op{SO}(2)$, then $\R + \tanalg H$
    must consist of a $2$-dimensional subspace of $T_p M$ and the $\lie{so}_2 \R$
    acting on it. Then since $\tanalg H$ is an abelian ideal of $\tanalg G$,
    $\R + \tanalg H$ is a subalgebra of $\tanalg G$ isomorphic to
    $\tanalg \isomplus \Euc^2$ with the $\R$ corresponding to $\lie{so}_2 \R$.
    This cannot occur since this $\R$ is the tangent algebra to
    a group of dilations in $\op{Conf}^+ \Euc^3$,
    while $\lie{so}_2 \R$ is the tangent algebra to
    the compact group $\op{SO}(2) \subset G_p$.
    Hence $G_p = \op{SO}(3)$, and $\dim H = 1$; so $\tanalg G$ is an extension
        \[ 0 \to \R \to \tanalg G \to \tanalg \op{Conf}^+ \Euc^3 \to 0 . \]


    \paragraph{Step 3: Unimodularity determines the action of $\op{Conf}^+ \Euc^3$ on $H$.}
    The above extension problem induces, via Lie brackets,
    an action $\phi: \tanalg \op{Conf}^+ \Euc^3 \to \op{Der} \R$.
    Every ideal in
    $\tanalg \op{Conf}^+ \Euc^3 \cong \R^3 \semisum (\lie{so}_3 \oplus \R)$ either:
    \begin{itemize}
        \item is contained in $\R^3$, in which case it's $0$ or all of $\R^3$
            since $\lie{so}_3$ acts irreducibly; or
        \item contains some $v$ projecting nontrivially
            to $\lie{so}_3 \oplus \R$, in which
            case it still contains $\R^3$ since $[v, \R^3]$ is
            a nonzero subspace of $\R^3$.
    \end{itemize}
    Since $\lie{so}_3$ is simple, the ideals of $\tanalg \op{Conf}^+ \Euc^3$ are
    $0$, $\R^3$, $\R^3 \semisum \lie{so}_3$, $\R^3 \semisum \R$,
    and $\tanalg \op{Conf}^+ \Euc^3$, corresponding to the quotients
    $\tanalg \op{Conf}^+ \Euc^3$, $\lie{so}_3 \oplus \R$, $\R$, $\lie{so}_3$, and $0$.
    Of these, only $\R$ and $0$ can occur as images in $\op{Der} \R \cong \R$;
    so $\phi$ is either zero or of the form
        \[ \R^3 \semisum (\lie{so}_3 \R \oplus \R) \ni (v,r,s) \mapsto ks \]
    for some $k \in \R$.
    Since $\op{ad} (0,0,1)$ is diagonal with three $1$s and four $0$s,
    unimodularity of $G$ requires $k = -3$.
    Then $G$ is an extension
        \[ 1 \to \R \to G \to \op{Conf}^+ \Euc^3 \to 1 \]
    where $A \in \op{Conf}^+ \Euc^3$ acts on $\R$ as dilation by $(\det A)^{-1}$.
    ($H \cong \R$ since $S^1$ admits no dilations.)

    \paragraph{Step 4: The extension splits---i.e.\ $G$ is the semidirect product.}
    A spanning set for $\tanalg \op{Conf}^+ \Euc^3$
    is given by translations $t_i \in \R^3$ ($1 \leq i \leq 3$),
    rotations $r_{ij} \in \lie{so}_3$ ($[r_{ij}, t_i] = t_j = [-r_{ji}, t_i]$),
    and a scaling $s \in \R$.
    Of these, only $s$ acts nontrivially; so given a $2$-cocycle $c$,
    we can subtract a coboundary to make
    its restriction to $\R^3 \semisum \lie{so}_3$ zero using
    Lemma \ref{cohozero}.

    Then when we apply the cocycle condition to
    $c(t_i, s) = c(r_{ij}, t_j], s)$
    and $c(r_{ij}, s) = c([r_{ik}, r_{kj}], s)$,
    we obtain a sum of terms of the following forms
    where the blanks are in $\R^3 \semisum \lie{so}_3$.
    \begin{itemize}
        \item $sc(*,*)$, which is zero since $c|_{\R^3 \semisum \lie{so}_3} = 0$;
        \item $*c(*,s)$, which is zero since $\R^3 \semisum \lie{so}_3$
            acts trivially; and
        \item $c(*, [*, s])$, which is zero since $[s, \tanalg \op{Conf}^+ \Euc^3] = \R^3$.
    \end{itemize}
    Thus $c = 0$, so $H^2 = 0$.

    By the classification of abelian extensions by second cohomology
    (Thm.~\ref{thm:extensions_h2}),
    the Lie algebra extension splits. Then $G$ is covered by
    $\R \rtimes \widetilde{\op{Conf}^+ \Euc^3}$,
    whose center
    (consisting of the elements lying over the identity in
    $\op{SO}(3) \subset \R \rtimes \op{Conf}^+ \Euc^3$)
    has order $2$.
    So $G$ is either $\R \rtimes \widetilde{\op{Conf}^+ \Euc^3}$
    or $\R \rtimes \op{Conf}^+ \Euc^3$; these deformation
    retract to their maximal compact subgroups $\op{SU}(2)$ and $\op{SO}(3)$,
    respectively.
    Since $G$ must contain $G_p \cong \op{SO}(3)$
    and all maximal compact subgroups are conjugate, 
    the geometry $M$ is $\R \rtimes \op{Conf}^+ \Euc^3 / \op{SO}(3)$.
\end{proof}

To show that
$\R \rtimes \op{Conf}^+ \Euc^3 / \op{SO}(3)$ is not a model geometry,
it suffices to show that $\R \rtimes \op{Conf}^+ \Euc^3$ admits no lattice.

\begin{proof}[Proof of Prop.~\ref{prop:main:ii}(iv)
    ($\R \rtimes \op{Conf}^+ \Euc^3 / \op{SO}(3)$ is not a model geometry)]
    Suppose there were a lattice
    \[ \Gamma \subset G = \R \rtimes \op{Conf}^+ \Euc^3
    \cong (\R \times \R^3) \rtimes (\op{SO}(3) \times \R) . \]

    \paragraph{Step 1: By general theory, produce an integer matrix $A$.}
    Recall the following facts.
    \begin{enumerate}
        \item For a connected Lie group in which every
            compact semisimple subgroup acts nontrivially on the solvable radical,
            the intersection of the nilradical with a lattice
            is a lattice in the nilradical \cite[Lemma 3.9]{mostow}.
        \item For a closed normal subgroup $H$ of $G$,
            the group $\Gamma/(\Gamma \cap H)$ is a lattice in $G/H$
            if and only if $\Gamma \cap H$ is a lattice in $H$
            \cite[Thm.~I.1.4.7]{onishchik2}.
    \end{enumerate}
    The nilradical $N$ of $G$ is $\R \times \R^3$;
    so $\Gamma \cap N \cong \Z^4$, and
    $\Gamma / (\Gamma \cap N)$ is a lattice in $G/N \cong \op{SO}(3) \times \R$.
    Then some $g \in \Gamma$ projects nontrivially to the $\R$ factor
    (the dilation part) in $\op{SO}(3) \times \R$.
    This $g$ acts by conjugation on $\Gamma \cap N$
    as an integer matrix $A$ acting on $\Z^4$.

    \paragraph{Step 2: $A$ has a real eigenvalue $\lambda \neq \pm 1$.}
    Since $G$ is unimodular, $\det A = 1$. Then the action
    of $\op{SO}(3) \times \R$ on $\R \times \R^3$ requires the eigenvalues to be
    of the form $\lambda$, $\lambda e^{i\theta}$, $\lambda e^{-i\theta}$,
    and $\lambda^{-3}$ where $\lambda \in \R$.
    Since Step 1 selected $A$ to act with nontrivial dilation,
    $\lambda$ is not $1$ or $-1$.

    \paragraph{Step 3: Using Galois theory, conclude $\lambda = \pm 1$.}
    Since $\lambda$ is real and not $\pm 1$ (and not $0$ as $A$ is invertible),
    the magnitudes of $\lambda$, $\lambda^{-3}$, and $\lambda^9$ are all distinct;
    so $A$ and $A^{-3}$ share only the eigenvalue $\lambda^{-3}$.
    Since each eigenvalue of an integer matrix must occur with all of its
    Galois conjugates, $\lambda^{-3} \in \Q$.
    Applying the rational root theorem to the characteristic polynomial
    of $A$ implies $\lambda^{-3} = \pm 1$.
    Then $\lambda = \pm 1$, which contradicts the conclusion
    of Step 2.
\end{proof}

\subsection{Groups acting on constant-curvature spaces}
\label{sec:fiber3_constk}

This subsection collects some facts used above,
concerning groups acting on spaces of constant curvature.
First, the classification required some results
about low-dimensional representations of $\isomplus \Euc^n$.
Their proof begins with the following observation, which will
be reused later to classify some extensions of $\tanisom \Euc^2$.
\begin{lemma} \label{isomeuc_normalsubs}
    Every nonzero ideal of $\tanisom \Euc^n \cong \R^n \semisum \lie{so}_n$
    contains the translation subalgebra $\R^n$.
\end{lemma}
\begin{proof}
    An ideal $\lie{n}$ of $\tanisom \Euc^n$ containing no nonzero elements of $\R^n$
    acts trivially on $\R^n$---since $\R^n$ is also an ideal,
        \[ [\lie{n}, \R^n] \subseteq \lie{n} \cap \R^n = 0. \]
    Since $\R^n$ is a faithful representation of $\lie{so}_n$,
    any ideal containing no nonzero elements of $\R^n$ is zero.

    Since $\R^n$ is an irreducible representation of $\lie{so}_n$,
    any ideal that does contain some nonzero $v \in \R^n$
    also contains $[\lie{so}_n, v] \cong \R^n$.
\end{proof}
Since $\op{SO}(k)$ ($k \neq 4$) is simple, Lemma \ref{isomeuc_normalsubs} above
implies that for $n \neq 4$, the connected normal subgroups
of $\isomplus \Euc^n$ are $\{1\}$, $\R^n$, and $\isomplus \Euc^n$.
Since both $\isomplus S^n \cong \op{SO}(n+1)$
and $\isomplus \Hyp^n \cong \op{SO}(n,1)$ are simple,
statments about quotients and representations of the isometry groups
can be made for all of $\Euc^n$, $S^n$, and $\Hyp^n$ at once, as follows.
\begin{cor} \label{quotients} Let $G = \widetilde{\isomplus M}$
    where $M$ is $\Euc^n$, $S^n$, or $\Hyp^n$.
    \begin{enumerate}[(i)]
        \item If $n \geq 3$, then $G$ has no quotient groups of dimension $1$ or $2$.
        \item If $n > 4$, then $G$ has no
            nontrivial quotient groups of dimension less than $\binom{n}{2}$.
        \item All $1$-dimensional real representations of $\tanalg G$ are trivial.
    \end{enumerate}
\end{cor}

The remainder of this section's facts are
only used in the above classification of geometries
fibering over $\Euc^3$, $S^3$, and $\Hyp^3$---starting
with the following classification of sufficiently large
groups acting by conformal automorphisms on $S^k$ and $\Euc^k$.

\begin{lemma} \label{essential_subgroups_of_conf}
	Let $B = \Euc^k$ or $S^k$ for some $k \geq 2$.
	The only connected, transitive, essential subgroup $H$ of $\op{Conf}^+ B$
	with a copy of $SO(k)$ in its point stabilizer $H_p$ is $\op{Conf}^+ B$.
\end{lemma}
\begin{proof}
    The two cases are handled separately but with broadly the same strategy:
    the goal is to show that $H_p$ is the entirety of a point stabilizer of
    $\op{Conf}^+ B$.

    \paragraph{Case 1: $B = \Euc^k$.}
    A point stabilizer of
    $\op{Conf}^+ \Euc^k \cong \R^k \rtimes (\op{SO}(k) \times \R)$ \cite[Thm.~A.3.7]{bp}
    is $\op{SO}(k) \times \R$.
    Since $H$ acts essentially, $H_p \subset \op{SO}(k) \times \R$
    projects nontrivially to the $\R$ factor.
    The homotopy exact sequence for $H_p \to H \to B$,
    along with the assumption that $B$ is simply-connected
    and $H$ is connected, implies $H_p$ is connected;
    so $H_p = \op{SO}(k) \times \R$.
    Then $H$ has the same point stabilizers as $\op{Conf}^+ \Euc^k$,
    so $H = \op{Conf}^+ \Euc^k$.

    \paragraph{Case 2, preparatory claim:
        No transitive subgroup $H \subset \op{Conf}^+ S^k$
        acts on $S^k$ with
        point stabilizers $H_p \cong \op{SO}(k) \times \R$.}
    A point stabilizer of $\op{Conf}^+ S^k \curvearrowright S^k$
    is $\op{Conf}^+ \Euc^k$ \cite[Cor.~A.3.8]{bp}.
    Up to conjugacy, $H_p$ is the standard
    $\op{SO}(k) \times \R \subset \op{Conf}^+ \Euc^k$,
    since $\op{SO}(k)$ is maximal compact and $\R$ is its centralizer.
    This $\op{SO}(k) \times \R$ fixes \emph{two} points on $S^k$,
    whose stabilizers in $H$ coincide
    since point stabilizers of a transitive action are isomorphic.
    Then $H$ preserves a pairing of points in $S^k$.
    In particular, if $\op{SO}(k) \times \R$ preserves $p$ and $q$
    then it acts transitively on $S^k \smallsetminus \{p,q\}$
    while preserving this pairing.
    So paired points
    \begin{enumerate}
        \item lie in the same $S^{k-1}$ in
            $S^k \smallsetminus \{p,q\} \cong S^{k-1} \times \R$
            since they are exchanged by an order-$2$ element; and
        \item are antipodal in this $S^{k-1}$
            since the $\op{SO}(k-1)$ fixing one member of the pair
            must fix the other.
    \end{enumerate}
    Interpret $S^k$ as the boundary at infinity of $\Hyp^{k+1}$,
    following \cite[Prop.~A.5.13(4)]{bp}.
    Any two geodesics in $\Hyp^{k+1}$ joining paired points of $S^k$
    must intersect, since $H$ acts transitively
    and they all intersect the geodesic joining $p$ and $q$.
    However, only geodesics whose endpoints lie in the same
    $S^{k-1} \subset S^k \smallsetminus \{p,q\}$ can intersect
    since each $S^{k-1}$ bounds a totally geodesic $\Hyp^k \subset \Hyp^{k+1}$.
    Therefore the pairing of points required by an action
    with $\op{SO}(k) \times \R$ stabilizers cannot be defined on all of $S^k$.

    \paragraph{Case 2: $B = S^k$.}
    Since $H$ acts essentially, it preserves no Riemannian metric;
    so some point stabilizer $H_p$ preserves no inner product on the tangent space at $p$.
    So the quotient map $\pi: \op{Conf}^+ \Euc^k \to \op{SO}(k) \times \R$
    is surjective when restricted to $H_p$.
    Since $H_p \ncong \op{SO}(k) \times \R$, it cannot also
    be injective.
    Then $H_p$ meets the translation subgroup
    $\op{ker} \pi \cong \R^k \subset \op{Conf}^+ \Euc^k$ nontrivially;
    so it contains all of $\R^k$
    since $\op{SO}(k) \curvearrowright \R^k$ is irreducible.
    Then $H_p = \op{Conf}^+ \Euc^k$,
    which implies as in Case 1 that $H = \op{Conf}^+ S^k$.
\end{proof}

Finally, the classification of geometries fibering over $3$-dimensional
isotropy-irreducible geometries required the following computation of
second cohomology.
\begin{lemma} \label{cohozero}
    $H^2(\tanisom M; \R) = 0$
    for $M = \Euc^3$, $S^3$, or $\Hyp^3$.
\end{lemma}
\begin{proof}
    This is a computation using a spanning set, though for $S^3$ and $\Hyp^3$
    one could instead appeal to the vanishing of $H^2$ for
    semisimple algebras \cite[Thm.~1.3.2]{onishchik3}.
    The spanning elements of $\tanisom M$ will be denoted
    $r_{ij}$ and $t_i$ ($1 \leq i, j \leq n = \dim M$);
    the linear dependency relations are $r_{ij} = -r_{ji}$,
    and the nonzero brackets are
    \begin{alignat*}{2}
        [r_{ij}, r_{jk}] &= r_{ki}  &\quad& \text{if } i,j,k \text{ distinct} \\
        [r_{ij}, t_i] &= t_j        && \text{if } i \neq j \\
        [t_i, t_j] &= Kr_{ij}       && \text{if } i \neq j ,
    \end{alignat*}
    where $K$ is $0$ or $\pm 1$ (the sectional curvature of $M$).

    The action of $\tanisom M$ on $\R$ is trivial (Corollary \ref{quotients}(iii));
    so the cocycle condition for a $2$-cocycle $c: \Lambda^2 \tanisom M \to \R$ becomes
        \[ c([x_1,x_2],x_3) + c([x_2,x_3],x_1) + c([x_3,x_1],x_2) = 0 , \]
    and the coboundaries are of the form
        \[ c(x_1,x_2) = f([x_1,x_2]) , \quad f \in (\tanisom M)^* . \]

    Suppose $c$ is a $2$-cocycle, and $i$, $j$, and $k$ are distinct.
    Applying the cocycle condition to the spanning elements yields
    \begin{align*}
        c(r_{ij},t_k)
            &= c([r_{jk}, r_{ki}], t_k) \\
            &= c(r_{jk},t_i) + c(r_{ki},t_j) \\
        c(r_{ij},t_i)
            &= c([r_{jk}, r_{ki}], t_i) \\
            &= c(r_{kj},t_k) \\
        Kc(r_{jk},r_{ki})
            &= c(Kr_{jk}, r_{ki}) \\
            &= c([t_j,t_k], r_{ki}) \\
            &= c(t_i,t_j) .
    \end{align*}
    A linear combination of cyclic permutations of the first equality is
    $c(r_{ij},t_k) = 3c(r_{ij},t_k)$, so $c(r_{ij}, t_k) = 0$.
    The second equality implies that $c(r_{ij}, t_i)$ only depends on $i$.
    Since $n = 3$, distinctness of $i$, $j$, and $k$
    implies $c(r_{jk}, r_{ki})$ only depends on $i$ and $j$,
    which permits defining\footnote{
        Independence of $c(r_{jk}, r_{ki})$ from $k$
        can be proven for $n > 3$
        by computing $c([r_{k\ell}, r_{\ell j}], r_{ki})$
        and using the $n = 3$, $K = 1$ case to show that
        $c(r_{ij}, r_{k\ell}) = 0$
        for distinct $i$, $j$, $k$, $\ell$.
        This extends Lemma \ref{cohozero} to dimensions other than $3$.
    }
    \begin{align*}
        f: \tanisom M &\to \R \\
            r_{ij} &\mapsto c(r_{jk}, r_{ki}) \\
            t_i &\mapsto c(r_{ji}, t_j) .
    \end{align*}
    This definition and the third equality
    imply $c(x,y) = f([x,y])$ for all $x$ and $y$
    in the spanning set, and therefore on all of $\tanisom M$.
    So every cocycle $c$ is a coboundary, i.e.\ $H^2(\tanisom M; \R) = 0$.
\end{proof}

\section{Geometries fibering over 2D geometries}
\label{chap:fiber2}

This section carries out the (unfortunately long)
task of proving part (iii) of Theorem \ref{thm:main}.
That is, it classifies the $5$-dimensional maximal model geometries
$M = G/G_p$ in case (iii) of the fibering description
(Prop.~\ref{prop:fibering_description})---those for which the
irreducible subrepresentations of $G_p \curvearrowright T_p M$
have dimensions $1$ and $2$.
The first step is to set up an extension problem that can be solved to find $G$.
\begin{lemma}[\textbf{The isometry group as an extension}]
    \label{lemma:fiber2_extension_formulation}
	If $M = G/G_p$ is a $5$-dimensional model geometry
    for which $G_p \curvearrowright T_p M$ decomposes into $1$-dimensional
    and $2$-dimensional summands,
	then $G$ is an extension
		\[ 1 \to H \to G \to Q \to 1 \]
	where:
    \begin{enumerate}[(i)]
		\item $Q$ is $\op{Conf}^+ S^2$, $\op{Conf}^+ \Euc^2$,
			$\isomplus S^2$, $\isomplus \Euc^2$, or $\isomplus \Hyp^2$;
		\item if $G_p \cong S^1$, then the identity component of
			$H$ is covered by $\R$, $\R^2$,
			$S^3$, $\widetilde{\op{SL}(2,\R)}$,
			$\op{Sol}^3$, $\widetilde{\isomplus \Euc^2}$,
			$\Heis_3$, or $\R^3$; and
        \item if $G_p = \op{SO}(2) \times \op{SO}(2)$, then
            the identity component of $H$ is either:
			\begin{itemize}
				\item one of $SO(3)$, $\op{PSL}(2,\R)$, and $\isomplus \Euc^2$; or
				\item covered by
					$S^3 \times \R$, $\widetilde{\op{SL}(2,\R)} \times \R$,
					$\widetilde{\isomplus \Euc^2} \times \R$, or
					$\widetilde{\isomplus \Heis_3}$.
			\end{itemize}
	\end{enumerate}
\end{lemma}

To organize the solution of what could be fifty extension problems,
the classification proceeds by skimming off classes of geometries
until only those with $\tilde{G} = \tilde{H} \times \tilde{Q}$ remain;
these either are product geometries or can be described as associated
bundles (Section \ref{sec:assoc_bundles}).
The plan is illustrated in Figure \ref{fig:fiber2_flowchart}.
\begin{figure}[h!]
    \caption{Classification strategy for geometries fibering over 2-D spaces.}
    \label{fig:fiber2_flowchart}
    \label{fig_thesis:fiber2_flowchart}
    \begin{center}
    \begin{tikzpicture}[%
        >=triangle 60,              
        start chain=going below,    
        node distance=6mm and 60mm, 
        scale=0.8,
        ]
    \tikzset{
      base/.style={draw, on chain, on grid, align=center, minimum height=4ex},
      proc/.style={base, rectangle, text width=12em},
      test/.style={base, diamond, aspect=2, text width=5em},
      term/.style={proc, rounded corners},
      coord/.style={coordinate, on chain, on grid, node distance=6mm and 25mm},
      nmark/.style={draw, cyan, circle, font={\sffamily\bfseries}},
      it/.style={font={\small\itshape}}
    }
        \node[test] (tess) {$M \to B$ essential?};
        \node[test] (tnil) {Extension of $\op{Isom} \Euc^2$ by $\R^3$?};
        \node[test] (tlevi) {Levi action nontrivial?};
        \node[term] (gprod) {Products and associated bundles (\S\ref{sec:fiber2_product_groups}).};

        \node[term, right=of tess] (gess)
            {$T^1 \Hyp^3$ and some non-nilpotent solvable Lie groups (\S\ref{sec:fiber2_essentials}).};
        \node[term, right=of tnil] (gnil) {Some nilpotent Lie groups (\S\ref{sec:fiber2_nilmanifolds}).};
        \node[term, right=of tlevi] (glevi) {$T^1 \Euc^{1,2}$ and the line bundles over $\mathbb{F}^4$ (\S\ref{sec:fiber2_semidirect}).};

        \path (tess.east) to node [near start, yshift=1em] {yes} (gess);
            \draw[*->] (tess.east) -- (gess);
        \path (tnil.east) to node [near start, yshift=1em] {yes} (gnil);
            \draw[*->] (tnil.east) -- (gnil);
        \path (tlevi.east) to node [near start, yshift=1em] {yes} (glevi);
            \draw[*->] (tlevi.east) -- (glevi);

        \path (tess.south) to node [anchor=west, xshift=0.5em] {no---isometric fibering} (tnil);
            \draw[o->] (tess.south) -- (tnil);
        \path (tnil.south) to node [anchor=west, xshift=0.5em] {no---then $\tilde{G}$ is a split extension} (tlevi);
            \draw[o->] (tnil.south) -- (tlevi);
        \path (tlevi.south) to node [anchor=west, xshift=0.5em] {no---$\tilde{G}$ is a direct product} (gprod);
            \draw[o->] (tlevi.south) -- (gprod);
    \end{tikzpicture}
    \end{center}
\end{figure}
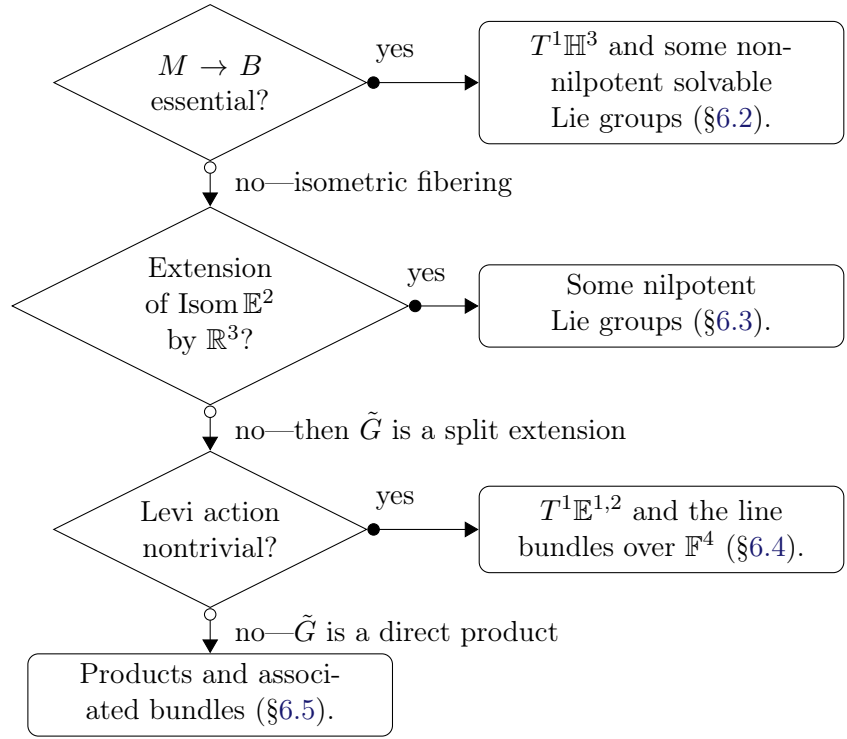

\subsection{Setting up the extension problem}

This section proves Lemma \ref{lemma:fiber2_extension_formulation},
which describes the extension problem that will be solved to determine
the transformation groups $G$ of the geometries $G/G_p$.
Having obtained a fibering over a 2-dimensional space
in Prop.~\ref{prop:fibering_description},
the bulk of the proof is in establishing the lists of
quotients $Q$ and kernels $H$.
One lemma is needed, in the form of
the following observation about geometries with
abelian isotropy---which will also be useful later in recovering
a faithfully-acting $G$ from its universal cover $\tilde{G}$.

\begin{lemma} \label{lemma:faithful_conjugation}
    If $M = G/G_p$ be a connected homogeneous space
    where $G$ is connected and $G_p$ is compact and abelian,
    then the following are equivalent.
    \begin{enumerate}[(i)]
        \item $G$ acts faithfully on $M$ (one of the requirements for a geometry)
        \item $G_p$ acts faithfully on $T_p M$
        \item $G_p$ acts faithfully by conjugation on $G$
    \end{enumerate}
\end{lemma}
\begin{proof}
    There are two equivalences to verify.

    \paragraph{(i) $\iff$ (ii):}
    Since $G_p$ is compact, $M$ has an invariant Riemannian metric
    \cite[Prop.~3.4.11]{thurstonbook}.
    An isometry of a connected Riemannian manifold
    is determined by its value and derivative at a point \cite[Prop.~A.2.1]{bp},
    so the action of $G_p$ on $M$
    is determined by the action of $G_p$ on $T_p M$.
    Then a nontrivial $g \in G$ acts as the identity on $M$
    if and only if it lies in $G_p$ and acts as the identity on $T_p M$;
    so $G$ acts faithfully on $M$ if and only if $G_p$ acts faithfully on $T_p M$.

    \paragraph{(ii) $\iff$ (iii):}
    As a $G_p$-representation, the Lie algebra of $G$ decomposes as
        \[ \tanalg G = \tanalg G_p \oplus T_p M . \]
    Since $G_p$ is abelian, $\tanalg G_p$ is trivial; so
    $\tanalg G$ under the adjoint action
    $T_p M$ is a faithful $G_p$-representation
    if and only if $\tanalg G$ under the adjoint action is too.
    The equivalence for the conjugation action on $G$ follows
    since a homomorphism from a connected Lie group
    is determined by its derivative at the identity.
\end{proof}

\begin{rmk} \label{rmk:quotient_by_center}
    Condition (iii) is equivalent to $G_p \cap Z(G) = \{1\}$.
    So an abelian-isotropy homogeneous space $G/G_p$
    satisfying all the conditions for a geometry except
    faithfulness of the $G$-action
    can be made into a geometry by passing from $G$ to $G/(G_p \cap Z(G))$.
\end{rmk}

\begin{proof}[Proof of Lemma \ref{lemma:fiber2_extension_formulation}
    (the extension problem for $G$)]
	Suppose $M = G/G_p$ is a $5$-dimensional model geometry
    for which $G_p \curvearrowright T_p M$ decomposes into $1$-dimensional
    and $2$-dimensional summands. Then there is a conformal fibering
    (Prop.~\ref{prop:fibering_description}(iii)) $M \to B$
    where $B$ is $S^2$, $\Euc^2$, or $\Hyp^2$.
    If $Q$ denotes the image of $G$ in $\op{Conf} B$,
    then $G$ is an extension
        \[ 1 \to H \to G \to Q \to 1 . \]

    \paragraph{Step 1: The image $Q$ is $\isomplus B$ or $\op{Conf}^+ B$.}
    If $G_p = \op{SO}(2)$, then $B = M/\foliation{F}^{\op{SO}(2)}$
    so $Q$ contains a copy of $\op{SO}(2)$.
    Otherwise, the trivial subrepresentation of $G_p \curvearrowright T_p M$
    is $1$-dimensional, so $G_p$ acts nontrivially on $B$.
    In either case, $Q$ acts transitively on $B$
    and contains a copy of $\op{SO}(2)$ fixing a point.
    
    Then $Q = \isomplus B$ if $B$ admits a $G$-invariant metric;
    otherwise, $Q$ is a connected, transitive, essential subgroup of
    $\op{Conf}^+ B$ and must therefore be all of $\op{Conf}^+ B$
    (Lemma \ref{essential_subgroups_of_conf}).
    The list in the original statement (Lemma \ref{lemma:fiber2_extension_formulation})
    omits $\op{Conf}^+ \Hyp^2$ since
    $\op{Conf}^+ \Hyp^2 = \isomplus \Hyp^2$ \cite[Thm.~A.4.1]{bp}.

    \paragraph{Step 2: If $G_p \cong S^1$, then $H$ is unimodular
        of dimension at most $3$.}
	If $G_p \cong S^1$, then $\dim G = 6$.
	Since $\dim Q \geq 3$ by the previous step, $\dim H \leq 3$.
    Since $G$ is unimodular \cite[Prop.~1.1.3]{filipk} and
	$H$ is a closed normal subgroup, $H$ is itself unimodular
    \cite[Prop.~1.1.4]{filipk}.
    The explicit list in Lemma \ref{lemma:fiber2_extension_formulation}(ii)
    can be found by consulting Bianchi's classification of all $3$-dimensional
    real Lie algebras; see e.g.\ \cite[Lec.~10]{fultonharris},
    \cite[Table I]{patera}, or \cite[Table 21.3]{maccallum}.\footnote{
        Alternatively, to carry out Bianchi's classification from scratch,
        notice that the only two real semisimple Lie algebras
        of dimension at most $3$ are $\lie{so}_3 \R$ and $\lie{sl}_2 \R$.
        The other $3$-dimensional real Lie algebras are therefore solvable.
        A nilradical in a solvable Lie algebra is at least half the dimension
        \cite[Thm.~2.5.2, attributed to Mubarakzyanov]{onishchik3},
        so these algebras are of the form $\R^2 \rtimes \R$,
        which can be systematically handled using Jordan forms.
    }

    \paragraph{Step 3: If $G_p = \op{SO}(2) \times \op{SO}(2)$,
        then some $\op{SO}(2) \subset H$
        acts faithfully by conjugation.}
	For each $Q$ in Step 1,
	the maximal compact subgroup has rank $1$;
	so $H \cap G_p$ is a closed subgroup of dimension at least $1$.
    Since $G$ acts with $\op{SO}(2)$ stabilizers on $B$,
	$H \cap G_p$ has dimension at most $1$.
    Therefore the identity component of $H \cap G_p$ is isomorphic to $\op{SO}(2)$.

	Since $G_p$ is a point stabilizer of a geometry,
    it acts faithfully on $\tanalg G$ (Lemma \ref{lemma:faithful_conjugation}).
    Since $H \cap G_p \subseteq H$,
    it acts trivially on $\tanalg Q \cong_{G_p} \tanalg G / \tanalg H$;
	so its action on $\tanalg H$---and therefore $H$---must be faithful.

    \paragraph{Step 4: Classify possible $H$ when $G_p = \op{SO}(2) \times \op{SO}(2)$.}
    First, $\dim H \leq 4$ since
    $\dim Q \geq 3$ and $\dim G = \dim M + \dim (\op{SO}(2) \times \op{SO}(2)) = 5 + 2 = 7$.
	If $\dim H \leq 3$, applying the restriction from Step 3
    to the list from Step 2 yields the $3$-dimensional groups $H$
    in Lemma \ref{lemma:fiber2_extension_formulation}(iii).

	The remainder of the groups occur if $\dim H = 4$.
    Then $H/(H \cap G_p)$ is a $3$-dimensional homogeneous
	space, with $H \cap G_p = SO(2)$ point stabilizers by Step 3
	and unimodular isometry group $H$.
	The proof of \cite[Thm.~3.8.4(b)]{thurstonbook}
	(classifying $3$-dimensional geometries with $SO(2)$ point stabilizer),
    finds the spaces listed in Table \ref{table:homogeneous_3d}.
    To obtain the final list in Lemma \ref{lemma:fiber2_extension_formulation}(iii),
    eliminiate duplicates by observing that
	semidirect products with inner action are isogenous with direct products.
    \begin{table}[h!]
        \caption[3-dimensional model geometries with $\op{SO}(2)$ isotropy.]{3-dimensional simply-connected homogeneous spaces $H/\op{SO}(2)$
            with unimodular $H$}
        \label{table:homogeneous_3d}
        \begin{center}\begin{tabular}{cc}
            \rule[-6pt]{0pt}{0pt}
            $H/(H \cap G_p)$  &  $H$ \\
            \hline
            \rule[-7pt]{0pt}{21pt}
            $S^3$  &
                $S^3 \rtimes S^1$ \\
            \rule[-9pt]{0pt}{0pt}
            $\widetilde{\op{PSL}(2,\R)}$  &
                $\widetilde{\op{PSL}(2, \R)} \rtimes \op{SO}(2)$  \\
            \rule[-9pt]{0pt}{0pt}
            $\Heis_3$  &
                $\Heis_3 \rtimes \op{SO}(2)$  \\
            \rule[-9pt]{0pt}{0pt}
            $S^2 \times \R$  &
                $\op{SO}(3) \times \R$  \\
            \rule[-9pt]{0pt}{0pt}
            $\Hyp^2 \times \R$  &
                $\op{PSL}(2, \R) \times \R$  \\
            \rule[-7pt]{0pt}{0pt}
            $\widetilde{\isomplus \Euc^2}$  &
                $\widetilde{\isomplus \Euc^2} \rtimes \op{SO}(2)$ \\
        \end{tabular}\end{center}
    \end{table}
\end{proof}

\begin{rmk}
	The list of $4$-dimensional groups in Step 4
	may also be obtained by computing with a convenient basis in the Lie algebra:
	let $r$ generate $H \cap G_p$, let $x$ and $y$ span the nontrivial
	$(H \cap G_p)$-subrepresentation of $\tanalg H$,
	and let $z$ (together with $r$) span the trivial subrepresentation.
	Then one can work out the values of $[x,y]$, $[x,z]$, and $[y,z]$
	that satisfy the Jacobi identity (and rescale the basis if it helps).
\end{rmk}

\subsection{Geometries fibering essentially}
\label{sec:fiber2_essentials}

This section handles the first side branch of Figure \ref{fig:fiber2_flowchart}:
proving the following classification of geometries that fiber essentially
over $2$-dimensional spaces.

\begin{prop}
    \label{prop:fiber2_essentials}
	Suppose $M = G/G_p$ is a $5$-dimensional model geometry
    for which $G_p \curvearrowright T_p M$ decomposes into $1$-dimensional
    and $2$-dimensional summands. Furthermore suppose that the fibering
    $M \to B$ from Prop.~\ref{prop:fibering_description}(iii) is essential.
    \begin{enumerate}[(i)]
        \item If $B = S^2$, then $M = T^1 \Hyp^3 \cong \op{PSL}(2,\C)/\op{PSO}(2)$,
            with isotropy $G_p = S^1_1$
            in the notation of Figure \ref{fig:isotropy_poset}.
        \item $T^1 \Hyp^3$ is a maximal model geometry.
        \item If $B = \Euc^2$, then $M$ is a solvable Lie group of the form
            $\R^4 \rtimes_{e^{tA}} \R$.
            Moreover, $M$ is maximal if and only if the multiset
            of characteristic polynomials
            of the Jordan blocks of $A$ is one of the following.
            \begin{enumerate}
                \item $\{x-1$, $x-1$, $x+1$, $x+1\}$
                \item $\{(x-1)^2$, $x+1$, $x+1\}$
                \item $\{x-1$, $x-1$, $x$, $x+2\}$ (This is $\Sol^4_0 \times \Euc$.)
                \item $\{x-1$, $x-1$, $x - a + 1$, $x + a + 1\}$;\; $a > 0$, $a \neq 1$,
                    $a \neq 2$
                    (This is a family of geometries.)
            \end{enumerate}
        \item The geometries in (iii).(a)--(c) are model geometries;
            and (iii).(d) is a model geometry if and only if $e^{tA}$
            has a characteristic polynomial in $\Z[x]$ for some $t > 0$.
    \end{enumerate}
\end{prop}
\begin{rmk}
    In (iii), the proof of maximality will reveal that
    the isotropy consists of an $\op{SO}(2)$ acting by automorphisms on $M$
    for every pair of identical Jordan blocks in $A$.
\end{rmk}

The case when $B = S^2$ is considerably easier than that of $B = \Euc^2$,
which has some complexities not immediately apparent in the above statement;
we treat these cases separately.

\subsubsection{Over the sphere}

This section contains the proof that $T^1 \Hyp^3$
is a maximal model geometry (ii), and that it is the only one
fibering essentially conformally over the sphere (i).
The extension problem from Lemma \ref{lemma:fiber2_extension_formulation}
is used only to determine that the isometry group covers
$\op{PSL}(2,\C)$; the rest of the proof is built on properties of $\op{PSL}(2,\C)$.

\begin{proof}[Proof of Prop.~\ref{prop:fiber2_essentials}(i)]
    If $M = G/G_p$ fibers essentially over $B = S^2$, then $G$
    is an extension
        \[ 1 \to H \to G \to \op{Conf}^+ S^2 \to 1 \]
    where $H$ is as named in Lemma \ref{lemma:fiber2_extension_formulation}.
    Since $\dim \op{Conf}^+ S^2 = 6$, either $G_p \cong S^1$ and $\dim H = 0$
    or $G_p = \op{SO}(2)^2$ and $\dim H = 1$.
    The list of possible $H$ for $\op{SO}(2)^2$ includes no $1$-dimensional
    entries; so $\dim H = 0$, $G_p \cong S^1$, and $G$ covers $Q$.

    As maximal tori in maximal compact subgroups,
    all copies of $S^1$ in $\op{Conf}^+ S^2 \cong \op{PSL}(2,\C)$ are conjugate,
    as are all copies of $S^1$ in the $2$-sheeted universal cover $\op{SL}(2,\C)$.
    Hence $M \cong \op{SL}(2,\C)/\op{SO}(2) \cong \op{PSL}(2,\C)/\op{PSO}(2)$.
    Since $\op{PSL}(2,\C)$ is centerless, the geometry $M$ expressed with
    faithful transformation group (Rmk.~\ref{rmk:quotient_by_center})
    is indeed $\op{PSL}(2,\C)/\op{PSO}(2)$.

    Since $\op{PSL}(2,\C) \cong \isomplus \Hyp^3$,
    choosing a basepoint in $T^1 \Hyp^3$
    identifies $\op{PSL}(2,\C)/\op{PSO}(2)$ with $T^1 \Hyp^3$ (hence the name).
    The point stabilizers have slope $1$ because $\op{PSO}(2)$
	rotates $\Hyp^3$ and a tangent space the same way
	(or one can explicitly decompose $\lie{sl}_2 \C$
    into $S^1$-representations).
\end{proof}

\begin{proof}[Proof of Prop.~\ref{prop:fiber2_essentials}(ii)]
	$T^1 \Hyp^3$ is a model geometry since it models
	the unit tangent bundle of any finite-volume hyperbolic $3$-manifold.

	For maximality, suppose $G$ were a larger connected group of isometries
	for $T^1 \Hyp^3$ under some metric.

    If $\dim G = 7$, then $G_p = \op{SO}(2)^2$.
	From the classification of simple Lie groups
    \cite[Ch. X, \S{}6 (p. 516)]{helgasonnew},
    the only connected semisimple Lie group containing $\op{PSL}(2,\C)$
    of dimension at most $7$ is $\op{PSL}(2,\C)$.
    Then using the Levi decomposition \cite[\S 1.4]{onishchik3}
    and the fact that $\op{PSL}(2,\C)$ is centerless,
    $G$ admits $\op{PSL}(2,\C)$ as a quotient.
    Since $\op{PSO}(2)$ is a maximal torus in $\op{PSL}(2,\C)$, it contains
    the image of $G_p$.
    Then $G/G_p$ fibers essentially over
    $S^2 \cong \op{Conf}^+ S^2 / \op{Conf}^+ \Euc^2$,
    which by part (i) is incompatible with $\dim G = 7$.

	If $\dim G > 7$, then $\dim G_p > 2$, and previous sections already listed
	maximal geometries with $\dim G_p > 2$. Of those, only $S^2 \times \Euc^3$
	and $S^2 \times \Hyp^3$ have the same diffeomorphism type.
    Since $\op{PSL}(2,\C)$ admits no nontrivial image in $\op{SO}(3) = \isomplus S^2$
    (both are simple and the domain has larger dimension),
	it cannot act transitively by isometries on either of these products.
\end{proof}

\subsubsection{Over the plane}

Suppose $M = G/G_p$ fibers essentially over $\Euc^2$.
The description of $G$ as an extension (Lemma \ref{lemma:fiber2_extension_formulation})
is
    \[ 1 \to H \to G \to \op{Conf}^+ \Euc^2 \to 1 \]
where, since $\op{Conf}^+ \Euc^2$ is $4$-dimensional,
$H$ is $\R^2$, $SO(3)$, $\op{PSL}(2,\R)$ or $\isomplus \Euc^2$.
This section classifies the resulting geometries
(Prop.~\ref{prop:fiber2_essentials}(iii)--(iv)),
all of which will be solvable Lie groups of the form $\R^4 \rtimes_{e^{tA}} \R$.
The geometry named in Prop.~\ref{prop:fiber2_essentials}(iii).(a)
occurs when $H = \isomplus \Euc^2$; the rest will come from $H = \R^2$.

Passing to Lie algebras, we aim to solve the
corresponding extension problem
	\[ 0 \to \lie{h} \to \lie{g} \to \lie{q} \to 0 . \]

The proof relies on the following three computations:
the outer derivation algebra of $\lie{h}$ (Lemma \ref{lemma:fiber2_essential_out}),
the action $\lie{q} \to \op{out} \lie{h}$ (Lemma \ref{lemma:fiber2_essential_action}),
and the cohomology $H^2(\lie{q}; \R^2)$ with the action
in Lemma \ref{lemma:fiber2_essential_action} (Lemma \ref{lemma:fiber2_essential_coho}).
    For the Lie algebra $\lie{q}$ of $\op{Conf}^+ \Euc^2$,
	we will use the basis $\{x, y, r, s\}$
	where $x$ and $y$ generate the translations,
	$r$ generates rotations ($[r,x] = y$),
	and $s$ generates scaling ($[s,x] = x$).

\begin{lemma}[\textbf{Outer derivation algebras}] \label{lemma:fiber2_essential_out}
	The Lie algebras of the above groups $H$ have the following
	algebras of outer derivations:
	\begin{align*}
		\op{out} \lie{so}_3 &= 0  &
		\op{out} \R^2 &= \lie{gl}_2 \R \\
		\op{out} \lie{sl}_2 &= 0  &
		\op{out} \left(\tanisom \Euc^2\right) &= \R
	\end{align*}
\end{lemma}
\begin{proof}
	$\R^2$ is abelian,
    and the outer derivation algebra of a semisimple Lie algebra
    is zero \cite[\S 5.3]{humphreys}, so only the last one needs
	any computation.

	Let $\{x, y, r\}$ be a basis for $\tanisom \Euc^2$
	where $[r,x] = y$ and $x$ and $y$ generate the translations.
    For a derivation $d$,
    the Leibniz rule $d[v,w] = [dv, w] + [v, dw]$
    implies that $d$ preserves the lower central series and
	the derived series---so
	in this case it takes translations to translations.
	Subtract inner derivations to ensure
	$dr = ar$ and $dx = bx$ for some $a$ and $b$; then
	\begin{align*}
		dy &= [dr, x] + [r, dx] = ay + by \\
		bx &= dx = [dy, r] + [y, dr] = (2a + b)x .
	\end{align*}
	The second line implies $a = 0$; then $d$ is zero on $r$
	and scales by $b$ on the translations.
\end{proof}

\begin{lemma}[\textbf{Restrictions on actions on $H$}]
    \label{lemma:fiber2_essential_action}
    In the action of $\lie{q} = \tanalg \op{Conf}^+ \Euc^2$ on $\lie{h}$,
	$s$ acts with trace $-2$, $x$ and $y$ act trivially,
	and $r$ generates a compact subgroup of $\op{Out} \lie{h}$.

    In particular, if $H = \R^2$,
    then up to conjugacy in $\op{GL}(2,\R) = \op{Aut} \lie{h}$,
    $s$ acts by one of the following
    matrices where $a$ is a real parameter.
    \begin{align*}
        &\begin{pmatrix}
            -1 & a \\ -a & -1
        \end{pmatrix} &
        &\begin{pmatrix}
            -1 & 1 \\ 0 & -1
        \end{pmatrix} &
        &\begin{pmatrix}
            -1+a & 0 \\ 0 & -1-a
        \end{pmatrix}
    \end{align*}
    And if $r$ acts nontrivially by rotations (or some conjugate thereof),
    then only the first of these can commute with $r$.
\end{lemma}
\begin{proof}
    There are three claims to prove in the first sentence of the Lemma statement:
    \begin{enumerate}
        \item Since $s$ acts with trace $2$ on $\lie{q}$
            and $\lie{g}$ is unimodular, $s$ must act with trace $-2$ on $\lie{h}$.
        \item
            The restriction of $\lie{q} \to \op{out} \lie{h}$
            to $\tanisom \Euc^2$ is either injective or
            contains the translation ideal (Lemma \ref{isomeuc_normalsubs}).
            Of the outer derivation algebras (Lemma \ref{lemma:fiber2_essential_out}),
            only $\lie{gl}_2 \cong \lie{sl}_2 \oplus \R$
            has high enough dimension to
            contain an injective image of $\tanisom \Euc^2$.
            For the kernel not to contain the translation ideal,
            the projection to at least one of the summands must be
            injective; but $\R$ has too low dimension,
            and $\lie{sl}_2$ contains no subalgebras of dimension
            $3$ other than itself.
            Hence $\tanisom \Euc^2 \to \op{out} \lie{h}$
            always factors through $\lie{so}_2$---i.e.\ $x$ and $y$ act trivially.
        \item Since the Lie algebra extension
            is induced by a Lie group extension, 
            the homomorphism $\lie{q} \to \op{out} \lie{h}$
            is induced by some $\op{Conf}^+ \Euc^2 = Q \to \op{Out} H$, which sends
            the compact $\op{SO}(2) \subset \op{Conf}^+ \Euc^2$ generated by $r$
            to a compact subgroup of $\op{Out} H \subseteq \op{Out} \lie{h}$.
    \end{enumerate}
    The claims when $H = \R^2$ then follow from listing $2 \times 2$
    Jordan forms with trace $-2$ and the fact that the centralizer
    of $\op{SO}(2) \curvearrowright \R^2$ is generated by itself and
    the real scalars.
\end{proof}

Since $\R^2$ is abelian, $\op{out} \R^2 = \op{der} \R^2$;
so the only additional data needed for $H = \R^2$
is second cohomology.

\begin{lemma} \label{lemma:fiber2_essential_coho}
	The Lie algebra cohomology $H^2(\tanalg \op{Conf}^+ \Euc^2; \R^2)$
	has the following values.
	\begin{itemize}
		\item If $s$ acts with an eigenvalue $2$, then $H^2$ is $1$-dimensional,
			represented by cocycles $c$ with $c(x,y)$ in the $2$-eigenspace.
		\item If $s$ acts with an eigenvalue $0$, then $H^2$ is $1$-dimensional,
			represented by cocycles $c$ with $c(r,s)$ in the $0$-eigenspace.
		\item Otherwise, $H^2 = 0$.
	\end{itemize}
\end{lemma}
\begin{proof}
	If $\beta$ is a $1$-cochain, then the coboundary $d\beta$ has values
	\begin{align*}
		d\beta(x,y) &= x\beta(y) - y\beta(x) - \beta([x,y]) = 0 \\
		d\beta(r,x) &= r\beta(x) - x\beta(r) - \beta([r,x]) = r\beta(x) - \beta(y) \\
		d\beta(r,y) &= r\beta(y) - y\beta(r) - \beta([r,y]) = r\beta(y) + \beta(x) \\
		d\beta(s,x) &= s\beta(x) - x\beta(s) - \beta([s,x]) = (s - 1)\beta(x) \\
		d\beta(s,y) &= s\beta(y) - y\beta(s) - \beta([s,y]) = (s - 1)\beta(y) \\
		d\beta(r,s) &= r\beta(s) - s\beta(r) - \beta([s,r]) = r\beta(s) - s\beta(r) .
	\end{align*}
	If $c$ is a cocycle, then
	\begin{align*}
		c(x,y) &= -c(x, [y, s]) \\
			&= c(y, [s, x]) + c(s, [x, y]) + xc(y,s) + yc(s,x) + sc(x,y) \\
			&= c(y, x) + 0 + 0 + 0 + sc(x,y) \\
		2c(x,y) &= sc(x,y) .
	\end{align*}
	Thus either $c(x,y) = 0$ or $s$ acts with an eigenvalue $2$.
	Similarly, applying the cocycle condition to each of the equalities
	\begin{align*}
		c(x,s) &= -c(s, [y, r]) \\
		c(y,s) &= -c(s, [r, x]) \\
		c(x,y) &= -c(x, [x, r]) \\
		c(x,y) &= -c(y, [y, r]) \\
		c(r,x) &= -c(r, [x, s]) \\
		c(r,y) &= -c(r, [y, s])
	\end{align*}
	yields mostly vacuous equalities, except for the following.
	\begin{align}
		c(x,s) &= -rc(y,s) + (1-s) c(r,y) \notag\\
		c(y,s) &= rc(x,s) + (s-1) c(r,x)
		\label{conf_euc_cocycle_condition}
	\end{align}
	Using this information,
	we will define $\beta$ to match $d\beta$ as closely as possible with $c$.
	\begin{itemize}
		\item If $r$ acts as $0$, then set
			$\beta(x) = c(r,y)$ and $\beta(y) = -c(r,x)$.
			Then $c - d\beta$ is zero on $r \wedge x$, $r \wedge y$,
			$x \wedge s$, and $y \wedge s$
			(by equations \ref{conf_euc_cocycle_condition}).
			\begin{itemize}
				\item If $s$ acts with eigenvalues $0$ and $-2$,
					then $c(x,y) = 0$. Set $\beta(r) = \frac{1}{2} c(r,s)$,
					so that the only potentially nonzero value of $c - d\beta$
					is on $r \wedge s$ and lies in the $0$-eigenspace of $s$.
				\item If $s$ acts with eigenvalues $2$ and $-4$,
					setting $\beta(r) = -s^{-1} c(r,s)$ makes $c - d\beta$ zero
					on $r \wedge s$. The only nonzero contribution to $H^2$
					is from $c(x,y)$ lying in the $2$-eigenspace of $s$.
				\item Otherwise, setting $\beta(r) = -s^{-1} c(r,s)$
					makes $c - d\beta = 0$, so $H^2 = 0$.
			\end{itemize}
		\item If $r$ acts by rotation, then since $s$ commutes with it,
			we reinterpret $\R^2$ as $\C$ on which $r$ acts by $ik$
			for some real $k \neq 0$. Defining
			\begin{align*}
				z_s &= c(x,s) + ic(y,s) \\
				z_r &= c(r,x) + ic(r,y) \\
				w_s &= c(x,s) - ic(y,s) \\
				w_r &= c(r,x) - ic(r,y) ,
			\end{align*}
			equations \ref{conf_euc_cocycle_condition} become
			\begin{align*}
				z_s &= -kz_s + i(s-1)z_r \\
				w_s &= kw_s - i(s-1)w_r .
			\end{align*}
			Since $s \neq 1$ (by having to act with trace $-2$),
			$\beta(x)$ and $\beta(y)$ can be selected to make $d\beta$
			reproduce $z_s$ and $w_s$; then $z_r$ and $w_r$ follow
			dependently.

			We further set $\beta(r) = 0$ and $\beta(s) = r^{-1} c(r,s)$
			to make $(c - d\beta)(r,s) = 0$. Finally, $c(x,y) = 0$
			since $s$ has no eigenvalue $2$ (both its eigenvalues have
			real part $1$).
	\end{itemize}
\end{proof}

\begin{proof}[Proof of Prop.~\ref{prop:fiber2_essentials}(iii)]
    The proof begins by determining the possibilities for $H$.

    \paragraph{Step 1: $H$ is $\isomplus \Euc^2$ or $\R^2$.}
    The outer derivation algebras for the Lie algebras of $\op{SO}(3)$
    and $\op{PSL}(2,\R)$ are zero
    \cite[Cor.\ to Thm.~1.3.2]{onishchik3}; and being unimodular, they have no inner
    derivations acting with trace $-2$. Since $\tanalg \op{Conf}^+ \Euc^2$
    contains an element acting on $\lie{h}$ with trace $-2$
    (Lemma \ref{lemma:fiber2_essential_action}),
    this rules out $\op{SO}(3)$ and $\op{PSL}(2,\R)$ as candidates for $H$.
    Then $G$ is an extension of the solvable group $\op{Conf}^+ \Euc^2$
    by either $\R^2$ or $\isomplus \Euc^2$. 

    \paragraph{Step 2: If $H = \isomplus \Euc^2$, then
    $M = \left[\left(\R^4 \rtimes \R \right) \rtimes \op{SO}(2)^{\times 2}\right] / \op{SO}(2)^{\times 2}$.}
	Since $Z(\tanisom \Euc^2) = 0$, the Lie algebra cohomology
	determining extensions is identically $0$;
	so every homomorphism
		\[ \tanalg \op{Conf}^+ \Euc^2 \to \op{out} \tanisom \Euc^2 \cong \R \]
	can be realized by an extension, and every such extension splits
    on the Lie algebra level (Thm.~\ref{thm:extensions_h2}).
    Since $r$ has to generate a compact subgroup
	and the nonzero part of $\op{out} \tanalg \isomplus \Euc^2$ scales
    the translation subalgebra, $r$ maps to $0$ in $\op{out} \tanisom \Euc^2$.
    Since $s$ acts with trace $-2$, it maps to the scalar $-1$.

    Then
        \( \tilde{G} \cong \C^2 \rtimes \R^3 , \)
    where $(x,y,z) \in \R^3$ acts on $\C^2$ by the matrix
        \[ \begin{pmatrix} e^{x+iy} & 0 \\ 0 & e^{-x+iz} \end{pmatrix} ; \]
    and $Z(\tilde{G}) = \{0\} \rtimes (0 \oplus 2 \pi \Z \oplus 2 \pi \Z)$.

    The point stabilizer $G_p$ is compact and acts faithfully by conjugation on $G$
    (Lemma \ref{lemma:faithful_conjugation}),
    so it maps injectively to a compact subgroup of
    $\op{Inn} G = \tilde{G}/Z(\tilde{G})
        \cong \C^2 \rtimes (R \times \op{SO}(2) \times \op{SO}(2))$,
    in which $\op{SO}(2) \times \op{SO}(2)$ is maximal compact since
    the quotient by it is contractible.
    Then $G_p$ covers this $\op{SO}(2) \times \op{SO}(2)$,
    but not nontrivially due to the faithful conjugation action requirement.
    So by conjugacy of maximal compact subgroups,
        \[ M = \left[\left(\R^4 \rtimes \R \right) \rtimes \op{SO}(2)^{\times 2}\right] / \op{SO}(2)^{\times 2} , \]
    the geometry named in Prop.~\ref{prop:fiber2_essentials}(iii)(a).

    \paragraph{Step 3:
	Non-split extensions of $\op{Conf}^+ \Euc^2$ by $\R^2$
    produce no model geometries.}
	The non-split extensions are with $r$ acting trivially,
	and with $s$ acting diagonally with an eigenvalue of $0$ or $2$.
    We'll show that if $0$ is an eigenvalue of $s$, then $\op{SO}(2)$
	fails to extend to a compact point stabilizer;
	and if $2$ is an eigenvalue of $s$, then $G$
	fails to admit lattices.

	When $s$ acts with an eigenvalue of $0$, a non-split extension
	is given by a nonzero value of $c(r,s)$ in the $0$-eigenspace;
	i.e. $[r,s]$ is nonzero in $\R^2 \subset \lie{g}$
    (Lemma \ref{lemma:fiber2_essential_coho}).
    Then $r$---indeed, any element of $\lie{g}$ lying over
    $r \in \tanalg \op{Conf}^+ \Euc^2$---acts nontrivially and nilpotently
    on a subalgebra of $\lie{g}$ and therefore cannot generate a compact
    subgroup (which must act semisimply).
    The same holds for an element of the $2$-eigenspace of $s$,
    while an element of the $0$-eigenspace generates a central subgroup of $G$;
    so no nontrivial compact subgroup of $G$ lies over
    $\op{SO}(2) \subset \op{Conf}^+ \Euc^2$.
    The point stabilizer $G_p$, being compact, must be such a subgroup;
    so the group $G$ resulting from this case produces no geometries $G/G_p$.

	When $s$ acts with an eigenvalue of $2$, a non-split
	extension is given by a nonzero value of $c(x,y)$
	in the $2$-eigenspace. So in $\lie{g}$,
	the elements $x$ and $y$ generate a copy of the
	Heisenberg algebra. Inspecting the actions of the
	other basis elements shows that this is an ideal,
	and in fact $\tilde{G}$ can be written as
		\[ (\Heis_3 \times \R) \rtimes (\R \times \R), \]
	where the first $\R$ rotates the $xy$-plane of the Heisenberg group
	and the second $\R$ scales diagonally with exponents
	$1$, $1$, $2$, and $-4$.

	Since $\R \times \R$ acts non-nilpotently, the nilradical
	of this is $\Heis_3 \times \R$. Since $G$ is solvable,
	a theorem of Mostow (see \cite[Prop 6.4.2]{filipk}) ensures
	that any lattice $\Gamma$ in its universal cover $\tilde{G}$
	intersects the nilradical in a lattice and projects
	to the quotient as a closed cocompact group.
	Then since $[\Heis_3, \Heis_3] = Z(\Heis_3)$
	and $Z(\Heis_3)$ is a copy of $\R$, cocompactness of $\Gamma$ ensures
	$[\Gamma, \Gamma] \cap Z(\Heis_3)$ is nontrivial.
	Since $\Gamma$ has cocompact image in $\R \times \R$,
	some $g \in \Gamma$ acts by nontrivial scaling on $\Heis_3 \times \R$;
	in particular, this action on $Z(\Heis_3)$
	implies $[\Gamma, \Gamma] \cap Z(\Heis_3)$ is not discrete;
	so $\Gamma$ fails to be a lattice.

    \paragraph{Step 4: Split extensions $\tilde{G} = \R^2 \rtimes \op{Conf}^+ \Euc^2$ where $s$
        acts by a complex scalar produce non-maximal geometries.}
	In this case, we assume $s$ acts as
	$\begin{pmatrix} -1 & a \\ -a & -1 \end{pmatrix}$ for some real $a$.
    Then $\tilde{G} \cong \C^2 \rtimes \R^2$
    where, for some real $b$, the action of $(x,y) \in \R^2$ on $\C^2$ is by the matrix
        \[ \begin{pmatrix} e^{-x+iax+iby} & 0 \\ 0 & e^{x+iy} \end{pmatrix} . \]
	To have a compact subgroup we can use as the point stabilizer,
	$0 \oplus \R$ must intersect the center nontrivially; so
	$G$ is covered by
		\[ ( \C^2 \rtimes \R ) \rtimes S^1 \]
	where the $S^1$ action on each $\C$ may have some degree other than $1$.
	Since this has trivial center, this actually is $G$ (and $G_p$
	is its maximal compact subgroup $S^1$).
	Whatever geometry it produces is subsumed by the geometry
        \[ [ ( \C^2 \rtimes \R ) \rtimes \op{SO}(2)^{\times 2} ] / \op{SO}(2)^{\times 2} \]
    from Step 2.

    \paragraph{Step 5: Identify the geometries that remain.}
    Similarly to Step 4,
	the remaining groups $G$ are all described as the semidirect product
		\[ ((\R^4) \rtimes \R) \rtimes S^1 \]
	(again, to obtain $G$ from the universal cover,
	we replaced a second $\R$ factor by $S^1$ and noted that
	the result has trivial center),
	where $t \in \R$ acts by one of the matrices below (omitted entries are zero)
    and $S^1$ acts as $\op{SO}(2)$ on the last two coordinates.
	\begin{align*}
		&\begin{pmatrix}
			e^t & e^t & &\\
			& e^t & &\\
			& & e^{-t} & \\
			& & & e^{-t} \\
		\end{pmatrix} &
		&\begin{pmatrix}
			e^{(1 + a)t} & & & \\
			& e^{(1 - a)t} & & \\
			& & e^{-t} & \\
			& & & e^{-t} \\
		\end{pmatrix}
	\end{align*}
	Step 4 eliminates $a = 0$; and since $S^1 \subset G$ is maximal
    compact, $M \cong \R^4 \rtimes \R$ where $t \in \R$ acts by one of the above matrices.
    The case $a = 1$ is $\op{Sol}^4_0 \times \Euc$;
    and the case $a = 2$ is a non-maximal form of
    $\R \rtimes \op{Conf}^+ \Euc^3 / \op{SO}(3)$
    (Section \ref{sec:fiber3_solconf}).

    \paragraph{Step 6: Maximality.}
    The maximal geometry realizing a solvable Lie group $M$ with real roots
    is of the form $M \rtimes K / K$, where $K \subseteq \op{Aut} M$
    is maximal compact (Lemma \ref{lemma:solvable_maximality}).

    So for these $M = \R^4 \rtimes_{e^{tA}} \R$, it suffices to determine the maximal
    compact subgroup of $\op{Aut} M \cong \op{Aut} \tanalg M$.
    An automorphism of the Lie algebra $\tanalg M$ must preserve its nilradical $\R^4$,
    each of the generalized eigenspaces
    by which $\R \cong \tanalg M / \R^4$ acts on $\R^4$,
    and the filtration by rank on each generalized eigenspace.
    So in coordinates the matrices in $\op{Aut} \tanalg M$
    are block upper-triangular, and the maximal compact subgroup
    $K \subset \op{Aut} \tanalg M$ is conjugate into a group
    whose only nonzero entries are in the diagonal blocks
    (Part II, \cite[Lemma \ref{ii:lemma:compacts_in_aut}]{geng2}).
    Since $K$ matches the dimension of the isotropy groups
    computed above in the classification, we conclude that the
    maximal geometries are
    \begin{align*}
        \SemiR{\R^4}{(x-1)^2,\, x+1,\, x+1} &\rtimes \op{SO}(2) / \op{SO}(2) \\
        \SemiR{\R^4}{x-1,\, x-1,\, x+1,\, x+1} &\rtimes \op{SO}(2)^2 / \op{SO}(2)^2 \\
        \SemiR{\R^4}{x-1,\, x-1,\, x-a+1,\, x+a+1} &\rtimes \op{SO}(2) / \op{SO}(2) .
            \qedhere
    \end{align*}
\end{proof}

\begin{proof}[Proof of Prop.~\ref{prop:fiber2_essentials}(iv) (model geometries)]
    Let $M = G/G_p \cong \R^4 \rtimes_{e^{tA}} \R$ be a geometry named in
    Prop.~\ref{prop:fiber2_essentials}.
    The proof splits into cases
    according to characteristic polynomials of Jordan blocks of $A$.

    \paragraph{Case (a): $x-1$, $x-1$, $x+1$, $x+1$.}
	This geometry models the solvmanifold $\C^2 \rtimes \R / (\Lambda \rtimes \Z)$
	where $1 \in \Z$ acts by a matrix conjugate to
	$\begin{pmatrix} 2 & 1 \\ 1 & 1 \end{pmatrix}$
	and $\Lambda \cong \Z^4$ is preserved by this action.

    \paragraph{Case (b): $(x-1)^2$, $x+1$, $x+1$.}
	When $A$ is the matrix below, \cite[Corollary 6.4.3]{filipk} ensures the
	existence of a lattice, if we choose $t$ to be the logarithm of an
	invertible quadratic integer, e.g. $\ln (3 + 2 \sqrt{2})$.
		\[ \begin{pmatrix}
			e^t & e^t & &\\
			& e^t & &\\
			& & e^{-t} & \\
			& & & e^{-t} \\
        \end{pmatrix} \]

    \paragraph{Cases (c, d): $x-1-a$, $x-1+a$, $x+1$, $x+1$.}
	We need a slightly stronger version of the
	same result; backtracking to Mostow's theorem \cite[Prop 6.4.2]{filipk},
	any lattice in $\tilde{G}$ intersects the nilradical $\R^4$
	in a lattice and projects to $\R \times \R$ as a closed cocompact subgroup
	(hence a lattice). So if $G$ admits a lattice $\Gamma$, then
	it lifts to a lattice in $\tilde{G}$
	containing the elements of $\R = \tilde{S^1}$
	lying over the identity in $S^1$; furthermore $\Gamma \cap (\R^2 \times \C)$
	is a lattice in $\R^2 \times \C$ preserved by two independent elements
	of $\R \times \R$. One of these can be chosen to lie over the identity in $S^1$
	so the other just needs to be independent from the $S^1$ direction.

	Hence $G$ of the second form admits a lattice if and only if, for some
	real $\theta$ and $t \neq 0$, the matrix with diagonal entries
	$e^{(1+a)t}$, $e^{(1-a)t}$, $e^{-t + i\theta}$, and $e^{-t - i\theta}$
	has a characteristic polynomial with integer coefficients---i.e.
	these four numbers are roots of an integer polynomial $p \in \Z[x]$.
	Then the extension $L$ of $\Q$ containing these roots is Galois.
	Let $\lambda = e^t$ and $z = e^{i\theta}$,
	so that these roots can be written as $\lambda^{1+a}$,
	$\lambda^{1-a}$, $\lambda^{-1} z$, and $\lambda^{-1} \overline{z}$.

	If $z$ is real, then $\lambda$ appears twice but the other roots appear
	only once, so $\lambda$ has no Galois conjugates---i.e. $\lambda \in \Q$.
	By Gauss's lemma, $x - \lambda \in \Z[x]$ so $\lambda \in \Z$.
	then $\lambda = 1$, which contradicts $t \neq 0$.
	Hence $z$ is not real.

	If $\lambda^{1-a} \in \Q$, then by the rational root theorem,
	$\lambda^{1-a} = 1$; so $a = 1$ since $\lambda \neq 1$.
	Then the resulting geometry is $\op{Sol}^4_0 \times \Euc$
	(in Filipkiewicz's notation, $G_5 \times \Euc$).

	If $\lambda^{1-a}$
	has degree $2$ over $\Q$, then its Galois conjugate must
	be the other real root of $p$, so
	$p$ factors---again in $\Z[x]$ due to Gauss's lemma---as
		\[
			\left( x^2 - (\lambda^{1+a} + \lambda^{1-a})x + \lambda^2 \right)
			\left( x^2 - (2 \lambda^{-1} \cos \theta)x + \lambda^{-2} \right) ,
		\]
	which again implies that $\lambda = 1$.
	If $\lambda^{1-a}$ has degree $3$ over $\Q$, then
	$\lambda^{1+a}$ has degree $1$ over $\Q$, again yielding $\op{Sol}^4_0 \times \Euc$.

	The last case is when $\lambda^{1-a}$ has degree $4$ over $\Q$. Then
		\[ p(x) = x^4 + ax^3 + bx^2 + cx + 1 \]
	is irreducible with exactly two real roots $\phi_1 = \lambda^{1-a}$
	and $\phi_2 = \lambda^{1+a}$, both positive.
	From the values of $\phi_1$ and $\phi_2$, we can recover $a$ by
		\[ a = \left|\frac{\ln \phi_1 - \ln \phi_2}{\ln (\phi_1 \phi_2)}\right| . \]
	The condition on $p$ having exactly two real roots is that its
	discriminant is negative.


    One can use Sturm's theorem or computer assistance to derive the condition
    that the two real roots are positive when $a + c < 0$.
\end{proof}

\subsection{Geometries requiring non-split extensions}
\label{sec:fiber2_nilmanifolds}

Having handled the essentially-fibering geometries in the previous section,
the remaining geometries $M = G/G_p$ fiber isometrically.
So in the description of $G$ as an extension $1 \to H \to G \to Q \to 1$
(Lemma \ref{lemma:fiber2_extension_formulation}),
$Q = \isomplus B$ where $B = S^2$, $\Euc^2$, or $\Hyp^2$---thus
reducing the number of extension problems from fifty to thirty.

This section's contribution is the second branch in
Figure \ref{fig:fiber2_flowchart}---that is, it proves
the following two claims.

\begin{lemma} \label{lemma:fiber2_nonsplit_oneext}
    If $M = G/G_p$ is a maximal model geometry for which
    \begin{enumerate}[(i)]
        \item $G_p \curvearrowright T_p M$ decomposes into $1$-dimensional
            and $2$-dimensional summands, and
        \item $\tilde{G}$ is not a split extension of
            $\widetilde{\isomplus B}$ ($B = S^2$, $\Euc^2$, or $\Hyp^2$),
    \end{enumerate}
    then $G$ is an extension
        \[ 1 \to \R^3 \to G \to \isomplus \Euc^2 \to 1 . \]
\end{lemma}
\begin{prop}
    \label{prop:fiber2_nilmanifolds}
    The $5$-dimensional
    maximal model geometries $M = G/G_p$ where $\tanalg G$ is an extension
        \[ 0 \to \R^3 \to \tanalg G \to \tanisom \Euc^2 \to 0 \]
    are the nilpotent Lie groups
    \begin{enumerate}[(i)]
		\item $M = \R^4 \rtimes_A \R$
			with $A = \begin{pmatrix}
				1 & 1 & & \\ & 1 & & \\ & & 1 & 1 \\ & & & 1 \\
			\end{pmatrix}$
		\item $M$ has Lie algebra with basis $\{x,y,u,v,w\}$ and
			\begin{align*}
				[x,w] &= u  &  [y,w] &= v  &  [x,y] &= w .
			\end{align*}
	\end{enumerate}
	On both of these, the point stabilizer is an $S^1$'s worth of
    Lie group automorphisms at the identity,
    acting as the diagonal circle in $\op{SO}(2) \times \op{SO}(2)$;
    and the full isometry group is $M \rtimes S^1$.
\end{prop}

The proofs will set up the tools for solving all thirty extension problems---i.e.\ for
classifying extensions of $Q$ by $H$ for all combinations of the three
remaining values of $Q$ and the ten values of $H$
(See Lemma \ref{lemma:fiber2_extension_formulation}).
The strategy is as follows.
Passing to Lie algebras, a (non-split) extension
    \[ 0 \to \lie{h} \to \lie{g} \to \lie{q} \to 0 \]
is given by the homomorphism $\lie{q} \to \op{out} \lie{h}$
and a (nonzero) class in $H^2(\lie{q}; Z(\lie{h}))$
(Thm.~\ref{thm:extensions_h2}).
Since each $Q$ is unimodular and a model geometry's transformation group
is unimodular \cite[Prop.~1.1.3]{filipk}, the action on $\lie{h}$
must be by traceless derivations
(Part II, \cite[Lemma \ref{ii:lemma:tracelessly}]{geng2}).
Hence a starting point 
would be to compute $Z(\lie{h})$ and the traceless outer derivation algebra
$\op{sout}(\lie{h})$ for each $\lie{h}$
(Table \ref{table:fiber2_derivations}).

This data and the actions of $\lie{q}$ on $\lie{h}$ are computed
in Section \ref{sec:fiber2_nonsplit_data}.
The reduction to extensions of $\tanisom \Euc^2$ by $\R^3$
(Lemma \ref{lemma:fiber2_nonsplit_oneext}) is proven in
Section \ref{sec:fiber2_nonsplit_oneext}).
Finally, Section \ref{sec:fiber2_nonsplit_geometries})
classifies the geometries resulting from this extension
(Prop.~\ref{prop:fiber2_nilmanifolds}).

\begin{table}[h!]
    \caption[Centers and traceless outer derivation algebras.
        ]{Centers and traceless outer derivation algebras.
        $\tanalg \Heis_3$ is written with basis $\{x,y,z\}$ where $[x,y]=z$.
        See Prop.~\ref{prop:fiber2_derivations} for details.}
    \label{table:fiber2_derivations}
    \begin{center}\begin{tabular}{ccc}
        $\lie{h}$ & $Z(\lie{h})$ & $\op{sout}(\lie{h})$ \\
        \hline
        $\lie{so}_3 \oplus \R$ & $0 \oplus \R$ & 0 \\
		$\lie{sl}_2 \oplus \R$ & $0 \oplus \R$ & 0 \\
		$\tanalg \isomplus \Euc^2 \oplus \R$ & $0 \oplus \R$ & $\R \semisum \R$ \\
		$\tanalg \isomplus \Heis_3$ & $\R z$ & $\R$ \\
		$\lie{so}_3$ & 0 & 0 \\
		$\lie{sl}_2$ & 0 & 0 \\
		$\tanalg \op{Sol}^3$ & 0 & 0 \\
		$\tanalg \isomplus \Euc^2$ & 0 & 0 \\
		$\tanalg \Heis_3$ & $\R z$ & $\lie{sl}_2$ \\
		$\R^3$ & $\R^3$ & $\lie{sl}_3$ \\
    \end{tabular}\end{center}
\end{table}

\subsubsection{Data for the extension problem}
\label{sec:fiber2_nonsplit_data}

This section establishes the list of traceless outer derivation algebras
in Table \ref{table:fiber2_derivations}
and classifies the homomorphisms $\lie{q} \to \op{sout} \lie{h}$.

\begin{lemma} \label{prop:fiber2_derivations} 
	The centers and traceless outer derivation algebras
	for the Lie algebras $\lie{h}$ are as given in
    Table \ref{table:fiber2_derivations}.
\end{lemma}
\begin{proof}
	Explicit calculation of the centers is omitted;
    they are verifiable in any convenient basis.

    Calculation of the derivation algebras uses the following shortcut:
	A derivation $d$ satisfies a Leibniz rule with respect to the Lie bracket;
    so if $d$ preserves some subset $A$ of $\lie{h}$, then it also preserves
    the derived series, the lower central series, and the centralizer of $A$.

	Lemma \ref{lemma:fiber2_essential_out} provides the derivation
    algebras for $\lie{so}_3$, $\lie{sl}_2$, and $\tanalg \isomplus \Euc^2$.
    The entry for $\lie{so}_3 \oplus \R$ follows since $\lie{so}_3$
    is the first term of the derived series and $\R$ is the center;
    $\lie{sl}_2 \oplus \R$ is handled likewise.
    Since $\R^3$ is abelian,
    $\op{sout} \R^3 = \lie{sl}(\R^3)$.
    Only four cases then remain.

    \paragraph{Case 1: $\tanalg \isomplus \Euc^2 \times \R$.}
    The first term of the derived
	series is the translation subalgebra, and the center is $\R$.
	The same calculation as in Lemma \ref{lemma:fiber2_essential_out}
	ensures that an outer derivation can be represented by $d$ which
	scales the translation subalgebra uniformly and has $dr \in \R$.
	Then if $d$ acts as the scalar $b$ on the translation subalgebra,
	tracelessness requires it to act as $-2b$ on $\R$. Writing this algebra
	as $\R \semisum \R$ is now just writing the matrices by which it acts
	on the span of $\R$ and $r$.

    \paragraph{Case 2: $\tanalg \op{Sol}^3$.}
	For the basis $\{x,y,z\}$ of $\tanalg \op{Sol}^3$ where $[z,x] = x$
	and $[z,y] = -y$, the first term of the derived series is $\R x + \R y$.
	Subtracting inner derivations, we may assume $dz = az$ for some $a$.
	Then
		\[ dx = d[z,x] = [dz, x] + [z,dx] = ax + [z,dx] ; \]
	so $(d-a)x = [z,dx]$, which implies $dx \in \R x$. Then $[z, dx] = dx$
	so the above implies $a = 0$, and we can subtract a multiple of
	$\op{ad} z$ to ensure $dx = 0$. Similarly, $dy \in \R y$,
	and tracelessness then implies $dy = 0$.

    \paragraph{Case 3: $\tanalg \Heis_3$.}
	(This calculation is outlined quickly in \cite[\S 6.3]{filipk}.)
    With the basis $\{x,y,z\}$
	in which $[x, y] = z$, inner derivations account for the
	$\R x + \R y \to \R z$ component; so a derivation $d$ can be taken to preserve
	$\R x + \R y$. Since $\R z$ is central, $d$ preserves $\R z$. Then
		\[ (\op{tr} d) z = [dx, y] + [x, dy] + dz = d[x,y] + dz = 2 dz , \]
	so $dz = 0$ and $d$ acts with trace zero on $\R x + \R y$.

    \paragraph{Case 4: $\tanalg \isomplus \Heis_3$.}
	Since $\tanalg \isomplus \Heis_3$ admits $\tanalg \isomplus \Euc^2$
	as a quotient, any outer derivation of the former induces an outer derivation
    of the latter. Hence (consulting Lemma \ref{lemma:fiber2_essential_out})
    any outer derivation of the former can be
	represented by a derivation $d$ such that
	\begin{align*}
		dx &= bx + c_1 z \\
		dy &= by + c_2 z \\
		dr &= c_3 z \\
		dz &= c_4 z
	\end{align*}
	for some real $b$ and $c_i$. Tracelessness implies $c_4 = -2b$, and
		\[ -2bz = dz = [dx, y] + [x, dy] = 2b[x, y] = 2bz \]
	implies $b = 0$. Then
		\[ dx = d[r,-y] = [dr, -y] + [r, -dy] = 0 \]
	and similarly $dy = 0$; so all that remains is $c_3$.
\end{proof}

\begin{lemma} \label{fiber2_isometric_actions}
	If $0 \to \lie{h} \to \lie{g} \to \lie{q} \to 0$ arises from
	a $5$-dimensional geometry, then
	$\lie{q} \to \op{sout} \lie{h}$ is nonzero only in the following cases,
	all of which lift to maps $\lie{q} \to \op{Der} \lie{h}$.
	\begin{align*}
		\lie{sl}_2
            &\overset{\sim}{\to} \op{sout} \tanalg \Heis_3 \\
		\lie{sl}_2
            &\hookrightarrow \lie{sl}_3 \cong \op{sout} \R^3 \\
		\lie{sl}_2 \cong \lie{so}_{2,1}
			&\hookrightarrow \lie{sl}_3 \cong \op{sout} \R^3 \\
		\lie{so}_3
            &\hookrightarrow \lie{sl}_3 \cong \op{sout} \R^3 \\
		\tanalg \isomplus \Euc^2 \cong \R^2 \semisum \lie{so}_2
			&\hookrightarrow \lie{sl}_3 \cong \op{sout} \R^3
            \text{ and its negative transpose} \\
		\tanalg \isomplus \Euc^2 \to \lie{so}_2
			&\overset{t \neq 0}{\hookrightarrow} \lie{sl}_2
            \cong \op{sout} \tanalg \Heis_3 \\
		\tanalg \isomplus \Euc^2 \to \lie{so}_2
			&\overset{t > 0}{\hookrightarrow} \lie{sl}_3 \cong \op{sout} \R^3 \\
	\end{align*}
\end{lemma}
\begin{proof} 
	If $\lie{q}$ is semisimple then only $\op{sout} \tanalg \Heis_3 \cong \lie{sl}_2 \R$
	and $\op{sout} \R^3 = \lie{sl}_3 \R$ have high enough dimension to admit
	nonzero images of $\lie{q}$. By counting representations, the only such are
	the first four maps listed above.

	Otherwise, $\lie{q} = \tanalg \isomplus \Euc^2$.
    The proof proceeds similarly to Lemma \ref{lemma:fiber2_essential_action},
    making note of the following facts.
    \begin{enumerate}
        \item Any homomorphism from $\tanalg \isomplus \Euc^2$ either is faithful or
            factors through the rotation part $\lie{so}_2$ (Lemma \ref{isomeuc_normalsubs}).
        \item Since $\lie{q} \to \op{sout} \lie{h}$ comes from the conjugation action
            $\isomplus \Euc^2 \to \op{Out} H$, the generator $r \in \lie{q}$ of $\op{SO}(2)$
            generates a compact subgroup of $\op{Out} H$.
    \end{enumerate}
    Lifting to $\op{Der} \lie{h}$ will follow from observing that
	$\op{out} \R^3 = \op{Der} \R^3$, and $\op{out} \tanalg \Heis_3$
	embeds in $\op{Der} \tanalg \Heis_3$ as the subalgebra preserving
	$\R x + \R y$.

    \paragraph{Case 1: $\rho$ is faithful.}
	Only $\op{sout} \tanalg \Heis_3 \cong \lie{sl}_2$ (dimension $3$)
	and $\op{sout} \R^3 \cong \lie{sl}_3$ (dimension $8$) have high enough dimension
    to admit a faithful image of $\tanalg \isomplus \Euc^2$.
	Since $\lie{sl}_2 \ncong \tanalg \isomplus \Euc^2$ (only the latter is solvable),
	a faithful image must land in $\lie{sl}_{3}$.

    Let $\tanalg \isomplus \Euc^2$ have basis $\{x,y,r\}$ with nonzero brackets
    $[r,x] = y$ and $[r,y] = -x$.
	An embedding sends $x$ and $y$ to commuting, similar matrices.
    Consulting the list of $2$-dimensional abelian subalgebras of $\lie{sl}_3 \R$
	(Part II, Lemma \ref{ii:lemma:r2_actions})
	yields exactly two embeddings
    $\tanalg \isomplus \Euc^2 \hookrightarrow \lie{sl}_3$ up to conjugacy
	(omitted entries are zero):
	\begin{align*}
		ax + by + cr &\mapsto
			\begin{pmatrix}
				& -c & a \\
				c & & b \\
				& & 0
			\end{pmatrix}
			\text{ or }
			\begin{pmatrix}
				& -c &  \\
				c & & \\
				-a & -b & 0
			\end{pmatrix} .
	\end{align*}

    \paragraph{Case 2: $\rho$ factors through $\lie{so}_2$.}
    From Table \ref{table:fiber2_derivations}, the volume-preserving parts of
    $\op{Out} (\widetilde{\isomplus \Euc^2} \times \R)$
    and $\op{Out} (\widetilde{\isomplus \Heis_3})$ are
    isomorphic to $\R \rtimes \R$ and $\R$ respectively,
    neither of which has nontrivial compact subgroups.
	So for a map factoring through $\lie{so}_2$ by which $r$ generates a compact
	subgroup of $\op{Out} H$, still only $\lie{sl}_2$ and $\lie{sl}_3$ admit
	nonzero images.
	In both of these, all images of $\lie{so}_2$ are conjugate,
	so the only flexibility is in how $\lie{so}_2$ maps to this image.
\end{proof}

\subsubsection{Reduction to one extension problem}
\label{sec:fiber2_nonsplit_oneext}

Using the data computed in the previous section,
this section proves Lemma \ref{lemma:fiber2_nonsplit_oneext}---the
claim that it suffices to consider extensions of $\isomplus \Euc^2$
by $\R^3$. The idea is to show that the maximal model geometries
produced from all other non-split extensions can be produced from
split extensions that will be handled in later sections.

\begin{proof}[Proof of Lemma \ref{lemma:fiber2_nonsplit_oneext}]
    Under the standing assumptions, $G$ is a extension of
    $Q = \isomplus B$ where $B = S^2$, $\Euc^2$, or $\Hyp^2$;
    and passing to Lie algebras yields a non-split extension
        \[ 0 \to \lie{h} \to \lie{g} \to \lie{q} \to 0 . \]

    \paragraph{Step 1: $\lie{q}$ is $\tanalg \isomplus \Euc^2$.}
    The other possibilities for $\lie{q}$ are semisimple.
    The first and second cohomology of a semisimple Lie algebra
    vanish in any coefficient system \cite[Thm.~1.3.2]{onishchik3}.
    So if $\lie{q}$ were semisimple, the classification of these extensions
    by $H^2(\lie{q}; Z(\lie{h}))$ (Thm.~\ref{thm:extensions_h2}) would imply
    that $\lie{g}$ is a split extension of $\lie{q}$.

    \paragraph{Step 2: If $\lie{h} \neq \R^3$,
        then $\lie{q}$ acts trivially on a $1$-dimensional $Z(\lie{h})$.}
    Since $H^2(\lie{q}; Z(\lie{h}))$ must be nonzero for a non-split extension
    to exist, $Z(\lie{h})$ is nonzero. Consulting the list of derivation algebras
    (Table \ref{table:fiber2_derivations})
    and the nontrivial actions (Lemma \ref{fiber2_isometric_actions}),
    the dimension of $Z(\lie{h})$ is $1$, and the action of $\lie{q}$ on it is trivial.
    We now aim to show that the resulting geometries are either non-maximal
    or realized with direct-product transformation groups $\tilde{G}$,
    which are handled later in Section \ref{sec:fiber2_product_groups}.

    \paragraph{Step 3: $H^2(\lie{q}; Z(\lie{h})) \cong \R$.}
	With trivial action, the cocycle condition reduces to
		\[ \sum_{\text{cyc}} c(x_1, [x_2, x_3]) = 0 . \]
	This expression is trilinear, so to check that it is zero,
    it suffices to check basis elements.
	If two of the basis elements are the same,
	then antisymmetry of $c$ and the Lie bracket imply the value is zero.
	Since $\lie{q}$ is $3$-dimensional, only one combination then needs
	to be checked.
    Letting $\lie{q}$ have basis $\{x,y,r\}$ where $[r,x] = y$ and $[r,y] = -x$,
	\begin{align*}
		c(x,[y,r]) + c(y,[r,x]) + c(r,[x,y])
			&= c(x,x) + c(y,y) + c(r,0) = 0 + 0 + 0 = 0 ;
	\end{align*}
	so \emph{every} linear map $\Lambda^2 \lie{q} \to Z(\lie{h})$ is a cocycle.

	With trivial action, coboundaries are maps
	$\Lambda^2 \lie{q} \to Z(\lie{h})$ which factor through
	the Lie bracket; so coboundaries account for any nonzero values
	on (and only on) $r \wedge x$ and $r \wedge y$.
    Therefore $H^2(\lie{q}; Z(\lie{h})) \cong \R$,
    generated by a cocycle that is nonzero on $x \wedge y$
    and zero on $r \wedge x$ and $r \wedge y$.

    \paragraph{Step 4: Describe the resulting groups $G$.}
	Since $\lie{q} \to \op{sout} \lie{h}$ lifts
	to $\op{Der} \lie{h}$ (Lemma \ref{fiber2_isometric_actions}),
    the Lie algebra structure of $\lie{g}$
    can be recovered from by interpreting
    the cocycle as a $Z(\lie{h})$-valued bracket on $\lie{q}$,
    as follows. On the vector space $\lie{h} \oplus \lie{q}$,
    define, following \cite[Eqn.~5.5]{alekseevsky_nonsuper},
        \[ [h_1 + q_1, h_2 + q_2]
            = q_1(h_2) - q_2(h_1) + c(q_1, q_2) + [q_1, q_2]_{\lie{q}} . \]
    If $\lie{h}$ is one of the direct sums $\lie{h}' \oplus \R$,
    then $Z(\lie{h})$ is this $\R$ summand. Using the above formula,
        \[ \lie{g}
        \cong \lie{h}' \oplus (\lie{n}_3 \semisum \lie{so}_2)
        \cong \lie{h}' \oplus \tanalg \isomplus \Heis_3 . \]
    Then $\tilde{G}$ is one of the direct products below,
    handled later in Section \ref{sec:fiber2_product_groups}.
	\begin{align*}
		S^3 &\times \widetilde{\isomplus \Heis_3}  &
		\widetilde{\op{PSL}(2,\R)} &\times \widetilde{\isomplus \Heis_3}  &
		\widetilde{\isomplus \Euc^2} &\times \widetilde{\isomplus \Heis_3}
	\end{align*}
    Otherwise, $\lie{h}$ is $\tanalg \Heis_3$ or $\tanalg \op{Isom} \Heis_3$.
	Then rescaling the $x,y$ plane gives an isomorphism of
    the resulting simply-connected group $\tilde{G}$ with either
	$\Heis_5 \rtimes \R$ or $\Heis_5 \rtimes \R^2$, with the action
	by rotations.
    Point stabilizers $\tilde{G}_p$
    act semisimply since they act by some sort of rotations,
    while the nilradical $\Heis_5$ acts nilpotently;
    so $\Heis_5$ meets the point stabilizer trivially
    and therefore acts freely on any resulting geometry $\tilde{G}/\tilde{G}_p$.
    Then $\tilde{G}/\tilde{G}_p \cong \Heis_5$, which is
    non-maximal since $\Heis_5 \rtimes \op{U}(2) / \op{U}(2)$ subsumes it.
\end{proof}

\subsubsection{The geometries}
\label{sec:fiber2_nonsplit_geometries}

Having established that geometries requiring non-split extensions
in the extension problem (Lemma \ref{lemma:fiber2_extension_formulation})
are accounted for by extensions of $\tanisom \Euc^2$ by $\R^3$
(Lemma \ref{lemma:fiber2_nonsplit_oneext}),
we now classify the resulting geometries.

\begin{proof}[Proof of Prop.~\ref{prop:fiber2_nilmanifolds}
    (Classification when $\tanalg G$ is an extension
    of $\tanisom \Euc^2$ by $\R^3$)]
    The key observation is the first step, that $M$ is a nilpotent Lie group;
    after that the problem reduces to computing some cohomology (Step 2)
    and some automorphism groups (Step 3).

    \paragraph{Step 1: $M$ is the simply-connected nilpotent Lie group
        whose Lie algebra is the induced extension of $\R^2$ by $\R^3$.}
    The inclusion of $\R^2$ as the translation subalgebra of $\tanisom \Euc^2$
    induces an inclusion of extensions
    \begin{center}
    \begin{tikzpicture}[description/.style={fill=white,inner sep=2pt}]
        \matrix (m) [matrix of math nodes, row sep=1em,
            column sep=2.5em, text height=1.5ex, text depth=0.25ex]
            { 0 & \R^3 & \tanalg G & \tanisom \Euc^2 & 0 \\
                0 & \R^3 & \lie{m} & \R^2 & 0 . \\ };
        \path[->]
            (m-1-1) edge (m-1-2)
            (m-1-2) edge (m-1-3)
            (m-1-3) edge (m-1-4)
            (m-1-4) edge (m-1-5);
        \path[->]
            (m-2-1) edge (m-2-2)
            (m-2-2) edge (m-2-3)
            (m-2-3) edge (m-2-4)
            (m-2-4) edge (m-2-5);
        \path[right hook->]
            (m-2-3) edge (m-1-3)
            (m-2-4) edge (m-1-4);
        \path
            (m-2-2) -- node[rotate=90] {$\cong$} (m-1-2);
    \end{tikzpicture}
    \end{center}
	Since $\R^2$ and $\R^3$ are abelian and
    $\R^2$ acts nilpotently
	(see the list of actions in Lemma \ref{fiber2_isometric_actions}),
    $\lie{m}$ is nilpotent.
    Moreover $\R^2$ is an ideal in $\tanisom \Euc^2$,
    so $\lie{m}$ is an ideal in $\tanalg G$, with codimension $1$.
    Since $\tanalg G$ is not nilpotent, having the non-nilpotent $\tanisom \Euc^2$
    as a quotient, $\lie{m}$ is the nilradical of $\tanalg G$.
    Then $\tanalg G$ is the extension
        \[ 0 \to \lie{m} \to \tanalg G \to \lie{so}_2 \to 0 , \]
    which splits since any linear map from
    $\lie{so}_2 \cong \R$ is a Lie algebra homomorphism.

    Since $G_p$ is compact, it acts semisimply in the adjoint representation on
    $\tanalg G$; while the nilradical of $\tanalg G$ acts nilpotently.
    Then $\tanalg G_p \cap \lie{m} = 0$---so $\tanalg G_p$ lies over $\lie{so}_2$
    in the above extension, and $\lie{m}$ is tangent to a group acting
    transitively on $M = G/G_p$.
    Since $\pi_1 (M) = 0$ and $\dim M = \dim \lie{m}$,
    this action is free, which identifies $M$
    with the simply-connected group whose Lie algebra is $\lie{m}$
    once a basepoint is chosen.

    \paragraph{Step 2: List candidate nilpotent Lie algebras $\lie{m}$.}
    This step proceeds using the classification of extensions by second cohomology
    (Thm.~\ref{thm:extensions_h2}).
	Since $\R^2$ is $2$-dimensional, it has no $3$-cocycles;
	so every $\Lambda^2 \R^2 \to \R^3$ is a cocycle.
    Let $\{x,y\}$ be a basis for $\R^2$; then
	coboundaries take the form $x \wedge y \mapsto x \cdot \beta(y) - y \cdot \beta(x)$,
    where $\cdot$ denotes the action of $\R^2$ on $\R^3$.
	Then (like in Part II, Lemma \ref{ii:lemma:solvable_r2_by_r3_cohomology})
	coboundaries account for values lying in $x \cdot \R^3 + y \cdot \R^3$, and
		\[ H^2(\R^2; \R^3) \cong \R^3/(x \cdot \R^3 + y \cdot \R^3) . \]
    Using the list of actions of $\tanisom \Euc^2$ on $\R^3$ in
    Lemma \ref{fiber2_isometric_actions}), one action has $x$ and $y$
    both acting as zero. In that case,
    every cocycle represents its own class, and $\lie{g}'$ is either isomorphic
	to $\R^5$ (zero class) or $R^2 \oplus \tanalg \Heis^3$ (nonzero class,
	normalized by scaling and conjugating in $\op{GL}(3,\R) = \op{Aut} \R^3$).

	Letting $\{u,v,w\}$ denote a basis of $\R^3$,
    the two nontrivial actions are
    \[
        \left\{\begin{array}{r}
            [x, w] = u \\\relax 
            [y, w] = v
        \end{array}\right\} \text{ and }
        \left\{\begin{array}{r}
            [x, u] = w \\\relax
            [y, v] = w
        \end{array}\right\}.
    \]
	On the left, the split extension can be re-expressed as the
	semidirect sum of $\R w$ acting on the span of the other four basis elements
    with two $2 \times 2$ Jordan blocks of eigenvalue $0$---call this
    $\R^4 \semisum_{x^2, x^2} \R$.
	On the right, the split extension is the $5$-dimensional
	Heisenberg algebra $\tanalg \Heis_5$.

	Up to changes of coordinates in the $x,y$ and $u,v$
	planes, the non-split extensions are accounted for by, respectively,
	\begin{align*}
		[x, y] &= w  & [x, y] &= u .
	\end{align*}

    \paragraph{Step 3: Eliminate $\lie{m}$ whose automorphism group
        is either too large or too small.}
    The maximal geometry realizing $M$ is $M \rtimes K / K$ where $K$ is maximal
    compact in $\op{Aut} M$ (Lemma \ref{lemma:solvable_maximality}).
    Since the assumptions imply $\dim G = 6$,
    we require the maximal compact subgroup of $\op{Aut} M$ to be $S^1$.

    In three of the groups $M$ produced in Step 2,
    the maximal compact subgroup of $\op{Aut} M$ is too large,
	which makes $G/G_p$ non-maximal.
	\begin{align*}
        \op{Aut} \R^5 &\cong \op{GL}(5,\R) \supset \op{SO}(5) \\
		\op{Aut} (\R^2 \times \Heis_3) &\supseteq \op{SL}(2,\R) \times \op{SL}(2,\R)
            \supset \op{SO}(2) \times \op{SO}(2) \\
		\op{Aut} \Heis_5 &\supseteq U(2)
			\text{ (see Section \ref{sec:fiber4_proof}, Step 3)}
	\end{align*}
	If $\lie{m}$ has basis $\{x,y,u,v,w\}$ with
	\begin{align*}
		[x,u] &= w  &  [y,v] &= w  &  [x,y] &= u ,
	\end{align*}
	then any automorphism preserves the following filtration
	of characterstic ideals.
	\begin{align*}
        Z(\lie{m}) &= \R w \\
		[\lie{m}, \lie{m}] &= \R u + \R w \\
		\text{The preimage of } Z(\lie{m}/Z(\lie{m})) &= \R u + \R v + \R w \\
		\text{The centralizer of } [\lie{m}, \lie{m}] &= \R y + \R u + \R v + \R w
	\end{align*}
	Therefore every automorphism is upper-triangular, so
	$\op{Aut} \lie{m}$ contains no compact subgroups of positive dimension.
    Then $M \rtimes S^1$ cannot be defined with a nontrivial action by $S^1$.

    \paragraph{Step 4: Both remaining geometries are maximal.}
	What remains are the two possibilities for $M$
    claimed in the statement of this Proposition.
    Their Lie algebras can be expressed with basis $\{x,y,u,v,w\}$ and
	\begin{align*}
		[x,w] &= u  &  [y,w] &= v  &  [x,y] &= w \text{ or } 0 .
	\end{align*}
	Similarly to Step 3 above,
	\begin{align*}
		Z(\lie{m}) &= \R u + \R v \\
		[\lie{m}, \lie{m}] &= \R u + \R v + \R w \text{ if } [x,y] = w \\
		\{a \mid \op{rank} \op{ad} a \leq 1\} &= \R x + \R y + \R u + \R v
			\text{ if } [x,y] = 0 .
	\end{align*}
	Either way, every automorphism is block upper-triangular
	with two $2 \times 2$ blocks and one $1 \times 1$ block. The group
	of such block upper-triangular matrices has a torus $T^2$
	(in each $2 \times 2$ block, something conjugate to a rotation)
	as its maximal compact subgroup
    (Part II, \cite[Lemma \ref{ii:lemma:compacts_in_aut}]{geng2});
	so $\op{Aut} \lie{m}$ has maximal compact subgroup inside that.

	Up to scaling $x$, $y$, and possibly $w$, any automorphism in this $T^2$
	merely rotates the $x,y$ plane. Since $\op{ad} w$ takes this plane
	to the $u,v$ plane, the rotation on the $u,v$ plane is determined;
	so $K \cong S^1$, acting by rotation at equal rates on these planes.

    \paragraph{Step 5: Both remaining candiates are model geometries.}
    A nilpotent Lie algebra with rational structure constants in some basis
    admits a cocompact lattice $\Gamma$ \cite[Thm 2.12]{raghunathan}.
    In the bases used above, the structure constants are all integers;
    so in each case
    $\Gamma \backslash M / \{1\} \cong \Gamma \backslash M \rtimes S^1 / S^1$
    is a compact manifold modeled on $M$.
\end{proof}

\subsection{Line bundles over \texorpdfstring{$\mathbb{F}^4$}{F4}
    and \texorpdfstring{$T^1 \Euc^{1,2}$}{T\^{}1 E(1,2)}
    (semidirect-product isometry groups)}
\label{sec:fiber2_semidirect}

Having just classified geometries $G/G_p$
where $\tanalg G$ is an extension of $\tanisom \Euc^2$ by $\R^3$
or any non-split extension,
this section classifies the next cluster of geometries---those
where $\tanalg G$ is a split extension,
i.e.\ a semidirect sum $\lie{h} \semisum \lie{q}$
with one of the remaining combinations of $\lie{h}$ and $\lie{q}$
that has a nontrivial action. Five combinations remain
of the list from Lemma \ref{fiber2_isometric_actions},
producing the five groups in Table \ref{table:fiber2_semidirect_centers}.
in terms of which this section's classification result is stated.
We will verify that these are model geometries in
Prop.~\ref{prop:fiber2_semidirect_model},
but maximality is deferred to Section \ref{sec:fiber2_maximality}
to take advantage of results from the remainder of the classification
(Section \ref{sec:fiber2_product_groups}).

\begin{prop} \label{semidirect_product_geometries}
    Let $M = G/G_p$ be a maximal model geometry such that
	$\tilde{G}$ is one of the semidirect products listed in
    Table \ref{table:fiber2_semidirect_centers}.
    Let $\gamma: \R \to \SLcover$ send $t$ to a rotation by $2\pi t$ radians,
    and fix a nontrivial $z \in Z(\Heis_3)$.
    Then $M$ is one of
	\begin{align*}
        \op{SAff} \R^2 = \R^2 \rtimes \SLcover
            &\cong (\R^2 \rtimes \SLcover) \rtimes \op{SO}(2) / \op{SO}(2) \\
		\Euc \times \mathbb{F}^4
            &= \R \times \R^2 \rtimes \op{SL}(2,\R) / \op{SO}(2) \\
        \mathbb{F}^5_a
            &= \Heis_3 \rtimes \SLcover / \{ atz, \gamma(t) \}_{t \in \R} ,
                \quad a = 0 \text{ or } 1 \\
		T^1 \Euc^{1,2}
            &= \R^3 \rtimes \op{SO}(1,2)/\op{SO}(2) .
	\end{align*}
\end{prop}

The proof requires knowing the center of the semidirect product $\tilde{G}$;
the following formula computes the centers listed in Table
\ref{table:fiber2_semidirect_centers}.
\begin{lemma} \label{semidirect_product_centers}
	If $C_x$ is conjugation by $x$, then
	\[
		Z(A \rtimes_\phi B)
		= \{ (x,y) \in A^B \times Z(B) \mid \phi(y) = C_{x^{-1}} \} .
	\]
	In particular, if $\phi(B) \cap \op{Inn} A = \{1\}$, then
	\[
		Z(A \rtimes_\phi B)
		= Z(A)^B \times (Z(B) \cap \ker \phi) .
	\]
\end{lemma}
\begin{proof}
	Suppose $xy = (x,y) \in Z(A \rtimes_\phi B)$. Then
	to commute with all $b \in B$,
	\begin{align*}
		b &= xyby^{-1}x^{-1} \\
		bx = \phi(b)(x) b &= xyby^{-1} ;
	\end{align*}
	so $\phi(b)(x) = x$ and $b = yby^{-1}$ for all $b \in B$, i.e.
	$x \in A^B$ and $y \in Z(B)$.
	To commute with all $a \in A$,
	\begin{align*}
		a
			&= xyay^{-1}x^{-1} \\
			&= x \phi(y)(a) x^{-1} \\
            &= C_x \phi(y)(a).
	\end{align*}
	Since $A$ and $B$ generate $A \rtimes_\phi B$, these
	conditions suffice.

	If $\phi(B) \cap \op{Inn} A = \{1\}$, then $\phi(y)$ is identity,
	and the second condition reduces to $x \in Z(A)$.
\end{proof}

\begin{table}[h!]
    \caption{Nontrivial semidirect products covering candidate isometry groups $G$}
    \label{table:fiber2_semidirect_centers}
    \begin{center}\begin{tabular}{ccl}
        $\tilde{G} = \tilde{H} \rtimes \tilde{Q}$ & center & notes on action \\
        \hline
        \rule[0pt]{0pt}{18pt}
        $\R^3 \rtimes \widetilde{\op{SO}(3)}$ &
            $\{0\} \times \Z/2\Z$ &
            standard representation \\
        \rule[0pt]{0pt}{16pt}
        $\R^3 \rtimes \widetilde{\op{SO}(2,1)}$ &
            $\{0\} \times \Z$ &
            standard representation \\
        \rule[0pt]{0pt}{16pt}
        $\Heis_3 \rtimes \widetilde{\op{SL}(2,\R)}$ &
            $\R \times \Z$ &
            as $\op{SL}(2,\R)$ on $x,y$ plane \\
        \rule[0pt]{0pt}{16pt}
        $\Heis_3 \rtimes \widetilde{\isomplus \Euc^2}$ &
            $\R \times \Z$ &
            family of actions,
            factors through $\widetilde{\op{SO}(2)}$ \\
        \rule[0pt]{0pt}{16pt}
        $\R \times \R^2 \rtimes \widetilde{\op{SL}(2,\R)}$ &
            $\R \times \Z$ &
            trivial action on $\R$ \\
    \end{tabular}\end{center}
\end{table}

\begin{proof}[Proof of Prop.~\ref{semidirect_product_geometries}
    (Classification when the isometry group is a semidirect product)]
    Suppose that $M = G/G_p$ is such a geometry (a maximal model geometry
    with $\tilde{G}$ listed in Table \ref{table:fiber2_semidirect_centers}),
    and let $\tilde{G}_p$ be the preimage in $\tilde{G}$ of $G_p$.

    \paragraph{Preparatory step: $\tilde{G}_p$
        is a $1$-parameter subgroup lying over a maximal torus
        of $\tilde{G}/Z(\tilde{G})$.}
    Since every candidate $\tilde{G}$ in Table
    \ref{table:fiber2_semidirect_centers} is $6$-dimensional,
    $G_p$ and $\tilde{G}_p$ are $1$-dimensional.
    The homotopy exact sequence for $\tilde{G}_p \to \tilde{G} \to M$
    and the assumption that $\pi_1(M) \cong \pi_0(M) \cong 0$
    imply that $\tilde{G}_p$ is connected.
    Moreover, $G_p = \tilde{G}_p / (\tilde{G}_p \cap Z(\tilde{G}))$
    (Rmk.~\ref{rmk:quotient_by_center}); 
    so $G_p$ becomes a copy of $S^1$ in $\tilde{G}/Z(\tilde{G})$.
    For every candidate $\tilde{G}$ listed, $\tilde{G}/Z(\tilde{G})$
    has maximal torus $\op{SO}(2) \cong S^1$.

    For every candidate in Table \ref{table:fiber2_semidirect_centers},
    the quotient $\tilde{G}/Z(\tilde{G})$
    so in the contractible $\tilde{G}$,
    where $1$-parameter subgroups are isomorphic to $\R$,
    the intersection of $\tilde{G}_p$ and $Z(\tilde{G})$ is nontrivial.

    \paragraph{Case 1: If $\dim Z(\tilde{G}) = 0$
        then $M = T^1 \Euc^{1,2}$.}
    If $\dim Z(\tilde{G}) = 0$ then $\tilde{G}$ covers $\tilde{G}/Z(\tilde{G})$;
    so $\tilde{G}_p$ is unique up to conjugacy in $\tilde{G}$.
    The resulting spaces, and the names chosen for them, are
    \begin{align*}
        \R^3 \rtimes \op{SO}(3) / \op{SO}(2) &= T^1 \Euc^3 \\
        \R^3 \rtimes \op{SO}(2,1) / \op{SO}(2) &= T^1 \Euc^{1,2} .
    \end{align*}
    The final classification omits $T^1 \Euc^3$ since it is non-maximal,
    being subsumed by $\Euc^3 \times S^2$.
    An interpretation of $T^1 \Euc^{1,2}$,
    by analogy to the Euclidean (i.e.\ positive-definite) case,
    is as one connected component of the sub-bundle of
    $T\R^3$ consisting of the vectors $v$ where $\left<v,v\right> = 1$
    and $\left<\cdot,\cdot\right>$ has signature $(+,-,-)$.

    \paragraph{Case 2: If $\tilde{G} = \Heis_3 \rtimes \widetilde{\idisom \Euc^2}$
        then $M$ is a non-maximal $\Heis_3 \times \Euc^2$.}
    In this case, $\tilde{G}$ is an extension
        \( 1 \to \Heis_3 \times \R^2 \to \tilde{G} \to \widetilde{\op{SO}(2)} \to 1 , \)
    with $\tilde{G}_p$ surjecting onto $\widetilde{\op{SO}(2)}$; so
        \[ G/G_p
            \cong (\Heis_3 \times \R^2) \rtimes \widetilde{\op{SO}(2)}
            / \widetilde{\op{SO}(2)} . \]
    Since $\Heis_3 \times \R^2$ admits at least a torus's worth of automorphisms
    (one $\op{SO}(2)$ acting by rotation on $\R^2$ and another acting by
    rotation on $\Heis_3$), this geometry is subsumed by
    $(\Heis_3 \times \R^2) \rtimes \op{SO}(2)^2 / \op{SO}(2)^2$.

    \paragraph{Case 3: If $\tilde{G} = \Heis_3 \rtimes \widetilde{\op{SL}(2,\R)}$
        then $M$ is $F^5_0$ or $F^5_1$.}
    Let $\gamma: \R \to \widetilde{\op{SO}(2)} \subset \widetilde{\op{SL}(2,\R)}$
    send $t$ to a rotation by $2\pi t$, and put coordinates on $\Heis_3$ so that
        \[ (x,y,z) (x'y'z') = (x+x', y+y', z+z'+xy') . \]
    Since $Z(\tilde{G})$ is $1$-dimensional, the
    groups $\tilde{G}_p$ lying over $\widetilde{\op{SO}(2)}$
    form a $1$-dimensional family: for $a \in \R$, let
        \[ F^5_a  = \Heis_3 \rtimes \widetilde{\op{SL}(2,\R)} /
            \{ (0,0,at), \gamma(t) \}_{t \in \R} . \]
    Conjugating $\widetilde{\op{SL}(2,\R)}$ by a reflection in $O(2)$
    would exchange $\gamma(t)$ with $\gamma(-t)$; so up to this conjugation
    we may assume $a \geq 0$. If $a > 0$, conjugating $\Heis_3$ by
    the automorphism
        \[ (x,y,z) \mapsto \left(xa^{-1/2}, ya^{-1/2}, za^{-1}\right) \]
    allows assuming $a = 1$; so $M = G/G_p$ is
    equivariantly diffeomorphic to either $F^5_0$ or $F^5_1$.
    The two cases are distinguished by whether a Levi subgroup of $G$ is
    isomorphic to $\op{SL}(2,\R)$ ($a = 0$) or its universal cover ($a = 1$).

    \paragraph{Case 4: If $\tilde{G} = \R \times \R^2 \rtimes \widetilde{\op{SL}(2,\R)}$
        then $M$ is $\Euc \times \mathbb{F}^4$ or $\R^2 \rtimes \widetilde{\op{SL}(2,\R)}$.}
    As in Case $3$, conjugation by an element of $\op{O}(2)$ and a rescaling of
    the $\R = \idcompo{Z(\tilde{G})}$ factor allow assuming
        \[ G/G_p \cong \R \times \R^2 \rtimes \widetilde{\op{SL}(2,\R)} /
            \{ (at, 0, 0), \gamma(t) \}_{t \in \R} \]
    with $a = 0$ ($M = \Euc \times \mathbb{F}^4$)
    or $a = 1$ ($M = \R^2 \rtimes \widetilde{\op{SL}(2,\R)}$);
    and the two are again distinguished by Levi subgroups.
\end{proof}

\begin{prop} \label{prop:fiber2_semidirect_model}
    All geometries listed in Prop.~\ref{semidirect_product_geometries}
    are model geometries.
\end{prop}
\begin{proof}
    The kernel $\Gamma_3$ of $\op{SL}(2,\Z) \to \op{SL}(2,\Z/3\Z)$
    is a torsion-free lattice in $\op{SL}(2,\R)$ \cite[Thm.\ 4.8.2 Case 1]{dwm};
    so its lift $\tilde{\Gamma}_3$ is a torsion-free lattice in
    the infinite cyclic cover $\widetilde{\op{SL}(2,\R)}$.
    Then $\Gamma_3$ (resp.\ $\tilde{\Gamma}_3$) acts without fixed points
    (i.e. freely) anywhere that $\op{SL}(2,\R)$ (resp.\ $\widetilde{\op{SL}(2,\R)}$)
    acts faithfully.
    So finite-volume manifolds modeled on three of the geometries from
    Prop.~\ref{semidirect_product_geometries} can be constructed using
    these groups and are listed in
    Table \ref{table:fiber2_semidirect_lattices}.
    \begin{table}[h!]
        \caption{Finite-volume manifolds $\Gamma\backslash G / K$ modeled on $G/K$
            fibering over $\mathbb{F}^4$}
        \label{table:fiber2_semidirect_lattices}
        \begin{center}\begin{tabular}{cccc}
            $\Gamma$ & $G$ & $K$ & $G/K$ \\
            \hline
            $\Z^2 \rtimes \tilde{\Gamma}_3$ &
                $(\R^2 \rtimes \widetilde{\op{SL}(2,\R)}) \rtimes \op{SO}(2)$ &
                $\op{SO}(2)$ &
                $\R^2 \rtimes \SLcover$ \\
            $\Z \times \Z^2 \rtimes \Gamma_3$ &
                $\R \times \R^2 \rtimes \SLcover$ &
                $\op{SO}(2)$ &
                $\Euc \times \mathbb{F}^4$ \\
            $\Heis_3(\Z) \rtimes \Gamma_3$ &
                $\Heis_3 \rtimes \op{SL}(2,\R)$ &
                $\op{SO}(2)$ &
                $\mathbb{F}^5_0$ \\
        \end{tabular}\end{center}
    \end{table}

    \paragraph{Special case 1: $\mathbb{F}^5_1$.}
    There is one complication:
    while $\Lambda = \Heis_3(\Z) \rtimes \tilde{\Gamma}_3$ is a torsion-free lattice
    in $\tilde{G}$, we must verify that it descends to a torsion-free lattice in $G$.
    Observe that
        \[ \tilde{G}_p = \{(0,0,t), \gamma(t)\}_{t \in \R} , \]
    where $\gamma(t)$ is rotation by $2\pi t$,
    meets $\Lambda$ in
    $\{(0,0,t), \gamma(t)\}_{t \in \Z} = \tilde{G}_p \cap Z(\tilde{G})$.
    Then $\Lambda$ remains discrete in
    $G = \tilde{G}/(\tilde{G}_p \cap Z(\tilde{G}))$.

    Only elliptical elements can become torsion in a quotient of
    $\widetilde{\op{SL}(2,\R)}$, and the elliptical elements of
    the torsion-free $\tilde{\Gamma}_3$
    all lie over the identity.
    Then anything in $\Lambda$ that becomes torsion in $G$
    lies over the identity of $\R^2 \rtimes \op{SL}(2,\R)$---i.e.\ in
    $Z(\tilde{G}) \cong \R \times \Z$.
    Since $Z(\tilde{G})$ has image in $G$
    isomorphic to $Z(\tilde{G})/(\tilde{G}_p \cap Z(\tilde{G})) \cong \R$,
    the image in $G$ of $\Lambda$ is torsion-free.

    Then $\Lambda \backslash \mathbb{F}^5_1$ is a manifold,
    with finite volume since $\Lambda$ is a lattice.

    \paragraph{Special case 2: $T^1 \Euc^{1,2}$.}
    By the classification of representations of $\widetilde{\op{SL}(2,\R)}$
    \cite[11.8]{fultonharris},
    the action of $\op{SL}(2,\R)$ on
    the space $\op{Sym}^2 \R^2$ of symmetric $2 \times 2$ matrices
    given by
        \[ g \cdot M = gMg^T \]
    factors through the standard representation of $\op{SO}(2,1)$.
    Since elements $\Gamma_3$ have integer entries,
    the action by $\Gamma_3$ preserves the subgroup $\Lambda$
    consisting of symmetric integer matrices.
    Then $\Lambda \rtimes \Gamma_3$ descends to a lattice in
    $\R^3 \rtimes \op{SO}(2,1)$.
    This lattice is torsion-free: since $\op{SL}(2,\R)$
    double covers $\op{SO}(2,1)$, an $n$-torsion element of
    $\R^3 \rtimes \op{SO}(2,1)$ lifts to a $2n$-torsion element
    of $\R^3 \rtimes \op{SL}(2,\R)$, and $\Lambda \rtimes \Gamma$
    was constructed to be torsion-free.

    Therefore $(\Lambda \rtimes \Gamma_3) \backslash T^1 \Euc^{1,2}$
    is a finite-volume manifold modeled on $T^1 \Euc^{1,2}$.
    %
\end{proof}

\begin{rmk}
    Maximality will need to wait until the classification is more complete,
    since the proof depends on knowing all geometries with point stabilizer
    containing $S^1_1$ or $S^1_{1/2}$. The proof is eventually given in
    Prop.~\ref{prop:fiber2_semidirect_maximality}.
\end{rmk}

\subsection{Product geometries and associated bundles}
\label{sec:fiber2_product_groups}

The last piece of the classification, necessarily a sort of catch-all
for the leftovers, is this section's main result.

\begin{prop} \label{prop:fiber2_bundle_list}
    The $5$-dimensional maximal model geometries $G/G_p$
    for which $\tilde{G}$ occurs in Table \ref{table:fiber2_product_groups}
    are
    \begin{enumerate}[(i)]
        \item products of $2$- and $3$-dimensional geometries and
        \item the following ``associated bundle'' geometries.
            \begin{align*}
                \Heis_3 &\times_\R S^3  &
                \SLcover \times_\alpha S^3 &, \quad 0 < \alpha < \infty \\
                \Heis_3 &\times_\R \SLcover  &
                \SLcover \times_\alpha \SLcover &, \quad 0 < \alpha \leq 1 \\
                && L(a;1) \times_{S^1} L(b;1) &, \quad 0 < a \leq b \text{ coprime in } \Z
            \end{align*}
    \end{enumerate}
\end{prop}

After the preceding sections,
all that remain are geometries where $\tilde{G}$
is a split extension with trivial action---i.e.\ direct products,
which Table \ref{table:fiber2_product_groups}
lists with duplicates hidden.
We start with a study the associated bundle geometries
from several viewpoints in Section \ref{sec:assoc_bundles},
since the vocabulary for describing them
is useful in proving the classification in Prop.~\ref{prop:fiber2_bundle_list}
(Section \ref{sec:fiber2_product_proof}).
For completeness, the product geometries are listed in
Section \ref{sec:fiber2_product_list}.
Maximality for non-products
is deferred to Section \ref{sec:fiber2_maximality}.

\begin{table}[h!]
    \caption{Groups covering the isometry group $G$ of geometries in this section}
    \label{table:fiber2_product_groups}
    \begin{center}
    \begin{tabular}{c|ccc}
        $\lie{h}$ & $B = \Euc^2$ & $B = S^2$ & $B = \Hyp^2$ \\
        \hline
        $\tanisom \Euc^2 \times \R$
            \rule[0pt]{0pt}{18pt}
            & $\widetilde{\isomplus \Euc^2} \times \widetilde{\isomplus \Euc^2} \times \R$
            & $\widetilde{\isomplus \Euc^2} \times S^3 \times \R$
            & $\widetilde{\isomplus \Euc^2} \times \widetilde{\op{PSL}(2,\R)} \times \R$ \\
        $\lie{so}_3 \times \R$
            \rule[0pt]{0pt}{16pt}
            & (duplicate)
            & $S^3 \times S^3 \times \R$
            & $S^3 \times \widetilde{\op{PSL}(2,\R)} \times \R$ \\
        $\lie{sl}_2 \times \R$
            \rule[0pt]{0pt}{16pt}
            & (duplicate)
            & (duplicate)
            & $\widetilde{\op{PSL}(2,\R)} \times \widetilde{\op{PSL}(2,\R)} \times \R$ \\
        $\tanisom \Heis_3$
            \rule[0pt]{0pt}{16pt}
            & $\widetilde{\isomplus \Heis_3} \times \widetilde{\isomplus \Euc^2}$
            & $\widetilde{\isomplus \Heis_3} \times S^3$
            & $\widetilde{\isomplus \Heis_3} \times \widetilde{\op{PSL}(2,\R)}$ \\
        $\tanisom \Euc^2$
            \rule[0pt]{0pt}{16pt}
            & $\widetilde{\isomplus \Euc^2} \times \widetilde{\isomplus \Euc^2}$
            & $\widetilde{\isomplus \Euc^2} \times S^3$
            & $\widetilde{\isomplus \Euc^2} \times \widetilde{\op{PSL}(2,\R)}$ \\
        $\lie{so}_3$
            \rule[0pt]{0pt}{16pt}
            & (duplicate)
            & $S^3 \times S^3$
            & $S^3 \times \widetilde{\op{PSL}(2,\R)}$ \\
        $\lie{sl}_2$
            \rule[0pt]{0pt}{16pt}
            & (duplicate)
            & (duplicate)
            & $\widetilde{\op{PSL}(2,\R)} \times \widetilde{\op{PSL}(2,\R)}$ \\
        $\tanalg \op{Sol}^3$
            \rule[0pt]{0pt}{16pt}
            & $\op{Sol}^3 \times{} \widetilde{\isomplus \Euc^2}$
            & $\op{Sol}^3 \times S^3$
            & $\op{Sol}^3 \times{} \widetilde{\op{PSL}(2,\R)}$ \\
        $\tanalg \Heis_3$
            \rule[0pt]{0pt}{16pt}
            & $\Heis_3 \times \widetilde{\isomplus \Euc^2}$
            & $\Heis_3 \times S^3$
            & $\Heis_3 \times \widetilde{\op{PSL}(2,\R)}$ \\
        $\R^3$
            \rule[0pt]{0pt}{16pt}
            & (Prop.~\ref{prop:fiber2_nilmanifolds}) & $\R^3 \times S^3$ & $\R^3 \times \widetilde{\op{PSL}(2,\R)}$ \\
    \end{tabular}
    \end{center}
\end{table}

\subsubsection{Associated bundle geometries}
\label{sec:assoc_bundles}

This section is a study of the associated bundle geometries---the only geometries
in this classification that have abelian isotropy and are not products.
The most concise names, the classification, and the verification of maximality
all use slightly different constructions of these spaces;
so this section will provide the notation for and relationships between
all three constructions.
The first definition motivates the name ``associated bundle''.

\begin{defn}[\textbf{Associated bundles}, see e.g.\ {\cite[Defn.~5.1]{husemoller}}]
    \label{defn:fiber2_associated_bundles}
    Let $E \to B$ be a principal $G$-bundle, 
    and let $\rho: G \to \op{Aut} F$ be an action of $G$ on some space $F$.
    The $F$-bundle $E \times_\rho F$ \emph{associated} to $E \to B$
    is the bundle over $B$ whose total space is
        \[ E \times F / (e,f) \sim (eg, \rho(g^{-1}) f) \text{ for } g \in G . \]
\end{defn}
\begin{eg}[\textbf{Associated bundle geometries}]
    \label{eg:fiber2_associated_homogeneous}
    For homogeneous $E = H/H_p$ and $F = K/K_q$,
    suppose $G$ is a $1$-dimensional subgroup of $H$ commuting with $H_p$
    by which it acts on $E$ on the right,
    and $\rho(G)$ is central in $K$. Then
        \[ E \times_\rho F
        = (H \times K) / \left(H_p \cdot K_q \cdot
            \left\{ (g, \rho(g)^{-1}) \mid g \in G \right\}\right) . \]
    In particular we may take $E$ and $F$ each to be one of the homogeneous spaces
    \begin{align*}
        \Heis_3 &\cong \Heis_3 \rtimes \widetilde{\op{SO}(2)} / \widetilde{\op{SO}(2)} \\
        L(a;1) &\cong S^3 \rtimes S^1 / \left(e^{(2 \pi i / a) \Z} \times S^1\right) \\
        \SLcover &\cong \SLcover \rtimes \widetilde{\op{SO}(2)} / \widetilde{\op{SO}(2)} ,
    \end{align*}
    with $G$ or $\rho(G)$ the identity component of the center in each case.

    The lens space $L(a;1)$
    is $S^3 / e^{(2 \pi i / a ) \Z}$ where $S^3$ is the unit quaternions
    and $e^{(2 \pi i / a)\Z} = \Z/a\Z \subset S^1 \subset \R + \R i$.
    In paticular, $L(1;1) \cong S^3$ and $L(2;1) \cong \R P^3$.
\end{eg}

Our notation is modeled on that in \cite[\S{1.3}]{sharpe} surrounding
the ``Vector Bundles'' subsection---we denote such a bundle by $E \times_G F$
when there is a natural choice of $G$-action on $F$, otherwise $E \times_* F$
where $*$ is enough data to specify the $G$-action. Since $G$ is $1$-dimensional,
a real number suffices to express this data once some conventions are chosen.

The next definition is a description of the associated bundle geometries
as homogeneous spaces in terms of some data arising from the classification strategy.
\begin{defn}[\textbf{Point stabilizers of associated bundle geometries}]
    \label{defn:fiber2_slope}
    Let $\tilde{G}$ be a connected, simply-connected Lie group,
    and let $\pi$ denote the quotient map $\tilde{G} \to \tilde{G}/Z(\tilde{G})$.
    Suppose $Z(\tilde{G})$ is $1$-dimensional,
    $\tilde{G}/Z(\tilde{G})$ has a $2$-dimensional maximal torus $T$,
    and $\tau_{0,0}: \R^2 \to \tilde{G}$ is a homomorphism
    such that $\pi \circ \tau_{a,b}$ has image $T$ and kernel the standard $\Z^2$.
    For $a$ and $b$ in $\R \cong Z(\tilde{G})^0$, define\footnote{
        $Z(\tilde{G})^0 \cong \R$ follows from computing its $\pi_1$
        using the homotopy exact sequence for
        $Z(\tilde{G}) \to \tilde{G} \to \tilde{G}/Z(\tilde{G})$
        and the fact that $\pi_2 = 0$ for Lie groups \cite[Prop.~V.7.5]{brockerdieck}.
    }
    \begin{align*}
        \tau_{a,b} : \R^2 &\to \tilde{G} \\
            x,y &\mapsto (xa + yb)\tau_{0,0}(x,y) .
    \end{align*}
\end{defn}

The final definition is a notation for line bundles and circle bundles
over products of $2$-dimensional geometries.
Its resemblance to the $3$-dimensional geometries with $\op{SO}(2)$ isotropy
may provide some intuition and will be used to distinguish the
associated bundle geometries from each other and from other geometries.
\begin{defn}[\textbf{Associated bundle geometries as bundles over products}]
    Let $X$ and $Y$ be $\Euc^2$, $S^2$, or $\Hyp^2$,
    scaled to have curvature $0$ or $\pm 1$;
    and let $H$ be $\R$ or $S^1 \cong \R/\Z$.
    Given $a$ and $b$ in $\tanalg H$,
    let $\xi^{a,b} \to {X \times Y}$ denote the simply-connected
    principal $H$-bundle over $X \times Y$ with a connection whose curvature form is
        \[ \Omega_{a,b} = \frac{1}{2\pi} \left( \mathrm{vol}_X \otimes a + \mathrm{vol}_Y \otimes b \right) . \]
\end{defn}

\begin{table}[h!]
    \caption{Descriptions of associated bundle geometries, assuming nonzero $a$ and $b$}
    \label{table:fiber2_associated_bundles}
    \begin{center}
    \begin{tabular}{cccc}
        $\tilde{G}$ for $\tilde{G}/\tau_{a,b}(\R^2)$ & $\tau_{0,0}(\R^2)$ & Associated bundle & $\xi^{x,y} \to X \times Y$ \\
        \hline
        $S^3 \times S^3 \times \R$ &
        $S^1 \times S^1$ &
        $L(a;1) \times_{S^1} L(b;1)$ &
        $\xi^{a,b} \to {S^2 \times S^2}$ \\
        $\SLcover \times S^3 \times \R$ &
        $\widetilde{\op{SO}(2)} \times S^1$ &
        $\SLcover \times_{a/b} S^3$ &
        $\xi^{-a,b} \to {\Hyp^2 \times S^2}$ \\
        $\SLcover \times \SLcover \times \R$ &
        $\widetilde{\op{SO}(2)} \times \widetilde{\op{SO}(2)}$ &
        $\SLcover \times_{a/b} \SLcover$ &
        $\xi^{-a,-b} \to {\Hyp^2 \times \Hyp^2}$ \\
        $(\widetilde{\op{Isom} \Heis_3})^0 \times \SLcover$ &
        $\widetilde{\op{SO}(2)} \times \widetilde{\op{SO}(2)}$ &
        $\Heis_3 \times_{\R} \SLcover$ &
        $\xi^{1,1} \to {\Euc^2 \times \Hyp^2}$ \\
        $(\widetilde{\op{Isom} \Heis_3})^0 \times S^3$ &
        $\widetilde{\op{SO}(2)} \times S^1$ &
        $\Heis_3 \times_{\R} S^3$ &
        $\xi^{1,1} \to {\Euc^2 \times S^2}$ \\
    \end{tabular}
    \end{center}
\end{table}

The correspondence between these three definitions is
summarized in Table \ref{table:fiber2_associated_bundles}.
The proof of the correspondence for the ``associated bundle'' column
is written out only for the first row for illustrative purposes,
since this part of the correspondence is only used
in the hope of selecting an evocative name.
The proof for the rest of the correspondence is
Prop.~\ref{prop:fiber2_associated_curvature}.

\begin{prop}
    \label{prop:fiber2_associated_lens}
    Let
    $L(a;1) \times_{S^1, d} L(b;1)$
    denote the associated bundle geometry as defined in
    Defn.~\ref{defn:fiber2_associated_bundles},
    where the fiber $S^1$ of $L(a;1) \to S^2$
    by translating along fibers of $L(b;1) \to S^2$, with kernel $\Z/d\Z$.
    Then the homogeneous space
    \begin{align*}
        S^3 \times S^3 \times \R / \tau_{ad,b}(\R^2)
            &\cong S^3 \times S^3 \times \R /
                \{e^{\pi is}, e^{\pi it}, ads + bt\}_{s, t \in \R} .
    \end{align*}
    is a $\op{gcd}(ad, b)$-fold cover of $L(a;1) \times_{S^1,d} L(b;1)$.
\end{prop}
The notation interprets $S^3$ as the unit quaternions and any complex
numbers as lying inside the quaternions. It follows from the above
(Prop.~\ref{prop:fiber2_associated_lens}) that
\begin{enumerate}
    \item $L(a;1) \times_{S^1,d} L(b;1) \cong L(ad;1) \times_{S^1} L(b;1)$;
    \item $L(a;1) \times_{S^1} L(b;1) \cong L(b;1) \times_{S^1} L(a;1)$; and
    \item $L(a;1) \times_{S^1} L(b;1)$ is simply-connected if and only if
        $\op{gcd}(a,b) = 1$. (Use the homotopy exact sequence for
        $\tau_{a,b}(\R^2) \to S^3 \times S^3 \times \R
        \to L(a;1) \times_{S^1} L(b;1)$.)
\end{enumerate}
Just one key Lemma is required for its proof.
\begin{lemma}
    \label{lemma:fiber2_semidirect_straightening}
    Let $\gamma: H \to G$ be a homomorphism,
    and let $C: G \to \op{Inn} G$ denote conjugation. Then
        \[ G \times H / \{ (\gamma(h), h) \}_{h \in H} \cong G \rtimes_{C \circ \gamma} H / H . \]
\end{lemma}
\begin{proof}
    Verify from the definition of a semidirect product that
    \begin{align*}
        \phi: G \times H &\to G \rtimes_{C \circ \gamma} H \\
            g, h &\mapsto g \gamma(h)^{-1}, h .
    \end{align*}
    is an isomorphism sending $\{(\gamma(h), h)\}_{h \in H}$
    onto $H \subseteq G \rtimes_{C \circ \gamma} H$.
\end{proof}

\begin{proof}[Proof of Prop.~\ref{prop:fiber2_associated_lens}]
    Interpret $S^3$ as the unit quaternions,
    and all complex numbers as lying in the same copy of $\C$ in the quaternions.
    Since semidirect products with inner action are isomorphic to direct products
    (Lemma \ref{lemma:fiber2_semidirect_straightening}),
    \begin{align*}
        L(a;1)
            &= S^3 \rtimes S^1 / \{e^{(2\pi i/a)\Z}, t \op{mod} 2\}_{t \in \R} \\
            &\cong S^3 \times \R / \{e^{\pi i t + (2\pi i/a)\Z}, t\}_{t \in \R} \\
            &\cong S^3 \times \R / \{e^{\pi it}, at + 2\Z\}_{t \in \R} .
    \end{align*}
    So with the formula for homogeneous associated bundles
    (Example \ref{eg:fiber2_associated_homogeneous}),
    \begin{align*}
        L(a;1) \times_{S^1, d} L(b;1)
            &= S^3 \times \R \times S^3 \times \R /
                \{e^{\pi is}, as + 2\Z + u, e^{\pi it}, bt + 2\Z - du\}_{s,t,u \in \R} .
    \end{align*}
    Since $u$ ranges over all of $\R$, the entire first $\R$ factor in
    $S^3 \times \R \times S^3 \times \R$ is part of the point stabilizer.
    Since this $\R$ factor is normal,
    the homogeneous space is equivariantly diffeomorphic to a
    coset space of $S^3 \times S^3 \times \R$.
    To see exactly which coset space, let $u'$ be the coordinate in this $\R$ factor.
    Then
    \begin{align*}
        L(a;1) \times_{S^1, d} L(b;1)
            &\cong S^3 \times \R \times S^3 \times \R /
                \{e^{\pi is}, u', e^{\pi it}, bt + 2\Z -
                d(u' - as - 2\Z)\}_{s,t,u' \in \R} \\
            &\cong S^3 \times S^3 \times \R /
                \{e^{\pi is}, e^{\pi it}, bt + 2\Z + d(as + 2\Z)\}_{s,t \in \R} \\
            &\cong S^3 \times S^3 \times \R /
                \{e^{\pi is}, e^{\pi it}, ads + bt + 2\Z\}_{s,t \in \R} .
    \end{align*}
    The resemblance to $S^3 \times S^3 \times \R / \tau_{ad,b}(\R^2)$
    is now apparent. To finish the proof, observe that
    since $e^{2\pi i} = 1$,
    \begin{align*}
        S^3 \times S^3 \times \R / \tau_{ad,b}(\R^2)
            &\cong S^3 \times S^3 \times \R /
                \{e^{\pi is}, e^{\pi it}, ads + bt\}_{s,t \in \R} \\
            &\cong S^3 \times S^3 \times \R /
                \{e^{\pi is}, e^{\pi it}, ad(s + 2\Z) + b(t + 2\Z)\}_{s,t \in \R} \\
            &\cong S^3 \times S^3 \times \R /
                \{e^{\pi is}, e^{\pi it}, ads + bt + 2\op{gcd}(ad,b)\Z\}_{s,t \in \R} .
                \qedhere
    \end{align*}
\end{proof}

The rest of this sub-section is mostly devoted to proving the correctness of
Table \ref{table:fiber2_associated_bundles} other than in the third column.
Before doing so, having established some understanding above of the
$L(a;1) \times_{S^1} L(b;1)$ family of geometries,
we observe one restriction
on the parameters $a$ and $b$ that occurs only for this family.

\begin{prop}
    \label{prop:fiber2_rational_curvature}
    $\tau_{a,b}(\R^2)$ is closed in $\tilde{G} = S^3 \times S^3 \times \R$
    if and only if $a$ and $b$ are linearly dependent over $\Q$.
\end{prop}
\begin{proof}
    Let $H$ be the preimage in $\tilde{G}$ of the maximal torus in
    $\tilde{G}/Z(\tilde{G})$ over which $\tau_{a,b}(\R^2)$ lies.
    As the continuous preimage of a closed set, $H$ is closed;
    so it suffices to consider closedness in $H$.

    If $\tilde{G} = S^3 \times S^3 \times \R$,
    then $H \cong S^1 \times S^1 \times \R \cong \R^3 / (\Z^2 \times \{0\})$;
    and $\tau_{a,b}(\R^2)$ is the image in $H$ of a vector subspace
    $V \subseteq \R^3$ given in coordinates $(x,y,z)$ by $ax + by - z = 0$.
    Since $V$ and $\Z^2$ are both closed in $\R^3$,
    all of the following inclusions are closed if any one of them is.
    \begin{align*}
        \tau_{a,b}(\R^2) = V/(V \cap \Z^2) &\subseteq \R^3 / \Z^2 \cong H \\
        V\Z^2 &\subseteq \R^3  \\
        \Z^2/(V \cap \Z^2) &\subseteq \R^3/V
    \end{align*}
    Consider the last of the above inclusions.
    Modulo $ax + by - z$, the generators $(1,0,0)$ and $(0,1,0)$ of $\Z^2$
    are equivalent to $(0,0,-a)$ and $(0,0,-b)$. These generate a closed
    subgroup of $\R^3/V \cong \{(0,0,z) \mid z \in \R\} \subset \R^3$
    if and only if $\op{gcd}(a,b)$ exists
    (i.e.\ the integer linear combinations of $a$ and $b$ have
    a smallest positive value rather than accumulating on zero).
\end{proof}

\begin{rmk}
    \label{rmk:fiber2_rational_curvature}
    The rational dependence constraint does not appear if the rank
    $k$ of $\pi_1(H) \cong \Z^k$ is less than $2$,
    since a rank-at-most-one subgroup of a vector space $\tilde{H}/V$
    is closed.

    In particular, for all $\tilde{G}$ in this part of the classification
    other than $S^3 \times S^3 \times \R$, the ratio $a/b$
    for $\tau_{a,b}$ is allowed to be any real number
    (really, any point of $\R P^1$).
    %
\end{rmk}

\begin{prop}[\textbf{Correctness of Table \ref{table:fiber2_associated_bundles}}]
    \label{prop:fiber2_associated_curvature}
    Let $\tilde{G}$ and $\tau_{0,0}$ be as in one of the rows of
    Table \ref{table:fiber2_associated_bundles}.
    Then
    \begin{enumerate}[(i)]
        \item $M = \tilde{G}/\tau_{a,b}(\R^2)$ (Defn.~\ref{defn:fiber2_slope})
            fibers over the product $X \times Y$ in the last column.
        \item For some invariant Riemannian metrics on $X$ and $Y$
            with curvature $0$ or $\pm 1$,
            the connection on $M \to X \times Y$
            given by $(TM^{\tau_{a,b}(\R^2)})^\perp$
            has the same curvature as the bundle in the last column,
            provided that both $a$ and $b$ are nonzero.
    \end{enumerate}
\end{prop}

\begin{proof}[Proof of Prop.~\ref{prop:fiber2_associated_curvature}(i)]
    The quotient map $\pi: \tilde{G} \to \op{Inn} G = \tilde{G}/Z(\tilde{G})$
    sends $\tau_{a,b}(\R^2)$ to a $2$-dimensional maximal torus $T$.
    Then $(\op{Inn} G)/T$ is a $4$-dimensional geometry with point stabilizer $T$,
    which is a product of $2$-dimensional geometries $X \times Y$ by
    \cite[Thm.~3.1.1(c)]{filipk}. Computing $\tilde{G}/Z(\tilde{G})$
    determines $X$ and $Y$.
\end{proof}

\begin{proof}[Proof of Prop.~\ref{prop:fiber2_associated_curvature}(ii)]
    Suppose $a, b \in Z(\tilde{G})^0 \cong \R$ are both nonzero.
    For the connection on $M = \tilde{G}/\tau_{a,b}(\R^2) \to X \times Y$ given by
    $(TM^{\tau_{a,b}(\R^2)})^\perp$, the curvature $\Omega$ is an invariant $2$-form
    on $X \times Y$ with values in the Lie algebra of the fiber.

    \paragraph{Step 1: Curvature is nonzero only along $X$ and $Y$.}
    If $V$ is the standard representation of $\op{SO}(2)$,
    then a tangent space to $X \times Y$ decomposes
    as the $\op{SO}(2) \times \op{SO}(2)$ representation
        \[ T_x X \oplus T_y Y \cong (V \otimes \R) \oplus (\R \otimes V), \]
    whose second exterior power is
    \begin{align*}
        \Lambda^2(V \otimes \R) \oplus (V \otimes V) \oplus \Lambda^2(\R \otimes V)
            &\cong 2\R \oplus (V \otimes V) .
    \end{align*}
    So the $2$-form $\Omega$, being $\op{SO}(2) \times \op{SO}(2)$-invariant,
    is determined by its values on the $2\R$ summand---that is, on
    $\Lambda^2(T_x X)$ and $\Lambda^2(T_y Y)$.

    \paragraph{Step 2: Express the preimage of $X$ as a homogeneous space.}
    Fix a copy of $X$, and let $E$ be its preimage in $M$.
    Then choosing a factor of $\tilde{G}$ acting transitively on $X$
    and containing $Z(\tilde{G})^0$
    expresses $E$ as covered by one of the following homogeneous spaces $H/H_p$.
    \begin{align*}
        \Heis_3 \rtimes \widetilde{\op{SO}(2)} /
            \{(0,0,at), \gamma(t)\}_{t \in \R}
            &\cong \Heis_3 \\
        S^3 \times \R / \{e^{\pi it}, at\}_{t \in \R}
            &\cong S^3 \text{ (or $S^2 \times \R$ if $a = 0$)} \\
        \SLcover \times \R / \{\gamma(t), at\}_{t \in \R}
            &\cong \SLcover \text{ (or $\Hyp^2 \times \R$ if $a = 0$)}
    \end{align*}
    The notation conventions are that:
    \begin{enumerate}
        \item $\Heis_3$ is $\R^3$ with the composition law
            \( (x,y,z)(x',y',z') = (x+x', y+y', z+z'+xy'); \)
        \item $\gamma: \R \overset{\sim}{\to} \widetilde{\op{SO}(2)}$
            sends $t \in \R$ to a rotation by $2\pi t$; and
        \item $S^3$ and $\C$ are interpreted as subsets of the quaternions.
    \end{enumerate}

    \paragraph{Step 3: If $X$ has nonzero curvature $K$,
        then the curvature along $X$ is $\frac{K}{2\pi}\mathrm{vol}_X \otimes a$.}
    The curvature of $(TM^{\tau_{a,b}(\R^2)})^\perp$ restricted to $X$
    is the curvature of $TE^H$ as a connection on $E \to X$.

    Equivariantly, $S^3$ covers $\op{SO}(3) \cong T^1 S^2$
    and $\SLcover$ covers $\op{PSL}(2,\R) \cong T^1 \Hyp^2$;
    which takes $TE^H$ to the connection induced by the Levi-Civita
    connection on $S^2$ or $\Hyp^2$.
    If the unit tangent circles are declared to have length $2\pi$, then
    this connection's curvature is the surface's
    scalar curvature $K$ times its area form $\mathrm{vol}_X$.
    (This is a version of the Gauss-Bonnet Theorem;
    see e.g.\ \cite[\S{4.5} p.~274]{docarmosurfs}.)
    In $E = H/H_p$, such a circle is most naturally assigned the
    length of the interval
    $[0,a) \subset \R \cong Z(\tilde{G})^0 \subset H$
    that maps bijectively onto it.
    Then identifying $T_0 \R$ with $\R \cong Z(\tilde{G})^0$
    by the exponential map gives the expression
    $\frac{K}{2\pi} \mathrm{vol}_X \otimes a$
    for the curvature over $X$.

    The same argument applies to $Y$,
    which establishes the claim (ii) for the first three rows of
    Table \ref{table:fiber2_associated_bundles}.

    \paragraph{Step 4: Curvature along $X = \Euc^2$ is any nonzero
        multiple of the volume form.}
    First, observe that
    \begin{align*}
        \Heis_3 \rtimes \widetilde{\op{SO}(2)} /
            \{ (0,0,at), \gamma(t) \}_{t \in \R}
            &\cong \Heis_3 \rtimes \op{SO}(2) / \op{SO}(2) \\
        x,y,z,\gamma(t) &\mapsto x,y,z-at,\gamma(t).
    \end{align*}
    Under the usual metrics, $\Heis_3 \to \Euc^2$
    with the connection $T\Heis_3^{\op{SO}(2)}$ has curvature $1$
    (i.e.\ $1$ times the area form on $\Euc^2$)
    \cite[discussion after Exercise 3.7.1]{thurstonbook}.
    Through an appropriate (possibly orientation-reversing)
    conformal automorphism of $\Euc^2$, the area form
    can be pulled back to any nonzero constant multiple of itself.

    To finish, apply Step 3 to $Y$,
    choosing a length scale on $Z(\tilde{G})^0 = Z(\Heis_3)$
    that makes the curvature along $Y$ equal to $\frac{1}{2\pi} \mathrm{vol}_Y$.
    Then choose an orientation and scale on $X = \Euc^2$ that
    makes the curvature along $X$ equal to $\frac{1}{2\pi} \mathrm{vol}_X$.
    This establishes claim (ii) for the remaining (last two) rows of
    Table \ref{table:fiber2_associated_bundles}.
\end{proof}

\subsubsection{Proof of the classification}
\label{sec:fiber2_product_proof}

This section proves Prop.~\ref{prop:fiber2_bundle_list},
the classification of geometries $G/G_p$
when $\tilde{G}$ is a direct product listed in Table \ref{table:fiber2_product_groups}.
The recurring method is to relate $G_p$ to a maximal torus of some group;
each individual step is merely whatever happens to decrease
the number of remaining cases.
Maximality is deferred to Section \ref{sec:fiber2_maximality}.

As part of the classification, we also prove that
the associated bundle geometry $\SLcover \times_\alpha S^3$
is a model geometry (Prop.~\ref{prop:fiber2_associated_lattice})
but has compact quotients if and only if $\alpha$ is rational
(Prop.~\ref{prop:fiber2_compact_rationals}).

\begin{proof}[Proof of Prop.~\ref{prop:fiber2_bundle_list}
    (except for maximality)]
    Let $G/G_p$ be a $5$-dimensional maximal model geometry where
    $G$ is covered by one of the groups in Table \ref{table:fiber2_product_groups}.

    \paragraph{Step 1: If $\op{dim} G = 6$ and $\tilde{G} \ncong * \times \Sol^3$,
        then $G/H$ is non-maximal.}
    We will show that if $\op{dim} G = 6$ and the maximal torus of
    $\op{Aut} \tilde{G}$ has dimension at least $2$,
    then any $5$-dimensional geometry with isometry group $G$ is non-maximal.

    Since $\op{dim} G = 6$, the point stabilizer $G_p$ is $1$-dimensional
    and therefore isomorphic to $S^1$. Let $H$ be the lift of $G_p$ to $\tilde{G}$,
    i.e.\ the group such that $\tilde{G}/H \cong G/G_p$.
    Since $\op{Aut} \tilde{G}$ has a maximal torus of dimension at least $2$,
    there is an $S^1 \subset \op{Aut} \tilde{G}$ that commutes
    with the inner action of $H$ and is independent---that is,
    $H \times S^1$ maps to a $2$-dimensional subgroup of $\op{Aut} \tilde{G}$.
    Then
        \[ (\tilde{G} \rtimes S^1) / (H \times S^1) \]
	is a homogeneous space subsuming $G/G_p$; and a geometry
    subsuming $G/G_p$ is realized by passing to the quotient by
    $Z(\tilde{G} \rtimes S^1) \cap (H \times S^1)$ (Rmk.~\ref{rmk:quotient_by_center}).

	In particular, if $\tilde{G}$ is a product of any two of the following groups---each
    with dimension $3$ and an $S^1$ in its automorphism group---then
	any $5$-dimensional geometry $G/G_p$ is non-maximal.
	\begin{align*}
		&{}\widetilde{\isomplus \Euc^2}  &
		&{}S^3  &
		&{}\widetilde{\op{PSL}(2,\R)}  &
		&{}\Heis_3  &
		&{}\R^3
	\end{align*}

    \paragraph{Step 2: $G_p$ maps injectively to
        $\op{Inn} G = \tilde{G}/Z(\tilde{G})$
        as a maximal torus.}
    In the remaining cases, $\tilde{G}$ is one of the following products,
	where each $\bullet$ denotes the identity component of the isometry group of
	$\Euc^2$, $S^2$, or $\Hyp^2$.
	\begin{align*}
		&\tilde{\bullet} \times \tilde{\bullet} \times \R &
		&\tilde{\bullet} \times \widetilde{\isomplus \Heis_3} &
		&\tilde{\bullet} \times \op{Sol}^3
	\end{align*}
    The dimensions of these $\tilde{G}$ are respectively $7$, $7$, and $6$.
	The corresponding $\op{Inn} G$ are
	\begin{align*}
		&\bullet \times \bullet &
		&\bullet \times \isomplus \Euc^2 &
		&\bullet \times \op{Sol}^3 .
	\end{align*}
    The maximal torus of $\bullet$
    has dimension $1$, so the maximal torus $T$ of $\op{Inn} G$ has dimension
    $2$, $2$, or $1$, respectively.
    Since $G_p$ has faithful conjugation action in a geometry
    (Lemma \ref{lemma:faithful_conjugation}),
    which is equivalent to injectivity of $G_p \to \op{Inn} G$,
    the claim in this step follows from
    $\op{dim} T = \op{dim} G - 5$.

    \paragraph{Step 3: If $\tilde{G} = * \times \Sol^3$
        then $G/G_p$ is a product with $\Sol^3$.}
    In this case, $\tilde{G} \to \op{Inn} G$ is a covering map;
    so if $H$ is the identity component of the preimage of $G_p$ in $\tilde{G}$,
    then $H$ covers a maximal torus of $\op{Inn} G$.
	Then all possible $H$ are conjugate in $\tilde{G}$, so
        \[ G/G_p \cong \Sol^3 \times \isomplus B / \op{SO}(2) \cong \Sol^3 \times B \]
    where $B$ is $\Euc^2$, $S^2$, or $\Hyp^2$.

    \paragraph{Step 4: $G/G_p \cong \tilde{G}/\tau_{a,b}(\R^2)$
        for some $a$ and $b$ in $Z(\tilde{G})^0$.}
    In the remaining cases, $\dim Z(\tilde{G}) = 1$ and
    $\op{Inn} G = \tilde{G}/Z(\tilde{G})$ has $2$-dimensional maximal torus $T$.
    Since $\pi: \tilde{G} \to \op{Inn} G$ descends to a map from $G$
    sending $G_p$ isomorphically to $T$ (Step 2),
    the preimage of $G_p$ in $\tilde{G}$ is the image of a homomorphism
    $\tau: \R^2 \to \pi^{-1}(T)$.
    Any two such $\tau$ that compose with $\pi$ to the same map $\R^2 \to T$
    differ only by some $\R^2 \to Z(\tilde{G})$;
    so $\tau$ can be identified with some $\tau_{a,b}$
    as defined in Defn.~\ref{defn:fiber2_slope}
    after choosing $\tau_{0,0}$.

    In $\widetilde{\isomplus \Euc^2} \times S^3 \times \R$,
    choose $\tau_{0,0}(\R^2) = \widetilde{\op{SO}(2)} \times S^1 \times \{0\}$.
    In the three other $\tilde{G}$ that have $\widetilde{\isomplus \Euc^2}$ as a factor,
    let $\tau_{0,0}(\R^2) = \widetilde{\op{SO}(2)} \times \widetilde{\op{SO}(2)}$.
    All remaining $\tilde{G}$
    occur in Table \ref{table:fiber2_associated_bundles},
    where the choice of $\tau_{0,0}$ is also recorded.

    \paragraph{Step 5: If $\tilde{G} = \widetilde{\isomplus \Euc^2} \times K$
        for some group $K$,
        then $G/G_p$ is a product with $\Euc^2$.}
    Express $\widetilde{\isomplus \Euc^2}$ as $\C \rtimes \widetilde{\op{SO}(2)}$,
    i.e. as $\C \times \R$ with the composition law
        \[ (x + iy, z)(x' + iy', z') = (x+iy + e^{iz}(x'+iy'), z + z') . \]
    Then in $\tilde{G}/\tau_{a,b}(\R^2) \to \Euc^2 \times Y$, the preimage $E$ of
    $\Euc^2$ has a transitive action by a subgroup of $\tilde{G}$
    isomorphic to $\widetilde{\isomplus \Euc^2} \times \R$;
    and there is an equivariant diffeomorphism
    \begin{align*}
        E = \widetilde{\isomplus \Euc^2} \times \R / \{(0 + 0i, z), az\}_{z \in \R}
            &\overset{\sim}{\to} \isomplus \Euc^2 \times \R / \op{SO}(2) \\
        x + iy, z, t &\mapsto x + iy, z, t - az .
    \end{align*}
    Extending this to $\tilde{G}/\tau_{a,b}(\R^2)$ yields\footnote{
        That the diffeomorphism extends may be easier to see on the Lie algebra level.
        It corresponds to an automorphism $\tanisom \Euc^2 \times \R$
        that sends a basis $(\hat{x}, \hat{y}, \hat{z}, \hat{t})$
        to $(\hat{x}, \hat{y}, \hat{z} - a\hat{t}, \hat{t})$.
        This extends to the appropriate automorphism of
        $\tanalg G = \tanisom \Euc^2 \oplus \lie{k}$
        by the identity on $\lie{k}$.
    }
        \[ \tilde{G}/\tau_{a,b}(\R^2) \cong \widetilde{\isomplus \Euc^2} \times K
            / \tau_{0,b}(\R^2) \cong \Euc^2 \times (K/K_q) . \]

    \paragraph{Step 6:
        If $\tilde{G} = (\widetilde{\op{Isom} \Heis_3})^0 \times K$,
        then $G/G_p$ is one of four geometries.}
    There are only two cases remaining for this step to handle:
    $K = S^3$ and $K = \SLcover$.
    Step 4 of Prop.~\ref{prop:fiber2_associated_curvature}
    showed that the equivariant diffeomorphism type of $\tilde{G}/\tau_{a,b}(\R^2)$
    is independent of $a$, so set $a = 0$. Then if
    $\Heis_3$ is $\R^3$ with the composition law
        \[ (x,y,z)(x',y',z') = (x+x', y+y', z+z'+xy') \] and
    $\gamma: \R \to K$
    is a $1$-parameter subgroup of $K$ with $\gamma(\Z) = \gamma(\R) \cap Z(K)$,
    \begin{align*}
        \tilde{G}/\tau_{0,b}(\R^2)
            &= \Heis_3 \rtimes \op{SO}(2) \times K /
                \{(0,0,bt), s, \gamma(t)\}_{s,t \in \R}
    \end{align*}
    is one of
    \begin{enumerate}
        \item a product $\Heis_3 \times S^2$ or $\Heis_3 \times \Hyp^2$
            if $b = 0$; or
        \item equivariantly diffeomorphic to $\tilde{G}/\tau_{0,1}(\R^2)$
            by a map sending $(x,y,z,r) \in \Heis_3 \rtimes \op{SO}(2)$ to
                \[ \left(x|b|^{-1/2}, sy|b|^{-1/2}, z|b|^{-1}, r^s\right) \]
            where $s = b/|b|$, if $b \neq 0$. So the names
            $\Heis_3 \times_\R S^3$ and $\Heis_3 \times_\R \SLcover$
            from Table \ref{table:fiber2_associated_bundles}
            can be used without ambiguity.
    \end{enumerate}

    \paragraph{Step 7:
        Parametrize the isomorphism types of $\tilde{G}/\tau_{a,b}(\R^2)$
        by $\big[|a|:|b|\big] \in \R P^1$.}
    In the remaining cases, $M = \tilde{G}/\tau_{a,b}(\R^2)$
    has a canonical fibering over
        \[ X \times Y \cong \tilde{G}/(Z(\tilde{G}) \cdot \tau_{a,b}(\R^2)), \]
    where $X$ and $Y$ are surfaces of nonzero constant curvature.
    Decomposing $G_p \curvearrowright T_p M$ into irreducible
    subrepresentations, every invariant inner product on $T_p M$
    (and hence every invariant metric on $M$)
    is determined by a scale factor along the fiber (i.e.\ on $T_p M^{G_p}$),
    a scale factor on $X$, and a scale factor on $Y$
    (Lemmas \ref{lemma:perp_orthogonality} and \ref{lemma:inner_product_unique}).
    The scale factors on $X$ and $Y$ can be chosen by normalizing
    their curvatures to be $\pm 1$.

    Given any member of this family of normalized metrics,
    the ratio of curvatures for the connection
    $(TM^G)^\perp$ on $M \to X \times Y$
    (listed in Table \ref{table:fiber2_associated_bundles})
    represents, up to finite choices,
    an invariant for the family.
    Explicitly, an invariant number can be recovered as
    the ratio of displacements along a fiber
    that result from horizontal lifts of loops
    enclosing the same small area in $X$ and $Y$.
    The choices that need to be made are
    \begin{enumerate}
        \item assigning $X$ and $Y$ to the two factors in the base space, and
        \item orientations on $X$ and $Y$.
    \end{enumerate}
    These reflect the following symmetries.
    \begin{enumerate}
        \item[0.] Rescaling $\R \subset \tilde{G}$
            induces $\tilde{G}/\tau_{a,b}(\R^2) \cong \tilde{G}/\tau_{at,bt}(\R^2)$
            for nonzero $t$.
        \item If $X \cong Y$, then exchanging $X$ and $Y$ allows assuming $|a| \leq |b|$.
        \item Conjugating by $j$ in $S^3$ reverses the $1$-parameter
            subgroup $e^{it}$; and conjugating $\SLcover$ by
            $\begin{pmatrix} 0 & 1 \\ 1 & 0 \end{pmatrix}$
            reverses the $1$-parameter subgroup $\widetilde{\op{SO}(2)}$.
            This allows assuming $a > 0$ and $b > 0$.
    \end{enumerate}
    There are two more considerations that affect the parametrization:
    \begin{enumerate}
        \item[3.] If $a = 0$ or $b = 0$, then products result as in Step 6.
        \item[4.] If $\tilde{G} = S^3 \times S^3 \times \R$,
            then $a$ and $b$ need to be rationally dependent
            (Prop.~\ref{prop:fiber2_rational_curvature}),
            which allows rescaling them both to be coprime integers.
            This constraint is not present for other $\tilde{G}$
            (Rmk.~\ref{rmk:fiber2_rational_curvature}).
    \end{enumerate}
    So excluding products, the geometries remaining to be classified---those
    fibering over products of $S^2$ and $\Hyp^2$---are specified exactly once
    each by the following.
    \begin{align*}
        \SLcover \times_{a/b} \SLcover
            &= \SLcover \times \SLcover \times \R / \tau_{a,b}(\R^2),
            \quad  a/b \in (0,1] \\
        \SLcover \times_{a/b} S^3
            &= \SLcover \times S^3 \times \R / \tau_{a,b}(\R^2) ,
            \quad a/b \in (0,\infty) \\
        L(a;1) \times_{S^1} L(b;1)
            &= S^3 \times S^3 \times \R / \tau_{a,b}(\R^2) ,
            \quad 0 < a \leq b \text{ coprime in } \Z .
    \end{align*}

    \paragraph{Step 8:
        All of the above are model geometries.}
    Products of model geometries are model geometries,
    since they model products of manifolds modeled on the factors.
    The $L(a;1) \times_{S^1} L(b;1)$ geometry is a model geometry
    since it is already compact. So it only remains to show
    that bundles associated to $\Heis_3$ and $\SLcover$
    are also model geometries.

    The construction is: if $E$ and $F$ model compact circle
    bundles $M$ and $N$, then $M \times_{S^1} N$ is modeled on some $E \times_\rho F$.
    In particular, $\Heis_3$ models the circle bundle $\Heis_3/\Heis_3(\Z)$ over a torus
    modeled on $\Euc^2 \cong \Heis_3/Z(\Heis_3)$;
    $S^3$ models the Hopf fibration over $S^2$;
    and $\SLcover$ models the unit tangent bundle of any compact hyperbolic surface.

    However, only some $E \times_\rho F$
    can be recovered as the universal cover of an $M \times_{S^1} N$:
    combinations of these three bundles are modeled only on $\tilde{G}/\tau_{1,1}(\R^2)$.
    This is enough for the bundles associated to $\Heis_3$.
    For bundles associated to $\SLcover$,
    compact bundles modeled on $\tilde{G}/\tau_{a,b}(\R^2)$ can be obtained, for integers
    $a$ and $b$, by starting with quotients of the initial circle bundles
    by rotations of $2\pi/a$ and $2\pi/b$.
    This, however, still leaves the case when $a/b$ is irrational, which
    is a somewhat different construction, given below in
    Prop.~\ref{prop:fiber2_associated_lattice}.
\end{proof}

\begin{prop} \label{prop:fiber2_associated_lattice}
    If $\tilde{G} = \SLcover \times * \times \R$,
    then $\tilde{G}/\tau_{a,b}(\R^2)$ is a model geometry.
\end{prop}
\begin{proof}
    More honestly, since a geometry $M = G/G_p$ must have $G$ acting faithfully,
    we need to use
    $G = \tilde{G}/(Z(\tilde{G}) \cap \tau_{a,b}(\R^2))
    \cong \tilde{G}/\tau_{a,b}(\Z^2)$
    (Rmk.~\ref{rmk:quotient_by_center}).

    \paragraph{Step 1: Construct a subgroup $\tilde{\Gamma} \times \tilde{\Lambda}$
        of $\tilde{G}$.}
    Assume $b \neq 0$ since otherwise $G/G_p$ is a product with $S^2$ or $\Hyp^2$.
	Let $\Gamma = \pi_1(S) \subset \op{PSL}(2,\R)$ for some orienable\footnote{
        Requiring orientability permits assuming that
        $\pi_1(S)$ embeds in $\op{PSL}(2,\R)$,
        the identity component of $\op{Isom} \Hyp^2$.
    }
    punctured surface $S$ of genus at least $2$, and let
	$\tilde{\Gamma}$ be its preimage in $\SLcover$.
	Then $\tilde{\Gamma}$ is central extension of a free group
	by $Z(\SLcover) \cong \Z$;
	so it splits as a semidirect product, and $Z(\SLcover)$ maps
	to a copy of $\Z$ in the abelianization of $\tilde{\Gamma}$.
	Then
	there is a homomorphism $f: \tilde{\Gamma} \to \R$
    that sends $Z(\SLcover) \subset \widetilde{\op{SO}(2)}$ to $\Z$.
	Let $\Lambda$ be the fundamental group of $S^2$ or a compact orientable
    surface modeled on $\Hyp^2$, let $\tilde{\Lambda}$ be its lift to
    the group $*$, and define
    \begin{align*}
        i: \tilde{\Gamma} \times \tilde{\Lambda} &\to \SLcover \times * \times \R \\
	    g,h &\mapsto (g, h, af(g)) .
    \end{align*}
    
    \paragraph{Step 2: $\tilde{\Gamma} \times \tilde{\Lambda}$
        descends to a discrete subgroup $\Delta$ of $G$.}
    Let $\gamma$ be a $1$-parameter subgroup of $*$ sending $\Z$ to the center.
    We first show that the image of $i$ does not accumulate on
    $\tau_{a,b}(\{0\} \times \Z) = \{1, \gamma(n), nb\}_{n \in \Z}$.
    By discreteness of $\tilde{\Gamma} \times \tilde{\Lambda}$,
    some neighborhood $U \times V$ of the identity in
    $\SLcover \times *$ meets $\tilde{\Gamma} \times \tilde{\Lambda}$
    only in the identity. Then with the standing assumption that $b = 0$,
        \[ \bigcup_{n \in \Z} U \times \gamma(n) V \times ((n-1)b,(n+1)b) \]
    is an open set containing and preserved by
    $\tau_{a,b}(\{0\} \times \Z)$ in which the only
    elements $i(g,h) = (g,h,af(g))$ of $i(\tilde{\Gamma} \times \tilde{\Lambda})$
    satisfy all of
    \begin{align*}
        g &= 1  &  h &= \gamma(n)  &
        0 = \frac{a}{b} f(1) = \frac{a}{b} f(g) &\in (n-1,n+1) .
    \end{align*}
    That is, only the identity in $i(\tilde{\Gamma} \times \tilde{\Lambda})$
    lies in this neighborhood of $\tau_{a,b}(\{0\} \times \Z)$.

    This implies $i(\tilde{\Gamma} \times \tilde{\Lambda})$
    remains discrete in $\tilde{G}/\tau_{a,b}(\{0\} \times \Z)$.
    Since $i(\tilde{\Gamma} \times \tilde{\Lambda})$
    was constructed to contain $\tilde{G}/\tau_{a,b}(\Z \times \{0\})$,
    it remains discrete in $G = \tilde{G}/\tau_{a,b}(\Z^2)$.

    \paragraph{Step 3: $\Delta$ is a lattice in $G$.}
    With discreteness established, it suffices to show that
    for the image $\Delta$ of $\tilde{\Gamma} \times \tilde{\Lambda}$ in $G$,
    the volume of $G/\Delta \cong \tilde{G} /
    \left(i(\tilde{\Gamma} \times \tilde{\Lambda}) \cdot \tau_{a,b}(\Z^2)\right)$
    is finite.

    So let $H = i(\tilde{\Gamma}) \times \tilde{\Lambda}) \cdot \tau_{a,b}(\Z^2)$,
    and observe that
    \begin{align*}
        \tilde{G}/(H \cdot (* \times \R))
            &\cong \SLcover/\tilde{\Gamma} \cong S
            \quad\text{(chosen at the start of Step 1)} \\
        H \cdot (* \times \R) / (H \cdot \R)
            &\cong * / \tilde{\Lambda}
            \quad\text{(the compact surface chosen at the end of Step 1)} \\
        H \cdot \R / H
            &\cong \R / \{nb\}_{n \in \Z} \cong S^1 .
    \end{align*}
    In a situation involving only closed subgroups $E \subseteq F \subseteq G$
    of a locally compact $G$, an invariant measure on $G/E$ is constructed
    as a product of invariant measures on $G/F$ and $F/E$
    as in \cite[2.4 Case 2]{mostow_homogeneous_1962};
    so the volume of $G/H$ is finite since all three intermediate
    spaces are.

    \paragraph{Step 4: $\Delta \backslash G / G_p$ is a finite-volume manifold.}
    The space $\Delta \backslash G / G_p$ has finite volume and
    is modeled on $G/G_p$ but might be
    an orbifold---ruling out orbifold points requires checking that
    $\Delta$ acts freely, i.e.\ that $\Delta$ meets each point stabilizer
    in only the identity. Since $\Delta$ is discrete, its intersection
    with any compact point stabilizer has finite order. So it suffices to check
    that a subgroup of $\tilde{G} = \SLcover \times * \times \R$
    surjecting onto $\Delta$---specifically,
    $i(\tilde{\Gamma} \times \tilde{\Lambda})$---contains no element
    $(g,h,t)$ outside of $\tau_{a,b}(\Z^2)$ with a nonzero power in $\tau_{a,b}(\Z^2)$.

    Since $\tau_{a,b}(\Z^2)$ is central,
    any $i(g,h) = (g,h,af(g))$ with a nonzero power in $\tau_{a,b}(\Z^2)$
    has finite-order image in $\tilde{G}/Z(\tilde{G})$.
    Since $g$ is in the lift $\tilde{\Gamma}$ of $\pi_1(S)$ where $S$
    is a surface (in particular, with no orbifold points),
    $g$ has finite-order image in $\op{PSL}(2,\R)$ only if this image is the identity.
    Similarly, $h$ lies over the identity of $\op{SO}(3)$ or $\op{PSL}(2,\R)$;
    so some $\tau_{a,b}(m,n)$ has the same first two coordinates $g$ and $h$.

    If $* = S^3$, then $h = 1$ and $(g,1,af(g)) = \tau_{a,b}(m,0)$.
    If $* = \SLcover$, then
    $Z(\tilde{G})$ has no finite-order elements, which makes $n$th roots unique;
    so $(g,h,af(g))$ has a nonzero power in $\tau_{a,b}(\Z^2)$ if and only if it
    lies in $\tau_{a,b}(\Z^2)$ itself.
    Either way, $i(\tilde{\Gamma} \times \tilde{\Lambda})$ contains nothing
    outside of $\tau_{a,b}(\Z^2)$ with a nonzero power in $\tau_{a,b}(\Z^2)$;
    so by the first paragraph, $\Delta \backslash G / G_p$ has no orbifold points
    and is a finite-volume manifold modeled on $G/G_p$.
\end{proof}

The above construction always produces noncompact $\Delta \backslash G / G_p$,
with fundamental group independent of $a$ and $b$.
In compact manifolds, the story is different---for instance,
one can prove the following.

\begin{prop} \label{prop:fiber2_compact_rationals}
    Let $\tilde{G} \cong \SLcover \times S^3 \times \R$,
    If there is a compact manifold $N$ modeled on $\tilde{G}/\tau_{a,b}(\R^2)$
    and $b \neq 0$, then $a/b \in \Q$.
\end{prop}
\begin{proof}
    As above in Prop.~\ref{prop:fiber2_associated_lattice},
    $G = \tilde{G}/\tau_{a,b}(\Z^2)$,
    in which $\pi_1(N)$ is a cocompact lattice.
    Its preimage $\tilde{\Gamma}$ in $\tilde{G}$ is also a cocompact lattice.
    We will study the projection of $\tilde{\Gamma}$ to
    $\Gamma \subset \SLcover \times (\R/2b\Z)$
    and show in particular that $\tau_{a,b}(\Z \oplus \{0\})$
    projects to a finite subgroup of $\R/2b\Z$.

    \paragraph{Step 1: $\Gamma$ is discrete in $\SLcover \times (\R/2b\Z)$.}
    Since $\tilde{\Gamma}$ is discrete and $S^3$ is compact,
    $\tilde{\Gamma}$ cannot accumulate on a coset of $S^3$.
    Then $\tilde{\Gamma}S^3$ is closed, so $\tilde{\Gamma}/(\tilde{\Gamma} \cap S^3)$
    is discrete in $\tilde{G}/S^3 \cong \SLcover \times \R$.
    Moreover,
        \[ \tilde{\Gamma} \supset \tau_{a,b}(\Z^2) \supset \tau_{a,b}(\{0\} \oplus \Z)
            = \{1, e^{\pi in}, bn\}_{n \in \Z} \subset \SLcover \times S^3 \times \R , \]
    so $\tilde{\Gamma}$ contains $2b\Z \subset \R$.
    Then $\Gamma = \tilde{\Gamma}/(S^3 \times 2b\Z)$
    is discrete in $\SLcover \times (\R/2b\Z)$.

    \paragraph{Step 2: $\Gamma$ descends to a cocompact lattice
        $\Lambda$ in $\op{PSL}(2,\R)$.}
    This proceeds similarly to Step 1.
    Let $\tilde{\Lambda}$ be the projection of $\Gamma$ to $\SLcover$.
    Since $\R/2b\Z$ is compact, $\tilde{\Lambda}$ is discrete in $\SLcover$.
    Furthermore, if $k$ is a generator of $Z(\SLcover)$,
    then the image in $\SLcover \times (\R/2b\Z)$
    of $\tau_{a,b}(\Z^2) \subset \tilde{\Gamma}$ is
        \[ \{k^n, an \op{mod} 2b\}_{n \in \Z} \subset \Gamma. \]
    Then $\tilde{\Lambda}$ contains $Z(\SLcover)$,
    so its image $\Lambda$ in $\op{PSL}(2,\R)$ is discrete;
    and $\Lambda$ is cocompact since $\Lambda \backslash \op{PSL}(2,\R)$
    is a quotient space of $\tilde{\Gamma} \backslash \tilde{G}$.

    \paragraph{Step 3: $k$ becomes torsion in the abelianization of $\tilde{\Lambda}$.}
    Since $\Lambda \subset \op{PSL}(2,\R)$ is a cocompact lattice,
    it is the orbifold fundamental group of a compact orbifold $O$ modeled on $\Hyp^2$;
    and $\tilde{\Lambda} = \pi_1(T^1 O)$
    (see e.g.\ \cite[\S{13.4}]{thurstonnotes} for some discussion
    of unit tangent bundles of orbifolds).
    Since a hyperbolic $2$-orbifold admits a finite cover by a hyperbolic surface
    \cite[Thm.~2.3 and 2.5]{scott}, 
    $\tilde{\Lambda}$ contains a subgroup isomorphic to the unit tangent bundle
    of a closed hyperbolic surface $S$. Its center is generated by $k$,
    and $k^{\chi(S)}$ is a product of commutators
    \cite[discussion surrounding Lemma 3.5]{scott}.
    So $k$ becomes finite order
    in the abelianization of $\tilde{\Lambda}$.

    \paragraph{Step 4: $\tau_{a,b}(\Z^2)$ becomes torsion in
        the abelianization of $\Gamma$.}
    The intersection of the compact $\R/2b\Z$ with the discrete $\Gamma$
    is a finite cyclic group $C$, which makes $\Gamma$ a central extension
        \[ 1 \to C \to \Gamma \to \tilde{\Lambda} \to 1 . \]
	Using the Stallings exact sequence
    \cite[II.5 Exercise 6(a)]{kenbrown}, the induced
        \[ C \to \Gamma^{\mathrm{Ab}} \to \tilde{\Lambda}^{\mathrm{Ab}} \]
    is exact in the middle. Then $(k, a \op{mod} 2b) \in \Gamma$,
    which lies over $k \in \tilde{\Lambda}$,
    becomes finite order in the abelianization;
    so its image $a$ in the abelian $\R/2b\Z$ has finite order.
    Therefore $a$ is a rational multiple of $b$.
\end{proof}

\subsubsection{Explicit enumeration of product geometries}
\label{sec:fiber2_product_list}

This section collects a list of the product geometries
with nontrivial abelian isotropy---i.e.\ those with the $2$-dimensional base
in the fibering description (Prop.~\ref{prop:fibering_description}(iii)).

\begin{prop}
    \label{prop:fiber2_explicit_products}
    The maximal model product geometries with nontrivial abelian isotropy are:
    \begin{enumerate}[(i)]
        \item $4$-by-$1$:
            \begin{align*}
                \mathbb{F}^4 &\times \Euc  &
                \Sol^4_0 &\times \Euc
            \end{align*}
        \item $2$-by-$2$-by-$1$:
            \begin{align*}
                S^2 &\times S^2 \times \Euc &
                S^2 &\times \Hyp^2 \times \Euc &
                \Hyp^2 &\times \Hyp^2 \times \Euc
            \end{align*}
        \item $3$-by-$2$:
            \begin{align*}
                \Heis_3 &\times \Euc^2 &
                \Heis_3 &\times S^2 &
                \Heis_3 &\times \Hyp^2 \\
                \Sol^3 &\times \Euc^2 &
                \Sol^3 &\times S^2 &
                \Sol^3 &\times \Hyp^2 \\
                \SLcover &\times \Euc^2 &
                \SLcover &\times S^2 &
                \SLcover &\times \Hyp^2
            \end{align*}
    \end{enumerate}
\end{prop}
\begin{proof}
    The list can be built up by looking up the factors in previous classifications
    by Thurston \cite[Thm.~3.8.4]{thurstonbook} and Filipkiewicz \cite{filipk}.
    Products with multiple Euclidean factors are non-maximal and omitted.

    As products of model geometries, these are all model geometries
    (Prop.~\ref{prop:products_are_models}). All are products with at most one
    Euclidean factor; and except for $\Sol^4_0 \times \Euc$, all are products
    of two maximal geometries where
    one has no trivial subrepresentation in its isotropy,
    which makes them maximal (Prop.~\ref{prop:products_are_maximal}).
    Maximality of $\Sol^4_0 \times \Euc$ was proven in Step 6
    of Prop.~\ref{prop:fiber2_essentials}(iii).
\end{proof}

\subsection{Deferred maximality proofs for isometrically fibering geometries}
\label{sec:fiber2_maximality}

Having just handled the products in Section \ref{sec:fiber2_product_list} above,
maximality remains to be proven only for
the five associated bundle geometries or families
and the six geometries from semidirect products
in Prop.~\ref{semidirect_product_geometries}.

The strategy in many cases
(perhaps because the author failed to think of anything less ad-hoc)
is to show that $G'/H'$ cannot subsume $G/H$
by showing that some subgroup of $G$ does not appear in $G'$.
The most useful in what follows is a restriction on
embeddings of $\Heis_3$.
\begin{lemma} \label{lemma:nilembedding} ~
    Let $\lie{g}$ be a direct product of algebras
    of the form $\tanalg \op{Isom} M$ ($M = S^k$, $\Euc^k$, or $\Hyp^k$).
    Then $\lie{g}$ contains no nonabelian nilpotent subalgebra.
\end{lemma}
\begin{proof}
    If a subalgebra has an abelian projection to each factor in the product,
    then it is itself abelian;
    so it suffices to prove the Lemma when
    $\lie{g} = \tanalg \op{Isom} M$ for $M = S^k$, $\Euc^k$, and $\Hyp^k$.

    \paragraph{Case 1: $M = S^k$.}
            In this case one can in fact show that every solvable subalgebra is abelian.
            A solvable subalgebra of $\tanalg \isomplus S^k \cong \lie{so}_{k+1} \R$
            is tangent to a connected solvable group,
            whose closure in $\op{SO}(k+1)$ is solvable
            since solvability is a closed condition.
            A connected solvable Lie group is abelian if it is compact
            \cite[Cor.~IV.4.25]{knapp}.

    \paragraph{Case 2: $M = \Euc^k$.}
            Suppose $x$ and $y$ generate a nonabelian subalgebra $\lie{g}$ of
            $\tanalg \isomplus \Euc^k \cong \R^k \semisum \lie{so}_k$.
            If their images in $\lie{so}_k$ do not commute,
            then Case 1 implies $\lie{g}$ is not solvable.

            Otherwise, write $x = (t_x \in \R^k, r_x \in \lie{so}_k)$ and $y = (t_y, r_y)$.
            Then
            \begin{align*}
                [x,y] &= r_x t_y - r_y t_x \\
                [x,[x,y]] &= r_x^2 t_y \\
                [y,[y,x]] &= r_y^2 t_x .
            \end{align*}
            Since $[x,y] \neq 0$, at least one of $r_x t_y$ and $r_y t_x$
            is nonzero. Then since $r_x$ and $r_y$ act semisimply on $\R^k$,
            at least one of $[x,[x,y]]$ and $[y,[y,x]]$ is nonzero.
            Recursing, the lower central series of $\lie{g}$ is never zero,
            so $\lie{g}$ is not nilpotent.

    \paragraph{Case 3: $M = \Hyp^k$.}
            Suppose $\lie{g}$ is a nilpotent subalgebra
            of $\lie{so}_{1,k} \cong \tanalg \isomplus \Hyp^k$,
            with corresponding group $G$.
            Every connected solvable subgroup of $\isomplus \Hyp^k$
            fixes a point either in $\Hyp^k$ or in its boundary at infinity
            $\partial_\infty \Hyp^k \cong S^{k-1}$ \cite[Thm.~5.5.10]{ratcliffe2006}.
            Then:
            \begin{itemize}
                \item If this fixed point is in $\Hyp^k$,
                    then $G \subseteq \op{SO}(k)$,
                    which makes $G$ abelian by Case 1.
                \item Otherwise, the fixed point is in $\partial_\infty \Hyp^k$.
                    Since $\isomplus \Hyp^k \cong \op{Conf}^+ S^{k-1}$
                    \cite[Prop.~A.5.13(4)]{bp}
                    acts on $S^{k-1}$ with point stabilizer $\op{Conf}^+ \Euc^{k-1}$
                    \cite[Cor.~A.3.8]{bp},
                    this implies that
                    $G \subseteq \op{Conf}^+ \Euc^{k-1} \cong \R \times \isomplus \Euc^{k-1}$,
                    which makes $G$ is abelian by Case 2. \qedhere
            \end{itemize}
\end{proof}

\begin{prop} \label{prop:assoc_bundle_maximality}
    The five associated bundle geometries classified in
    Prop.~\ref{prop:fiber2_bundle_list}
    are maximal.
\end{prop}
\begin{proof}
    Let $M = G/G_p$ be one of the associated bundle geometries.

    \paragraph{Step 1: The isotropy of any subsuming geometry is $\op{SO}(5)$
        or $\op{SO}(3) \times \op{SO}(2)$.}
    In any geometry $G'/G'_p$ properly subsuming $M$, the isotropy $G'_p$
    must contain $\op{SO}(2)^2$---so consulting Figure \ref{fig:isotropy_poset},
    $G'_p$ is one of $\op{SO}(3) \times \op{SO}(2)$,
    $\op{U}(2)$, $\op{SO}(4)$, or $\op{SO}(5)$.

    In fact, $G'_p$ cannot be $\op{U}(2)$ or $\op{SO}(4)$.
    If this were the case, $G'$ would preserve $TM^G$,
    inducing a $G$-equivariant diffeomorphism
    $M/\foliation{F}^G \to M/\foliation{F}^{G'}$.
    But $M/\foliation{F}^G$ is a product of $2$-dimensional maximal model geometries
    with $G$ acting by the isometry group,
    whereas (consulting the classification in Prop.~\ref{prop:main:i})
    a base space of a geometry with $\op{U}(2)$ or $\op{SO}(4)$ isotropy can only be
    $S^4$, $\Euc^4$, $\Hyp^4$, $\C P^2$, $\C^2$, and $\C\Hyp^2$.

    \paragraph{Step 2: Geometries involving $\Heis_3$ are maximal.}
    In the isometry group of a constant-curvature geometry,
    every connected nilpotent subgroup is abelian (Lemma \ref{lemma:nilembedding});
    so this holds for products of such groups too.
    Since all geometries with isotropy $\op{SO}(5)$ or $\op{SO}(3) \times \op{SO}(2)$
    are constant-curvature geometries or products thereof,
    their isometry groups do not contain subgroups covered by $\Heis_3$.
    Therefore the geometries $\Heis_3 \times_\R S^3$ and $\Heis_3 \times_\R \SLcover$
    are maximal.

    \paragraph{Step 3: $\op{SO}(5)$-isotropy geometries do not subsume.}
    Since $\SLcover \times_\alpha S^3$ is an $S^3$ bundle over $\R^2$,
    its nonzero $\pi_3$ distinguishes it from the $\op{SO}(5)$-isotropy geometries.
    Similarly, $L(a;1) \times_{S^1} L(b;1)$ is an $S^3$ bundle over $S^2$,
    so it is distinguished from the $\op{SO}(5)$ geometries by having
    nontrivial $\pi_2$.

    $\SLcover \times_\alpha \SLcover$ cannot be subsmed by $S^5$ (which is compact)
    or $\Euc^5$ (whose isometry group has compact semisimple part).
    To it distinguish from $\Hyp^5$, observe that the image in
    $\SLcover \times_\alpha \SLcover$ of
    $\SLcover \times \{1\} \times \R \subseteq \SLcover \times \SLcover \times \R$
    is a copy of $\SLcover \rtimes \op{SO}(2) / \op{SO}(2)$,
    fixed by an $\op{SO}(2)$ with trivial projection to the first $\SLcover$ factor;
    whereas in $\Hyp^5$, every group of isometries conjugate to $\op{SO}(2)$
    fixes a copy of $\Hyp^3$.

    \paragraph{Step 4: $\op{SO}(3) \times \op{SO}(2)$ geometries
        do not subsume.}
    If $\R$ and $V$ denote the $1$-dimensional trivial and
    $2$-dimensional standard representations of $\op{SO}(2)$, then
    the tangent space $T_p M$ decomposes into three nonisomorphic
    representations of $G_p = \op{SO}(2) \times \op{SO}(2)$ as\footnote{
        Over $\C$, the irreducible representations of a direct product of groups
        are the tensor products of their irreducible representations;
        see e.g.\ \cite[Prop.~II.4.14]{brockerdieck}
    }
    \[ T_p M \cong (\R \otimes \R) \oplus (\R \otimes V) \oplus (V \otimes \R) , \]
    which determines two invariant $2$-dimensional distributions.
    From the description of $M$
    as a bundle over $2$-by-$2$ product geometries
    (Table \ref{table:fiber2_associated_bundles},
    Prop.~\ref{prop:fiber2_associated_curvature}),
    neither of these is integrable.
    Then $M$ cannot be subsumed by a $3$-by-$2$ product geometry,
    since every geometry with $\op{SO}(3) \times \op{SO}(2)$ isotropy
    has an invariant integrable $2$-dimensional distribution.
\end{proof}

\begin{rmk} \label{rmk:nonunique_maximality}
    Proposition \ref{prop:assoc_bundle_maximality}
    provides a negative answer to the question raised by Filipkiewicz
    in the discussion after \cite[Prop.~1.1.2]{filipk}:
    is every non-maximal geometry
    subsumed by a \emph{unique} maximal geometry?

    The counterexample is
    $T^1 S^3 \cong \op{SO}(4)/\op{SO}(2) \cong S^3 \times S^3 / \Delta(S^1)$
    where $\Delta: S^3 \to S^3 \times S^3$ is the diagonal map.
    It is subsumed:
    \begin{enumerate}
        \item by $S^3 \times S^2$
            since $S^3$ is parallelizable by left-invariant vector fields, and
        \item by $L(1;1) \times_{S^1} L(1;1) = S^3 \times S^3 \times \R / \tau_{1,1} (\R^2)$
            since $\tau_{1,1}(\R^2)$ meets $S^3 \times S^3$ in
            an antidiagonal (which is conjugate to diagonal) copy of $S^1$.
    \end{enumerate}
    Since $S^3 \times S^2$ is maximal as a product and
    $L(1;1) \times_{S^1} L(1;1)$ is maximal
    due to Prop.~\ref{prop:assoc_bundle_maximality},
    a subsuming geometry for $\op{SO}(4)/\op{SO}(2)$ is not unique.

    One may alternatively check that $S^3 \times S^2$ does not
    subsume $L(1;1) \times_{S^1} L(1;1)$
    by classifying all homomorphisms
        \[ S^3 \times S^3 \times \R \to (S^3)^3
            \cong \widetilde{\left(\op{Isom} (S^3 \times S^2)\right)^0} \]
    using the fact that $\op{Aut} S^3 = \op{Inn} S^3$.
\end{rmk}

\begin{prop}
    \label{prop:fiber2_semidirect_maximality}
    All geometries listed in Prop.~\ref{semidirect_product_geometries}
    are maximal.
\end{prop}
\begin{proof}
    Let $M = G/H$ be one of the geometries named,
    and suppose $G'/H'$ subsumes it.

    \paragraph{Step 1: List restrictions on subsuming geometries.}
    Since $G/H$ is contractible, so is $G'/H'$.
    Since $G$ contains a group covered by $\SLcover$, so does $G'$.
    Finally, $H'$ contains $H$ ($S^1_{1/2}$, or $S^1_1$ in the case of $T^1 \Euc^{1,2}$)
    and cannot preserve $TM^G$---since if it did, then it would preserve the fibering
    $M \to M/\foliation{F}^G$; but these geometries are the only ones encountered
    in the classification for which $M/\foliation{F}^G$ is $\mathbb{F}^4$ or $T\Hyp^2$.
    Consulting Figure \ref{fig:isotropy_poset}, this last restriction implies $H$ is
    $\op{SO}(5)$, $\op{SO}(3) \times \op{SO}(2)$, or $\op{SO}(3)_5$.

    Then by the classification so far, $G'/H'$ is either
    $\op{SL}(3,\R)/\op{SO}(3)$ or a product of a hyperbolic space
    with zero or more Euclidean or hyperbolic spaces.

    \paragraph{Step 2: Find $3$-dimensional fixed sets of order $2$ isometries.}
    Such a fixed set in $G/H$ must also appear in $G'/H'$, so we list these.
    See Table \ref{table:fiber2_fixed_sets_3d} for a summary.
    \begin{table}
        \caption{$3$-dimensional fixed sets of order $2$ isometries}
        \label{table:fiber2_fixed_sets_3d}
        \begin{center}\begin{tabular}{cc}
            Geometry & Fixed set \\
            \hline
            (product) & (product) \\
            $\op{SL}(3,\R)/\op{SO}(3)$ & $\Hyp^2 \times \Euc$ \\
            $\mathbb{F}^5_0 = \Heis_3 \rtimes \op{SL}(2,\R) / \op{SO}(2)$ & $\Hyp^2 \times \Euc$ \\
            $\mathbb{F}^5_1$ & $\SLcover$ \\
            $\R^2 \rtimes \SLcover$ & $\SLcover$ \\
        \end{tabular}\end{center}
    \end{table}
    \begin{itemize}
        \item Using the classification of isometries of $\Hyp^n$ \cite[A.5.14]{bp}
            and the fact that $\op{SO}(k)$ is maximal compact in $\isomplus \Euc^k$,
            these fixed sets in products of Euclidean and hyperbolic spaces
            are also products of Euclidean and hyperbolic spaces.
        \item In $\op{SL}(3,\R)/\op{SO}(3)$, an order $2$ isometry
            is a $3 \times 3$ matrix with a positive even number of $-1$ eigenvalues.
            Its centralizer is conjugate (by diagonalization) to the subgroup of
            $2+1$ block matrices; and the orbit of this subgroup
            is the isometry's fixed set,
                \[ \op{S}\left(\op{GL}(2,\R) \times \op{GL}(1,\R)\right)^0 / \op{SO}(2)
                    \cong \op{SL}(2,\R) \times \R / \op{SO}(2) \cong \Hyp^2 \times \Euc .
                \]
        \item In $\Heis_3 \rtimes \op{SL}(2,\R) / \op{SO}(2)$,
            the order $2$ isometries are all conjugate to the order $2$ element of
            $\op{SO}(2)$. Its centralizer is $Z(\Heis_3) \times \op{SL}(2,\R)$,
            with orbit $\Euc \times \Hyp^2$.
        \item Following the same recipe for $\R^2 \rtimes \SLcover$ and $\mathbb{F}^5_1$
            produces a centralizer of $\SLcover \rtimes \op{SO}(2)$,
            with orbit $\SLcover$.
    \end{itemize}
    This last result implies $\R^2 \rtimes \SLcover$ and $\mathbb{F}^5_1$
    are maximal, since none of the candidates for $G'/H'$ have $\SLcover$
    as a fixed set of an order $2$ isometry.

    \paragraph{Step 3: $\Heis_3 \rtimes \op{SL}(2,\R) / \op{SO}(2)$
        is maximal.}
    If $G$ contains $\Heis_3$, then so does $G'$. A copy of $\Heis_3$
    in a direct product must have nonabelian---hence locally injective---image
    in at least one factor. Since $\Heis_3$ covers no subgroup of isometries
    of $\Euc^k$ or $\Hyp^k$ (Lemma \ref{lemma:nilembedding}),
    only $\op{SL}(3,\R)/\op{SO}(3)$ can subsume.

    By the classification of irreducible representations of
    $\op{SL}(2,\R)$---one in each dimension \cite[11.8]{fultonharris}---all
    $\op{SL}(2,\R)$ in $\op{SL}(3,\R)$ are conjugate to a standard copy.
    Denoting the standard representation of $\op{SL}(2,\R)$ by $V$,
    counting weights (see e.g.\ \cite[\S{11.2}]{fultonharris})
    produces the decomposition
        \[
            \lie{sl}_3 \R \cong_{\op{SL}(2,\R)}
            \R \oplus 2V \oplus \lie{sl}_2 \R .
        \]
    The two copies of $V$ are $\op{Hom}(\R^2, \R)$ and $\op{Hom}(\R, \R^2)$,
    which are abelian subalgebras of $\lie{sl}_3 \R$ and therefore
    cannot generate a copy of $\Heis_3$.
    Therefore $\op{SL}(3,\R)/\op{SO}(3)$ does not subsume
    $\Heis_3 \rtimes \op{SL}(2,\R) / \op{SO}(2)$.
    
    \paragraph{Step 4: $T^1 \Euc^{1,2}$ is maximal.}
    Since the isotropy of $\op{SL}(3,\R)/\op{SO}(3)$
    does not contain $S^1_1$
    (Fig.~\ref{fig:isotropy_poset}; see also
    \cite[Prop.~\ref{ii:isotropy_classification} footnote]{geng2}),
    it cannot subsume $T^1 \Euc^{1,2}$.
    So it only remains to eliminate the products of hyperbolic
    and Euclidean spaces, as follows.

    In $\lie{so}_{1,2} \R$, there is a matrix
        \[ A = \begin{pmatrix}  & & 1 \\ & & -1 \\ 1 & 1 & 0 \end{pmatrix} , \]
    which sends the third standard basis vector $e_3$
    to $e_1 - e_2$ and sends $e_1 - e_2$ to zero.
    Then $A$, $e_3$, and $e_1 - e_2$ span a copy of the
    $3$-dimensional Heisenberg algebra in $\R^3 \semisum \lie{so}_{1,2} \R$.
    Since $\tanalg \isomplus \Hyp^k$ and $\tanalg \isomplus \Euc^k$
    cannot contain the Heisenberg algebra (Lemma \ref{lemma:nilembedding}),
    neither can a product of them, since the projection to at least one
    factor would have to be nonabelian and thereby injective.
    Then none of the products of hyperbolic and Euclidean spaces
    can subsume $T^1 \Euc^{1,2}$.
\end{proof}

\bibliographystyle{amsalpha}
\bibliography{main}

\end{document}